\documentclass[american]{amsart}

\usepackage[utf8]{inputenc}
\usepackage[T1]{fontenc}
\usepackage{a4wide}

\usepackage{graphicx}
\usepackage{pb-diagram}
\usepackage[]{hyperref}
\usepackage{caption}
\usepackage{subcaption}
\usepackage{color}
\usepackage{cleveref}
\usepackage{mathtools}
\usepackage{amssymb}

\usepackage{tikz-cd}

\usepackage{ifthen}
\usepackage{newunicodechar}

\newcommand{\defin}[1]{\textbf{#1}}

\DeclareMathOperator{\Ad}{Ad}

\DeclareMathOperator{\Crit}{Crit}

\DeclareMathOperator{\Fix}{Fix}

\DeclareMathOperator{\id}{id}

\DeclareMathOperator{\length}{length}

\newcommand{\restricted}[2]{{\left.{#1}\right|_{#2}}}
\newcommand{\lcan}{{\lambda_{\mathrm{can}}}}

\newcommand{\p}{\partial}
\newcommand{\lie}[1]{{\mathcal{L}_{#1}}}
\newcommand{\abs}[1]{{\mathopen\lvert #1\mathclose\rvert}}
\newcommand{\norm}[1]{{\mathopen\lVert #1\mathclose\rVert}}

\renewcommand{\epsilon}{\varepsilon}
\makeatletter
\usepackage{xspace}

\makeatother

\newcommand{\0}{{\mathbf 0}}

\newcommand{\CC}{{\mathbb C}}

\newcommand{\DD}{{\mathbb D}}

\newcommand{\gfrak}{{\mathfrak g}}

\newcommand{\NN}{{\mathbb N}}

\newcommand{\QQ}{{\mathbb Q}}
\newcommand{\RR}{{\mathbb R}}
\renewcommand{\SS}{{\mathbb S}}
\newcommand{\TT}{{\mathbb T}}
\newcommand{\tfrak}{{\mathfrak t}}

\newcommand{\x}{{\mathbf x}}
\newcommand{\y}{{\mathbf y}}

\theoremstyle{plain}

\newcounter{maintheorem}

\newtheorem{main_theorem}[maintheorem]{Theorem}
\newtheorem*{theorem_symplectic_action_always_hamiltonian}{\Cref{corollary:
symplectic group action}}
\newtheorem*{theorem_fixed_points_boundary_components}{\Cref{theorem:
fixed points and boundary components of hamiltonian S1-manifolds}}
\newtheorem*{theorem_boundary_connected_dim4}{\Cref{boundary connected in dimension 4}}
\newtheorem*{theorem_rank_two_connected_boundary}{\Cref{rank two connected boundary}}
\newtheorem*{theorem_connected_level_sets}{\Cref{theorem: connected level sets}}

\newtheorem*{convergence_trajectory}{\Cref{convergence gradient trajectory}}

\newtheorem{theorem}{Theorem}[section]
\newtheorem{lemma}[theorem]{Lemma}

\newtheorem{corollary}[theorem]{Corollary}
\newtheorem*{coru}{Corollary}
\newtheorem{question}[theorem]{Question}

\newtheorem{proposition}[theorem]{Proposition}

\theoremstyle{remark}
\newtheorem{remark}[theorem]{Remark}
\newtheorem{example}[theorem]{Example}

\theoremstyle{definition}
\newtheorem{definition}{Definition}

\numberwithin{equation}{section}

\DeclareMathOperator{\Interior}{Int}

\begin{document}
\title[Symplectic actions on manifolds with boundary]{Symplectic
  circle actions on manifolds with contact type boundary}

\author[A.\ Marinković]{Aleksandra Marinković}

\address[A.\ Marinković]{Matematicki fakultet, Studentski trg 16, 11
  000 Belgrade, SERBIA}

\email{aleks@math.bg.ac.rs}

\author[K.\ Niederkrüger-Eid]{Klaus Niederkrüger-Eid}

\address[K.\ Niederkrüger-Eid]{
	Institut Camille Jordan\\
        Université Claude Bernard Lyon~1\\
	43 boulevard du 11 novembre 1918\\
	F-69622 Villeurbanne Cedex \\
        FRANCE}

\email{niederkruger@math.univ-lyon1.fr}

\begin{abstract}
  Many of the existing results for closed Hamiltonian $G$-manifolds
  are based on the analysis of the corresponding Hamiltonian functions
  using Morse-Bott techniques.  In general such methods fail for
  non-compact manifolds or for manifolds with boundary.

  In this article, we consider circle actions only on symplectic
  manifolds that have (convex) contact type boundary.  In this
  situation we show that many of the key ideas of Morse-Bott theory
  still hold, allowing us to generalize several results from the
  closed setting.

  Among these, we show that in our situation any symplectic group
  action is always Hamiltonian, we show several results about the
  topology of the symplectic manifold and in particular about the
  connectedness of its boundary.  We also show that after attaching
  cylindrical ends, a level set of the Hamiltonian of a circle action
  is either empty or connected.

  We concentrate mostly on circle actions, but we believe that with
  our methods many of the classical results can be generalized from
  closed symplectic manifolds to symplectic manifolds with contact
  type boundary.
\end{abstract}

\maketitle

\tableofcontents

Many of the classical results on Hamiltonian group actions are based
on the study of the Hamiltonian functions with Morse-Bott methods: The
fixed points of a Hamiltonian circle action are precisely the critical
points of the corresponding Hamiltonian function, and its gradient
flow partitions the symplectic manifold into even dimensional
submanifolds.  Among these results one finds for example
\cite{AtiyahConvexity, GuilleminSternbergConvexity,
  McDuffFixedPointsHamiltonianGroupsActions,
  AharaHattoriHamiltonianCircleActions,
  KarshonHamiltonianCircleActions} but also many others.

For non-compact manifolds or for manifolds that have non-empty
boundary, Morse theory encounters serious problems.  The reason for
this is that gradient trajectories do not need to converge anymore to
critical points, but may instead enter and leave the manifold at the
boundary so that the topology of the underlying manifold cannot be
controlled anymore by the critical points and their stable and
unstable manifolds alone, and one needs to understand the behavior of
the gradient vector field along the boundary.

Note that in contrast to Smale's cobordism theory where the Morse
function is an auxiliary tool that can be chosen to be constant on
boundary components, the Morse-Bott function in our setup is given by
the Hamiltonian group action, and cannot be modified without changing
the action.  For this reason, most results about Hamiltonian
$\SS^1$-manifolds concerning closed symplectic manifolds, have only
been generalized to Hamiltonian actions with \emph{proper} moment
maps.

\smallskip

In this article, we will deal with the case of compact Hamiltonian
$\SS^1$-manifolds that have (convex) contact type boundary and
Hamiltonian $\SS^1$-manifolds with cylindrical ends.  This includes
some physically relevant situations as for example the configuration
space of many mechanical systems.  As we will see the contact property
implies that the boundary is also convex with respect to the gradient
flow of the Hamiltonian function, that is, gradient trajectory cannot
\emph{touch} the boundary from the inside of the manifold.  This fact
alone gives us sufficient control on the gradient flow to work out
Morse theoretic methods in case that all fixed points are confined to
the interior of the symplectic manifold.

If the circle action does have fixed points on the boundary of the
symplectic manifold, then we show that our Morse-Bott methods continue
to work provided the corresponding gradient field satisfies in the
neighborhood of the fixed points lying on the boundary an additional
property.  Circle actions on contact type boundaries automatically
satisfy this additional assumption.

\medskip

We have organized our article in the following way:

In \Cref{sec: morse-bott with convex boundary}, we explain how
Morse-Bott theory behaves on manifolds that have boundary, assuming
that none of the gradient trajectories \emph{touches} the boundary
from the inside.  Here we do not make any reference to Hamiltonian
group actions in the hope that our arguments might be interesting to
researchers from other fields than symplectic geometry.  Unknown to
us, convexity and concavity with respect to gradient flows of Morse
functions had already been studied much earlier
\cite{MorseSingularPointsBoundary}, see also the beautiful monograph
\cite{KatzMorseWithConvexBoundary}.  The situation we are interested
in imposes on us to consider Morse-Bott functions, leading to several
non-trivial complications that our knowledge had not been studied so
far because there might be critical points on the boundary (see
\Cref{example convex and concave boundary}).

\medskip

In \Cref{sec: hamiltonian with contact type boundary}, we apply then
the results obtained in the previous section to Hamiltonian actions.
For symplectic manifolds with contact type boundary we show that every
symplectic action of a compact Lie group~$G$ is automatically
Hamiltonian, that is, every symplectic $G$-manifold admits a moment
map.

\begin{main_theorem}\label{corollary: symplectic group action}
  Let $G$ be a compact Lie group that acts symplectically on a compact
  symplectic manifold with non-empty contact type boundary.  Then it
  follows that the action is Hamiltonian.
\end{main_theorem}

For a Hamiltonian $\SS^1$-manifold~$W$ with contact type boundary, we
distinguish the subsets $\xi^+$ and $\xi^-$ of the boundary, where the
orbits of the circle action are positively transverse or negatively
transverse to the contact structure.  In the context of
$3$-dimensional contact manifolds, this decomposition goes back to
\cite{Lutz_invariantes}.  We show that the properties of $\xi^+$ and
$\xi^-$ are strongly related via the gradient flow to the circle
action on the interior of the symplectic manifold.

Recall that the Hamiltonian of an $\SS^1$-action on a \emph{closed}
symplectic manifold will always have a unique component of local
maxima and a unique component of local minima.  In the case with
contact type boundary, we show that there is either a unique component
of $\xi^+$ and none of the critical points of the Hamiltonian is a
local maximum, or there is a unique component of critical points that
are local maxima, and $\xi^+$ is empty.  An analogous statement holds
for the connected components of local minima and $\xi^-$.

\begin{main_theorem}\label{theorem: fixed points and boundary components of
    hamiltonian S1-manifolds}
  Let $(W,\omega)$ be a connected compact Hamiltonian $\SS^1$-manifold
  with convex contact type boundary, and let $H\colon W\to \RR$ be an
  associated Hamiltonian function.

  \medskip
  
  \textbf{(a)} Then it follows that the set of critical points
  $\Crit(H)$ is equal to the set of fixed points $\Fix(\SS^1)$ of the
  circle action.  These decompose into finitely many connected
  components
  \begin{equation*}
    \Fix(\SS^1) = \bigsqcup_j C_j \;,
  \end{equation*}
  that intersect $\p W$ transversely.  A component~$C_j$ of
  $\Fix(\SS^1)$ is a closed symplectic submanifold if and only if
  $C_j\cap \p W = \emptyset$, otherwise $C_j$ is a compact symplectic
  manifold with convex contact type boundary~$\p C_j = C_j \cap \p W$.

  \medskip
  
  \textbf{(b)} The choice of an $\SS^1$-invariant Liouville vector
  field along the boundary defines an invariant contact
  structure~$\xi$ on $\p W$.  Let $X$ be the infinitesimal generator
  of the circle action.  Decompose then the boundary of $W$ into the
  three subsets
  \begin{align*}
    \xi^+ &= \bigl\{p\in \p W\bigm|
            \text{ $X(p)$ is positively transverse to $\xi_p$}\bigr\} \;, \\
    \xi^- &= \bigl\{p\in \p W\bigm|
            \text{ $X(p)$ is negatively transverse to $\xi_p$}\bigr\} \text{ and }\\
    \xi^0 &= \bigl\{p\in \p W\bigm| \; X(p) \in \xi_p\bigr\} \;.
  \end{align*}
  If $H' = \lambda(X)$ is the Hamiltonian function of $X$ associated
  to the Liouville form~$\lambda := \iota_Y\omega$, then
  $\xi^+ = \{H' > 0\}$, $\xi^- = \{H' < 0\}$, $\xi^0 = \{H' = 0\}$.

  The closed subset~$\xi^0$ is composed of all fixed points in
  $\Fix(\SS^1) \cap \p W$ and all non-trivial isotropic $\SS^1$-orbits
  in $\p W$.  The fixed points in $\p W$ form a finite collection of
  closed contact submanifolds.  The non-trivial isotropic orbits form
  a finite union of cooriented disjoint hypersurfaces in $\p W$ that
  separate $\xi^-$ on one side from $\xi^+$ on the other side, so that
  $\xi^0$ is nowhere dense in $\p W$.

  The hypersurfaces of non-trivial isotropic orbits do not need to be
  compact as their closure may contain fixed points.  If this is the
  case, we can also not expect that the closure of the hypersurface is
  a smooth submanifold.

  \medskip
  
  \textbf{(c)} The function~$H$ is Morse-Bott, and all the Morse-Bott
  indices of the different components of $\Crit(H)$ are even.
  
  Let $C_j$ be a fixed point component that intersects $\p W$.  It
  follows that $\p C_j = C_j \cap \p W$ necessarily lies in $\xi^0$.
  It is surrounded in $\p W$ by $\xi^+$, if and only if $C_j$ is a
  local minimum of $H$.  Similarly, $\p C_j$ is surrounded by $\xi^-$,
  if and only if $C_j$ is a local maximum of $H$.

  \smallskip
  
  One of the following mutually exclusive statements holds:
  \begin{itemize}
  \item The set of critical points~$\Crit(H)$ has a unique
    component~$C_{\max}$ whose points are local maxima (a unique
    component~$C_{\min}$ whose points are local minima) and $H$ is
    everywhere else on $W\setminus C_{\max}$ strictly smaller than on
    $C_{\max}$ (everywhere else on $W\setminus C_{\min}$ strictly
    larger than on $C_{\min}$), so that these local maxima (local
    minima) are actually the global ones.

    The subset~$\xi^+$ is empty ($\xi^-$ is empty).  The inclusion
    $C_{\max} \hookrightarrow W$ ($C_{\min} \hookrightarrow W$)
    induces a \emph{surjective} homomorphism
    $\pi_1(C_{\max}) \to \pi_1(W)$ ($\pi_1(C_{\min}) \to \pi_1(W)$).
  \item None of the points in $\Crit(H)$ are local maxima (local
    minima), and $H$ takes its global maximum on $\xi^+ \subset \p W$
    (its global minimum on $\xi^-$).

    The subset~$\xi^+$ ($\xi^-$) is open, non-empty and connected.
    The inclusion $\xi^+ \hookrightarrow W$ (or
    $\xi^- \hookrightarrow W$) induces a \emph{surjective}
    homomorphism $\pi_1(\xi^+) \to \pi_1(W)$ (or
    $\pi_1(\xi^-) \to \pi_1(W)$).
  \end{itemize}
\end{main_theorem}

Our initial aim with this project was to find symplectic manifolds
with disconnected contact boundary.  We do not know of any example of
dimension$\ge 6$, but we can show that in dimension~$4$ there aren't
any.

\begin{main_theorem}\label{boundary connected in dimension 4}
  A $4$-dimensional compact Hamiltonian $\SS^1$-manifold cannot have
  disconnected contact type boundary.
\end{main_theorem}

A compact symplectic manifold with a torus action also has always
connected contact type boundary.

\begin{main_theorem}\label{rank two connected boundary}
  Let $G$ be a compact Lie group of rank\footnote{Recall that the rank
    of a compact Lie group is the dimension of its maximal torus.}  at
  least~$2$.  If $(W,\omega)$ is a connected compact Hamiltonian
  $G$-manifold with convex contact type boundary, then it follows that
  $\p W$ cannot be disconnected.
\end{main_theorem}

\medskip

It is well-known that the level sets of a Hamiltonian function that
generates a circle action on a closed manifold is are connected or
empty.  For manifolds with contact type boundary, this claim turns out
to be wrong but in \Cref{sec: hamiltonian circle manifold with
  cylindrical ends has connected level sets} we show that the
statement can be saved by attaching cylindrical ends to the manifold.
This includes for example cotangent bundles.

We prove:

\begin{main_theorem}\label{theorem: connected level sets}
  Let $(W,\omega)$ be a connected compact Hamiltonian $\SS^1$-manifold
  that has convex contact type boundary, and let $H\colon W\to \RR$ be
  the Hamiltonian function of the circle action.  Complete $W$ by
  attaching cylindrical ends with respect to some invariant Liouville
  field and denote the resulting manifold by
  $(\widehat{W}, \widehat \omega)$ and the extended Hamiltonian by
  $\widehat{H}$.

  Then it follows that the level sets of $\widehat{H}$ are either
  connected or empty.
\end{main_theorem}

As an easy corollary, we obtain the convexity for $\TT^2$-actions, see
\Cref{sec: hamiltonian circle manifold with cylindrical ends has
  connected level sets}.

\begin{coru}[Convexity for $\TT^2$-actions]
  Let $(\widehat{W},\widehat{\omega})$ be a symplectic manifold with
  cylindrical ends equipped with a Hamiltonian $\TT^2$-action, as
  explained in \Cref{def: cylindrical end}.  The corresponding moment
  map image $\widehat{\mu}(\widehat W)$ is then a convex set.
\end{coru}

\subsection*{Acknowledgments}

We thank Krzysztof Kurdyka for explaining to us why a trajectory of a
Morse-Bott function always converges to a critical point, and Emmanuel
Giroux for suggesting to us to talk to Krzysztof Kurdyka.  We thank
Marco Mazzucchelli for helping us out with several questions regarding
the dynamics of vector fields.  We thank Dusa McDuff for discussing
with us the details of why a symplectic circle action on a closed
manifold is Hamiltonian if and only if the action has fixed points.
We thank Jean-Yves Welschinger for pointing out to us that our dynamic
viewpoint of convexity can also be interpreted as $J$-convexity for
holomorphic annuli.

\section{Definitions and preliminary remarks}\label{sec: definitions}

Let $G$ be a compact Lie group with Lie algebra~$\gfrak$ that acts
smoothly and effectively on a manifold~$W$.  To every $X\in \gfrak$,
we associate the vector field
\begin{equation*}
  X_W(p) := \left.\tfrac{d}{dt}\right|_{t=0} \exp(tX)\cdot p
\end{equation*}
called the \defin{infinitesimal generator of $X$}.

\medskip

A \defin{(weakly) Hamiltonian action} of $G$ on a symplectic
manifold~$(W,\omega)$ is a smooth action preserving the symplectic
structure, and for which we can additionally associate to every
$X\in \gfrak$ a function $H_X\colon W \to \RR$ satisfying
\begin{equation*}
  \iota_{X_{W}}\omega = -dH_X \;.  
\end{equation*}
The collection of Hamiltonian functions can be represented in a
unified way by defining a \defin{moment map}
\begin{equation*}
  \mu\colon W \to \gfrak^*
\end{equation*}
that satisfies $\langle \mu(p), X\rangle = H_X(p)$ for any $p\in W$,
and any $X\in \gfrak$.  Here $\langle \cdot, \cdot\rangle$ denotes the
natural pairing between $\gfrak^*$ and $\gfrak$.

Every weakly Hamiltonian action admits a moment map: Simply choose a
basis of $\gfrak$ given by elements $X_1,\dotsc,X_k$, each with a
Hamiltonian function $H_1,\dotsc, H_k$, and define
$\langle \mu(p), a_1 X_1 + \dotsm + a_k X_k\rangle = a_1\,H_1(p) +
\dotsm + a_k\,H_k(p)$.

We say that a weakly Hamiltonian $G$-action is \defin{Hamiltonian}, if
it can be equipped with a moment map $\mu\colon W \to \gfrak^*$ that
is $G$-equivariant.  Here $G$ acts via the coadjoint representation on
$\gfrak^*$, given for every $g\in G$, $X\in \gfrak$, and
$\nu \in \gfrak^*$ by
\begin{equation}
  \langle\Ad_g^* \nu, X\rangle := \langle \nu, \Ad_{g^{-1}}X\rangle \;.
  \label{eq:coadjoint equivariance}
\end{equation}
For a circle action, a weakly Hamiltonian action is of course
trivially Hamiltonian.

\bigskip

A symplectic manifold~$(W,\omega)$ has (convex) \defin{contact type
  boundary~$V = \p W$} if there exists a Liouville vector field~$Y$ in
a neighborhood of $V$ that points transversely \emph{out} of $W$.  The
vector field~$Y$ induces a Liouville
form~$\lambda_Y := \iota_Y \omega$ on the boundary collar which in
turn determines a (cooriented) contact
structure~$\xi = \ker \bigl(\restricted{\lambda_Y}{TV}\bigr)$ on $V$.

Note that $\alpha = \restricted{\lambda_Y}{T\p W}$ is a contact form
for $\xi$ that defines a volume form $\alpha\wedge (d\alpha)^{n-1}$
that is positive with respect to the boundary orientation of $\p W$.

\begin{example}\label{example: cotangent bundle}
  The most natural examples from classical mechanics are cotangent
  bundles with their natural actions induced by an action on the base
  manifold.
  
  Let $L$ be any closed manifold carrying a Riemannian metric.  The
  cotangent bundle~$T^*L$ is a symplectic manifold with symplectic
  structure $d\lcan$.  This is the standard phase space of classical
  mechanics.  A classical Hamiltonian function is of the form
  \begin{equation*}
    H\colon T^*L \to \RR,\quad (x,\nu_x) \mapsto
    \frac{\norm{\nu_x}^2}{2m} + V(x) \;,
  \end{equation*}
  where $x\in L$ and $\nu_x \in T^*_xL$.  The first term
  $\frac{1}{2m}\,\norm{\nu_x}^2$ represents the kinetic energy of the
  system, and $V\colon L \to \RR$ the potential energy.

  We can define a Liouville vector field~$Y$ by the equation
  $\iota_Y d\lcan = \lcan$.  Indeed, if $\pi\colon T^*L \to L$ is the
  natural bundle projection, the canonical Liouville form~$\lcan$
  satisfies for every $(x,\nu_x)\in T^*L$ and every
  $v\in T_{(x,\nu_x)}T^*L$ the equation
  \begin{equation*}
    \lcan(v) = \nu_x\bigl(d\pi_{(x,\nu_x)}\,v\bigr) \;.
  \end{equation*}
  It is not hard to show that $Y$ is transverse to any level set of
  $H$ that does not intersect the $0$-section so that if we choose
  $C \gg 0$ large enough, the subdomain
  $H^{-1}\bigl((-\infty,C]\bigr)$ yields a symplectic manifold with
  contact type boundary (the boundary being simply a constant energy
  level hypersurface).

  \medskip

  Note that every diffeomorphism $\phi\colon L\to L$ lifts to a
  symplectomorphism $\Phi\colon T^*L\to T^*L$ given for every $x\in L$
  and $\nu_x \in T_x^*L$ by
  \begin{equation*}
    \Phi(x,\nu_x) := \bigl(\phi(x),\;\nu_x \circ (d\phi_x)^{-1}\bigr)
    \in T_{\phi(x)}^*L \;,
  \end{equation*}
  where $d\phi_x\colon T_xL \to T_{\phi(x)}L$ is the differential of
  $\phi$ at the point~$x$.

  It follows that any smooth action of a Lie-group~$G$ on $L$ induces a
  natural $G$-action on $T^*L$ by symplectomorphisms, and in fact,
  combining \Cref{lemma: hamiltonian function through liouville form}
  with $\bigl(d\pi \, X_{T^*L}\bigr)(x,\nu_x) = X_L(x)$, we obtain a
  $G$-equivariant moment map~$\mu$ defined by
  \begin{equation*}
    \langle \mu(x,\nu_x), X \rangle =
    \lcan\bigl(X_{T^*L}(x,\nu_x)\bigr) = \nu_x \bigl(d\pi_{(x,\nu_x)}
    X_{T^*L}(x,\nu_x)\bigr) = \nu_x\bigl(X_L(x)\bigr)
  \end{equation*}
  for every $X$ in the Lie algebra of $G$, and every
  $(x,\nu_x)\in T^*L$.

  This shows that the $G$-action on $T^*L$ is always Hamiltonian.
  Additionally assuming that $G$ also preserves the energy
  function~$H$, we obtain a Hamiltonian $G$-action on the symplectic
  domain $H^{-1}\bigl((-\infty,C]\bigr)$ with contact type boundary.
\end{example}

\section{Morse-Bott functions on manifolds with convex
  boundary}\label{sec: morse-bott with convex boundary}

A classical result states that the Hamiltonian function of a circle
action on a closed symplectic manifold is of Morse-Bott type, and the
study of Hamiltonian actions on closed manifolds relies fundamentally
on this fact.  Unfortunately, Morse theory breaks down when
considering non-compact manifolds or manifolds with boundary because
gradient trajectories can escape, and cannot be controlled anymore.

In this article we are interested in symplectic manifolds that have
contact type boundary.  In this case we will show that the boundary is
\emph{convex with respect to the gradient flow} of the Hamiltonian
function.  Additionally we need a second fact that in our case will
also always be satisfied: The stable set of a local maximum and the
unstable set of a local minimum need to be submanifolds.

These two properties allow us to generalize many results from closed
manifolds to our situation.

\medskip

The description in this section is about gradient flows and does not
make any reference to symplectic topology.  We will discuss in how far
the study of Morse-Bott functions on closed manifolds generalizes to
compact manifolds whose boundary is convex with respect to the
gradient field (see \Cref{def: X-convex boundary}).  In the context of
this article, we define a Morse-Bott function for a manifold with
boundary as follows (for other possible definitions see
\cite{LaudenbachMorseBoundary,MorseBottBoundary}).

\begin{definition}\label{def:Morse-Bott}
  Let $W$ be a compact manifold with boundary.  We say that a smooth
  function $f\colon W\to \RR$ is \defin{Morse-Bott} if
  \begin{itemize}
  \item The set of its critical points~$\Crit(f)$ is a disjoint union
    of finitely many smooth compact submanifolds
    \begin{equation*}
      \Crit(f) = \bigsqcup_j C_j
    \end{equation*}
    that might possibly have boundary.
  \item If a component~$C_j$ intersects $\p W$ then it does so
    transversely and the boundary of $C_j$ is
    $\p C_j = C_j \cap \p W$.  In particular, isolated critical points
    never lie in $\p W$, and more generally the \emph{closed}
    components of $\Crit(f)$ are those that do not intersect the
    boundary of $W$.
  \item The Hessian of $f$ is at every critical point $p\in C_j$
    non-degenerate in the normal direction to $C_j$.  At a boundary
    point $p\in \Crit(f) \cap \p W$, we mean by this that the Hessian
    of $\restricted{f}{\p W}$ is non-degenerate when restricted to the
    normal direction of $\p C_j$ in $\p W$.
  \end{itemize}

  The \defin{index of $f$ at a component~$C_j\subset \Crit(f)$} is the
  pair
  \begin{equation*}
    \bigl(i^-(C_j) , i^+(C_j)\bigr) \;,
  \end{equation*}
  where $i^-(C_j)$ is the number of negative eigenvalues of the
  Hessian, and $i^+(C_j)$ is the number of positive eigenvalues of the
  Hessian, so that
  \begin{equation*}
    \dim W =  \dim C_j + i^+(C_j) + i^-(C_j)    \;.
  \end{equation*}
\end{definition}

\begin{remark}\label{indices of boundary crit in relation to normal index}
  Let $f$ be a Morse-Bott function on a compact manifold~$W$ with
  boundary.  Assume that $C_j$ is a component of $\Crit(f)$ that
  intersects $\p W$ so that $\p C_j = C_j \cap \p W$ is a component of
  Morse-Bott type of the critical points
  $\Crit\bigl(\restricted{f}{\p W}\bigr)$ for the restricted function
  $\restricted{f}{\p W}$.  The indices of $C_j$ in $W$ and the ones of
  $\p C_j$ in $\p W$ are related by
  \begin{equation*}
    i^-(C_j) = i^-(\p C_j)\;,\quad
    i^+(C_j) = i^+(\p C_j)\;, \quad \text{ and } \quad
    \dim(C_j) =  \dim(\p C_j) + 1 \;.
  \end{equation*}
\end{remark}

\bigskip

Let $(W, g)$ be a compact Riemannian manifold that may have boundary,
and let $f\colon W\to \RR$ be any smooth function.  Define the
gradient vector field~$\nabla f$ with respect to the metric~$g$, and
denote its flow by $\Phi^{\nabla f}_t$.  Independently of whether $f$
is Morse-Bott, we can define the following stable and unstable subsets
(that may or may not be submanifolds).

Let $A$ be any subset of $W$.  We say that the gradient trajectory
$\gamma(t) = \Phi^{\nabla f}_t(p)$ through $p\in W$ \defin{accumulates
  at $A$}, if $\gamma$ exists for any positive time and if we find for
every neighborhood~$U$ of $A$, a time~$T_U > 0$ such that all
$\gamma(t)$ with $t > T_U$ lie in $U$.

\begin{definition}
  The \defin{stable} and \defin{unstable sets of a component
    $C_j \subset \Crit(f)$} are defined as
  \begin{equation*}
    \begin{split}
      W^s(f;C_j) &:= \bigl\{p\in W\bigm|\;
      \text{$\Phi_{\; t}^{\nabla f}(p)$ accumulates at $C_j$} \bigr\}
      \quad\text{ and }\\
      W^u(f;C_j) &:= \bigl\{p\in W\bigm|\;
      \text{$\Phi_{\; t}^{-\nabla f}(p)$ accumulates at $C_j$} \bigr\}
    \end{split}
  \end{equation*}
  respectively.
\end{definition}

For Morse-Bott functions, a gradient trajectory converges on a closed
manifold always to a critical point.  This result which simplifies
significantly the previous definition was explained to us by Krzysztof
Kurdyka.  The underlying idea going back to Łojasiewicz for analytic
functions, see \cite{Lojasiewicz} where applied to Morse-Bott
functions in \cite{KurdykaGradientConjecture}.

\begin{theorem}[Łojasiewicz]\label{convergence gradient trajectory}
  Let $(W, g)$ be a compact Riemannian manifold that might possibly
  have boundary, and let $f\colon W\to \RR$ be a Morse-Bott function.

  If $\gamma$ is a gradient trajectory such that $\gamma(t)$ is
  defined for all $t\ge 0$, then it follows that $\gamma(t)$ converges
  for $t\to \infty$ to a critical point of $f$.
\end{theorem}
\begin{proof}
  The proof of this statement can be found in \Cref{section:
    Lojasiewicz} of the appendix.  We follow the explanations kindly
  given to us by Krzysztof Kurdyka.
\end{proof}

Thus if $f$ is a Morse-Bott function, we can equivalently characterize
the stable and unstable sets of a component $C_j \subset \Crit(f)$ by
\begin{equation*}
  W^s(f;C_j) := \bigl\{p\in W\bigm|\; \lim_{t\to \infty} \Phi_{\;
    t}^{\nabla f}(p) \in C_j \bigr\}
  \quad\text{ and }\quad
  W^u(f;C_j) := \bigl\{p\in W\bigm|\; \lim_{t\to -\infty}
  \Phi_{\; t}^{\nabla f}(p) \in C_j \bigr\}
\end{equation*}
respectively.

It is a well-known fact that on \emph{closed} manifolds, stable and
unstable subsets of Morse-Bott functions are smooth submanifolds.
This follows from the Hadamard-Perron theorem in combination with the
fact that functions increase along their gradient trajectories.  This
is true even without assuming any particular form of the Riemannian
metric close to the critical points.  We state this result in
\Cref{stable subset smooth on closed manifold} in the appendix as a
reference to deal later on with manifolds with boundary.

\medskip

If $W$ has boundary, the flow of $X$ is usually not defined for all
times $t\in \RR$, and we denote for every point $p\in W$ the maximal
time up to which $\Phi_t^{\nabla f}(p)$ exists in forward direction by
$T_{\max}(p)$ and the maximal time in backward direction by
$T_{\min}(p)$, so that $T_{\max}(p) \in [0,\infty) \cup \{+\infty\}$
and $T_{\min}(p) \in (-\infty,0] \cup \{-\infty\}$.

\begin{definition}
  The \defin{stable and unstable sets of a boundary
    component~$\p_k W$} are the subsets
  \begin{align*}
    W^s(f;\p_k W) &:= \bigl\{p\in W\bigm|\; \text{$T_{\max}(p) < \infty$ and
                    $\Phi_{T_{\max}(p)}^{\nabla f}(p)  \in \p_k W$} \bigr\}
    \intertext{ and }
    W^u(f;\p_k W) &:= \bigl\{p\in W\bigm|\; \text{$T_{\min}(p) > -\infty$
                    and $\Phi_{T_{\min}(p)}^{\nabla f}(p) \in \p_k W$} \bigr\}
  \end{align*}
  respectively.
\end{definition}

Note that if $\Phi^{\nabla f}_t(p)$ tends for $t\to \infty$ to a
critical point~$p_\infty$ that lies in the intersection of a
component~$C_j \subset \Crit(f)$ with the boundary~$\p W$, then $p$
will clearly lie in the stable set of $C_j$, but by our definition it
does \emph{not} lie in the stable set of the boundary component: for
this the boundary component has to be reached by the gradient
trajectory \emph{in finite time} (and stop there); converging to the
boundary for $t\to \infty$ is not sufficient.  In particular, all
points of $\p C_j$ lie in $ W^s(f;C_j)$, but not in $W^s(f;\p_k W)$.

\begin{lemma}\label{stable sets partition}
  Let $(W, g)$ be a compact Riemannian manifold with boundary, and let
  $f\colon W\to \RR$ be a Morse-Bott function.  Then it follows that
  $W$ is partitioned by the stable sets
  \begin{equation}
    W = \qquad \bigsqcup_{\mathclap{C_j\subset \Crit(f)}} W^s(f;C_j)
    \qquad\bigsqcup\qquad \bigsqcup_{\mathclap{\p_k W\subset \p W}} W^s(f;\p_k W)
    \label{eq: partition into stable manifolds}
  \end{equation}
  where the $C_j$ and $\p_kW$ are the different components of
  $\Crit(f)$ and $\p W$ respectively.  Analogously, $W$ is also
  partitioned by the corresponding unstable sets.
\end{lemma}

The proof of this lemma generalizes easily to any smooth function~$f$
that does not need to be Morse-Bott but for which all components of
$\Crit(f)$ are isolated in the interior of $W$.

\begin{proof}
  Let $p$ be any point in $W$, and let $\gamma$ be the gradient
  trajectory with $\gamma(0) = p$.  If $\gamma$ is defined in forward
  direction only up to time~$T_{\max}(p) < \infty$, then $\gamma$
  necessarily hits one of the boundary components at $T_{\max}(p)$,
  and $p$ lies in the stable set of that component.  If the flow is
  instead defined for all $t>0$, then it follows from
  \Cref{convergence gradient trajectory} that $\gamma(t)$ converges
  for $t\to \infty$ to a critical point of $f$ so that
  $p \in W^s(f;C_j)$ for some component~$C_j$ of $\Crit(f)$.

  This shows that every point in $W$ lies in the stable set of a
  boundary component $\p_k W$ or in the stable set of a critical
  component~$C_j$.
  
  \smallskip

  To prove that $W$ is partitioned by the stable sets, note firstly
  that a point cannot lie in the stable set of a boundary component
  and in the stable set of a critical point, because the first
  condition requires the flow only to exist for finite time while the
  other one requires that the flow exists for all times.  Furthermore
  a point cannot lie in the stable sets of two different boundary
  components $\p_k W$ and $\p_{k'}W$, because by our definition $p$
  lies in $W^u(f;\p_k W)$ if and only if the trajectory through $p$
  \emph{ends} on $\p_k W$ (in finite time).  The trajectory might
  intersect other boundary components before reaching the final one,
  but this is not sufficient to lie in their respective stable sets.
  Finally, $p$ lies in the stable set of a component~$C_j$ of
  $\Crit(f)$ if the trajectory of $p$ \emph{converges} to a point in
  $C_j$, thus excluding that it also converges to a point in another
  component of $\Crit(f)$.

  \smallskip

  The argument in backward time does not require any modifications and
  shows that the corresponding statement about the unstable sets is
  also true.
\end{proof}

\medskip

We need to be more precise about the boundary: Let $W$ be a manifold
with non-empty boundary, and let $X$ be a vector field on $W$.  We
distinguish the following three subsets that are determined by the
behavior of $X$ along $\p W$
\begin{equation*}
  \begin{aligned}
    \p^+ W &:= \bigl\{ p\in \p W\bigm|\, \text{$X(p)$ points
      transversely out of  $W$}\bigr\} \;, \\
    \p^- W &:= \bigl\{ p\in \p W\bigm|\, \text{$X(p)$ points
      transversely into $W$}\bigr\} \;, \\
    \p^0 W &:= \bigl\{ p\in \p W\bigm|\, X(p)\in T_p (\p W)\bigr\} \;.
  \end{aligned}
\end{equation*}

\begin{definition}\label{def: X-convex boundary}
  We say that the boundary of $W$ is \defin{convex with respect to the
    vector field~$X$}, if the maximal trajectory of $X$ of any point
  $p\in \p^0 W$ is just $p$ itself, that is, $X(p)$ either vanishes so
  that $\Phi_t^X(p) = p$ for all $t\in \RR$ or if $X(p)\ne 0$, then
  the flow~$\Phi_t^X(p)$ is not defined for any $t \ne 0$.
\end{definition}

The intuition behind this definition is that if we embed $W$ into a
slightly larger open manifold~$\widehat{W}$, and then we extend the
vector field~$X$ smoothly to $\widehat{W}$, the convexity of $\p W$
implies that the only non-trivial trajectories of $X$ that are tangent
to $\p W$ touch the boundary from the \emph{outside} of $W$, see
\Cref{fig: X-convex boundary}.

\begin{figure}[htbp]
  \centering
  \includegraphics[height=2.5cm,keepaspectratio]{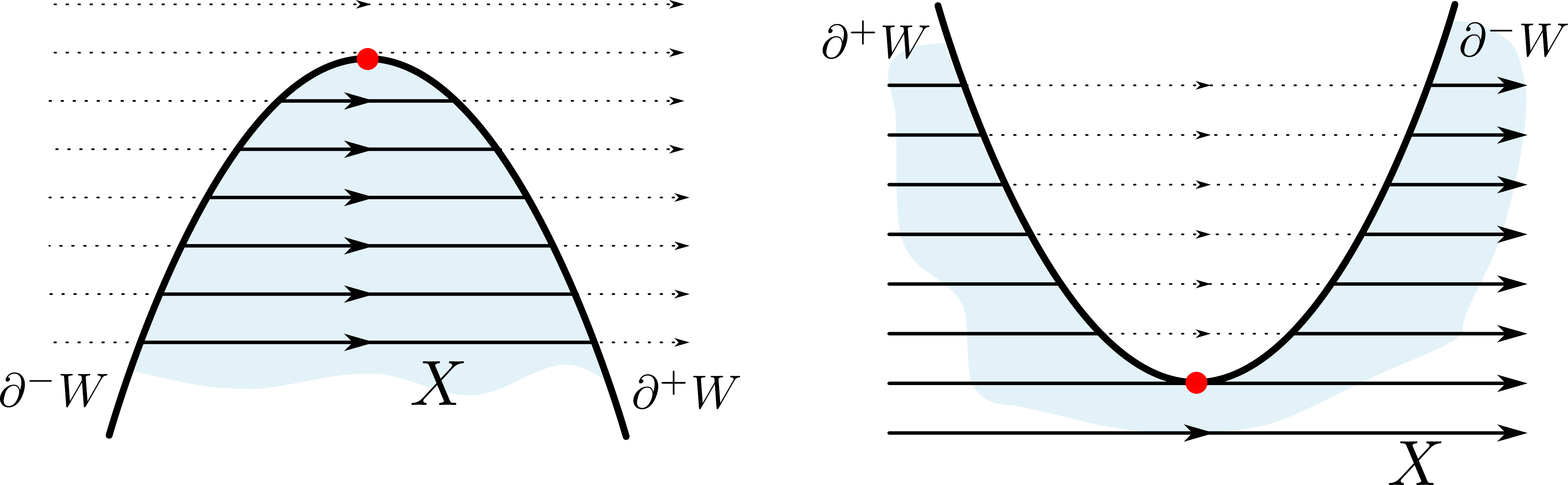}
  \caption{\label{fig: X-convex boundary} In the two cases depicted
    above, the domain~$W$ is colored in blue, and the vector field~$X$
    has horizontal flow lines.  On the left side, we have drawn the
    case of a boundary that is \emph{convex with respect to $X$}.  The
    vector field enters through $\p^- W$ on one side and leaves again
    through $\p^+ W$ on the other side.  It is only tangent to $\p W$
    at the red point, but touching $W$ from the outside its flow is
    neither defined (in $W$) for positive nor negative times.  On the
    right side, $X$ also has a tangency with $\p W$ in the red dot,
    but this time the flow comes from the interior of $W$, becomes
    tangent to $\p W$, and then moves back into the interior of $W$ so
    that this boundary is \emph{not} convex.}
\end{figure}

\begin{example}\label{example convex and concave boundary}
  \begin{itemize}
  \item [(a)] Consider the vector field~$X := \partial_x$ on $\RR^2$
    so that the flow lines are horizontal lines.  It is then easy to
    convince oneself that the closed unit disk (or any other
    geometrically convex domain) is convex with respect to $X$.  For
    an annulus on the other hand, the inner boundary is not convex
    with respect to $X$.
  \item [(b)] Consider again the closed annulus~$A$ of all
    points~$(x,y)$ in $\RR^2$ with $1/2 \le x^2 + y^2 \le 1$, and let
    $X = -x\, \partial x$ be the gradient vector field of the function
    $f(x,y) = -x^2/2$ with respect to the Euclidean metric.

    Note that against all intuition, this time $\p A$ is convex,
    because $X$ vanishes at all boundary points where $X$ is not
    transverse to $\p A$.  This shows that convexity by itself is not
    sufficient to obtain for example results like the ones in
    \Cref{topology simple if no codim 1 stable manifolds}.  We will
    have to impose additionally properties on the vector field close
    to critical points lying on the boundary of the domain.
  \end{itemize}
  \begin{figure}[htbp]
    \begin{center}
      \includegraphics[height=3.5cm,keepaspectratio]{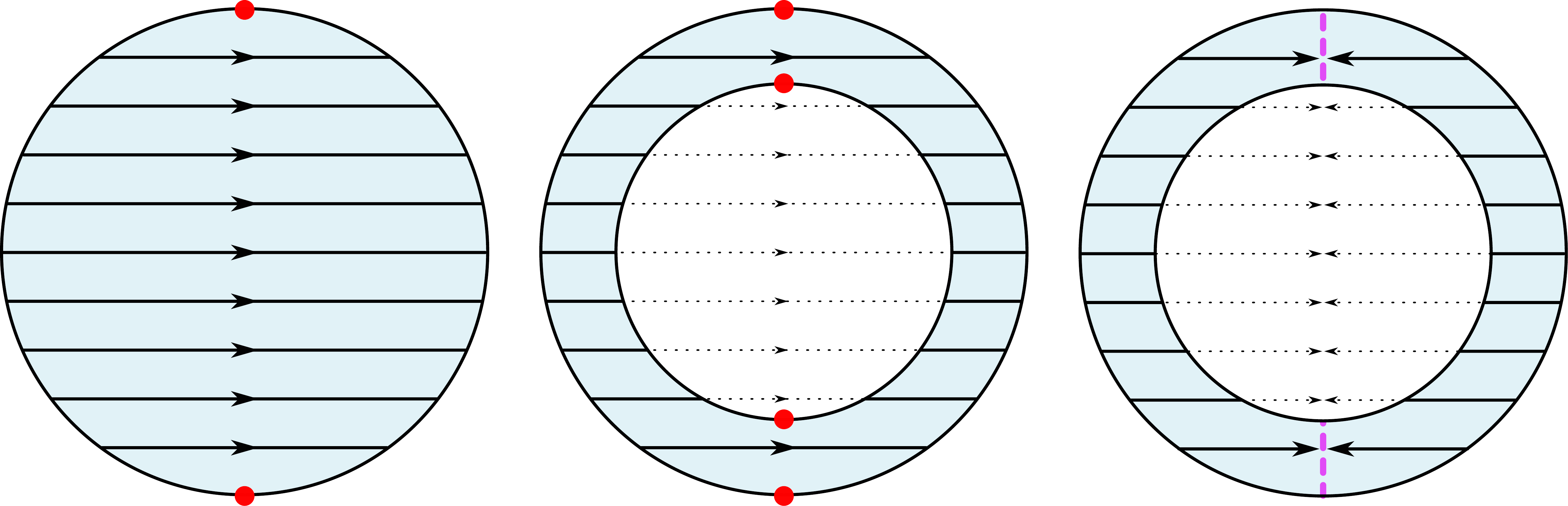}
    \end{center}
    \caption{The left picture and the central picture show the two
      situations described in \Cref{example convex and concave
        boundary}.(a). The disk is convex with respect to the
      horizontal flow, while the annulus isn't.  The picture on the
      right shows \Cref{example convex and concave boundary}.(b).
      There is a vertical line of critical points, and the field
      points horizontally towards this line.  Note that this example
      is also convex, because the gradient field is everywhere
      transverse to the boundary except for the points where the
      gradient vanishes.}
  \end{figure}
\end{example}

In this article, we will be mostly interested in gradient vector
fields.  When there is no possible confusion we use the following
notation.

\begin{definition}
  Let $(W,g)$ be a Riemannian manifold with boundary carrying a smooth
  function $f\colon W\to \RR$.  If $\p W$ is convex with respect to
  $\nabla f$ we often say for simplicity that $\p W$ is
  \defin{$\nabla$-convex}.
\end{definition}

The gradient vector field is almost everywhere transverse to a
$\nabla$-convex boundary:

\begin{lemma}\label{convex boundary cannot be everywhere tangent to
    gradient}
  Let $(W,g)$ be a connected Riemannian manifold with boundary, and
  let $f\colon W\to \RR$ be a non-constant Morse-Bott function such
  that $W$ has convex boundary with respect to $\nabla f$.

  The subset of points where $\nabla f$ is transverse to the boundary
  is open and dense in $\p W$.  In particular, none of the boundary
  components of $W$ can lie in
  $\p^0W = \bigl\{ p\in \p W\bigm|\, \nabla f(p)\in T_p (\p
  W)\bigr\}$.
\end{lemma}
\begin{proof}
  Assume that $U\subset \p W$ is an open set in $\p W$ that lies in
  $\p^0W$.  If $p \in U$ were a point such that $\nabla f(p) \ne 0$,
  then attach a small collar to $W$ along $\p W$ and extend $\nabla f$
  smoothly to this enlarged manifold.

  Choose a flow-box chart in $\RR^n$ with coordinates
  $(x_1,\dotsc, x_n)$ such that $p$ corresponds to the origin and such
  that $\nabla f$ agrees with $\partial_{x_1}$.  We can assume by a
  rotation of the coordinates that $\p W$ is at the origin of the
  coordinate chart tangent to the hyperplane $\{x_n = 0\}$.  Then we
  can write $\p W$ as a graph of a smooth function~$h$, that is,
  $\p W$ is parametrized by
  \begin{equation*}
    (x_1,\dotsc, x_{n-1})\mapsto
    \bigl(x_1,\dotsc, x_{n-1}, h(x_1,\dotsc, x_{n-1})\bigr)
  \end{equation*}
  where $h$ is a function that vanishes at the origin, and whose
  differential also vanishes at the origin.

  By our assumption that $\nabla f$ is everywhere along $U$ tangent to
  $\p W$ it follows that $\partial_{x_1} h = 0$ on a neighborhood of
  the origin so that $h$ does not depend on the $x_1$-coordinate.  In
  particular, we obtain that $\gamma(s) = (s,0,\dotsc,0)$ is a segment
  of a non-trivial gradient trajectory that is contained in $\p W$.
  This is a contradiction to the convexity assumption.  We deduce that
  $\nabla f$ needs to vanish everywhere on $U$.

  By our definition of Morse-Bott function, $\Crit(f)$ intersects
  $\p W$ transversely.  Thus if $U\subset \Crit(f)$, there would need
  to be a neighborhood of $p$ in $W$ that lied in $\Crit(f)$, and
  since $\Crit(f)$ is a finite union of submanifolds, it would follow
  that $W = \Crit(f)$ so that $f$ had to be constant.
\end{proof}

We can often verify the boundary convexity using the following
(sufficient but not necessary) criterion.

Let $W$ be a manifold with boundary carrying a smooth vector
field~$X$.  Choose a boundary collar of the form
$(-\epsilon,0]\times \p W$ with coordinates $(s,p)$ such that
$\{0\}\times \p W$ corresponds to $\p W$.  The vector field~$X$ is on
this collar of the form
\begin{equation}
  X(s,p) = S(s,p)\, \p_s + Z(s,p) \;,
  \label{eq: decomposition of vector field with respect boundary collar}
\end{equation}
where $S(s,p)$ is a smooth function, and $Z(s,p)$ is a vector field
that is tangent to the slices $\{s_0\}\times \p W$.

Clearly, it follows that
$\p^\pm W = \bigl\{p\in \p W |\; \pm S(0,p) > 0\bigr\}$, and
$\p^0 W = \bigl\{p\in \p W |\; S(0,p) = 0\bigr\}$.

\begin{definition}\label{def: strongly convex}
  The boundary~$\p W$ is \defin{strongly convex with respect to $X$},
  if $X$ given by \Cref{eq: decomposition of vector field with respect
    boundary collar} satisfies at every point $p \in \p^0 W$ either
  that
  \begin{itemize}
  \item $Z(0,p)$ vanishes or that
  \item $dS\bigl(Z\bigr) > 0$.
  \end{itemize} 
\end{definition}

\begin{proposition}\label{read off convexity from derivatives}
  \emph{Strong} convexity implies convexity.
\end{proposition}
\begin{proof}
  Assume that $p\in \p^0 W$ so that $S(0,p) = 0$.  If $Z(0,p) = 0$,
  then $X(p) = 0$ and the trajectory through $p$ is constant and thus
  satisfies the definition of convexity.

  If on the other hand $X(p) \ne 0$, then if follows in particular
  that $Z(0,p) \ne 0$.  Extend $X$ smoothly to
  $(-\epsilon,\epsilon)\times \p W$, and let
  $\gamma(t) = \bigl(s(t), p(t)\bigr)$ be the trajectory of $X$
  through $\gamma(0) = (0,p)$.  The convexity of $\p W$ at $\gamma(0)$
  is equivalent to the function~$s(t)$ having a strict minimum at
  $t=0$.

  Choose a chart for $\p W$ that is centered at $p$ and that has
  coordinates $\x = (x_1,\dotsc,x_n)$ such that
  $Z(0,\dotsc,0) = \p_{x_1}$.  In these coordinates, the trajectory is
  of the form $\gamma(t) = \bigl(s(t); x_1(t), \dotsc, x_n(t)\bigr)$,
  where $s(t), x_1(t),\dotsc, x_n(t)$ are smooth functions satisfying
  $s(0) = x_1(0) = \dotsb = x_n(0) = 0$, and
  \begin{equation*}
    s'(t) = S\bigl(s(t); x_1(t), \dotsc, x_n(t)\bigr)\;, \text{ and }
    x_j'(t) = Z_j\bigl(s(t); x_1(t), \dotsc, x_n(t)\bigr)
    \text{ for all $j\in \{1,\dotsc,n\}$.}
  \end{equation*}

  A sufficient condition for $s(t)$ to have a strict minimum at $t= 0$
  is that $s'(0) = 0$ and $s''(0) > 0$.  The first condition is
  obviously satisfied, since $s'(0) = S(0,\dotsc,0) = 0$; for the
  second condition we compute
  \begin{equation*}
    s''(0) = s'(0)\cdot \frac{\partial S}{\partial s}(0,\dotsc,0) +
    \sum_{j=1}^n x_j'(0)\cdot \frac{\partial S}{\partial x_j} (0,\dotsc,0)
    = \frac{\partial S}{\partial x_1} (0,\dotsc,0) \;.
  \end{equation*}
  Since $\frac{\partial S}{\partial x_1} (0,\dotsc,0) = dS (Z) > 0$,
  the second condition is also verified.
\end{proof}

\bigskip

\begin{definition}
  Let $\p^+W$, $\p^-W$, and $\p^0W$ be defined with respect to the
  gradient field~$\nabla f$.  Then we denote the \defin{stable subset
    of a connected component~$\p_l^+W$ of $\p^+W$} by
  \begin{align*}
    W^s(f;\p_l^+ W) &:= \bigl\{p\in W\bigm|\; \text{$T_{\max}(p) < \infty$ and
                    $\Phi_{T_{\max}(p)}^{\nabla f}(p)  \in \p_l^+ W$} \bigr\} \; ,
    \intertext{and  the \defin{unstable subset of a connected component~$\p_l^-W$ of $\p^-W$} by}
    W^u(f;\p_l^- W) &:= \bigl\{p\in W\bigm|\; \text{$T_{\min}(p) > -\infty$
                      and $\Phi_{T_{\min}(p)}^{\nabla f}(p) \in \p_l^- W$} \bigr\} \;.
  \end{align*}
\end{definition}

Unfortunately, if $C_j$ is a component of $\Crit(f)$ intersecting
$\p W$, then the stable and unstable subsets of $C_j$ will in general
not be \emph{nicely embedded} submanifolds (even assuming that $\p W$
is $\nabla$-convex).

Consider again \Cref{example convex and concave boundary}.(b).  The
gradient trajectories are horizontal segments pointing towards the
$y$-axis.  The stable set of the inner boundary is an open subset, but
the stable set of the components of critical points are not.  The
reason is that the critical points that lie on the inner boundary have
gradient trajectories that do not lie in the interior of the
corresponding stable set.  Note that no such problem appears at the
critical points lying on the outer boundary.

This behavior is depicted in \Cref{fig: stable subset is manifold if
  boundary suitable}, and further illustrated by the elementary
example below.  The behavior does not only depend on the Morse-Bott
function itself, but also on the choice of the Riemannian metric.

\begin{example}\label{example gradient along boundary depends on metric}
  Let $f\colon W\to \RR$ be a Morse-Bott function on a $2$-dimensional
  Riemannian manifold~$(W,g)$ with boundary, and let $C$ be a
  $1$-dimensional component of $\Crit(f)$ intersecting $\p W$.

  Take a Morse-Bott chart with coordinates $(x,z)$ centered at a point
  $p\in C\cap \p W$ such that $(x,0)$ corresponds to the boundary of
  $W$ and such that $f(x,z) = -x^2$ on this chart.

  Assume the Riemannian metric~$g_1$ to be of the form
  \begin{equation*}
    g_1(x,z) =
    \begin{pmatrix}
      1  &  c \\
      c & 1
    \end{pmatrix} \;,
  \end{equation*}
  where $c$ is a constant such that $|c| < 1$.  The gradient
  vector field is then $\nabla_1 f = \frac{2x}{1-c^2} \, (-1, c)$.

  The sign of $2x c$ determines if $\nabla_1 f$ points transversely
  into or out of $W$.  The oriented trajectories of $\nabla_1 f$ are
  for $x<0$ parallel to the vector~$(1, -c)$, and for $x>0$ parallel
  to $(-1, c)$.  Thus if $c>0$, then the stable set~$W^s(f;C)$ of $C$
  is the union of $\bigl\{ x\le 0, \; z\le 0 \bigr\}$ and
  $\bigl\{ x> 0, \; z\le -c x \bigr\}$, that is, the stable set is
  \emph{not} a neat submanifold\footnote{The boundary of $W^s(f,C)$
    does not lie in $\p W$ and as a consequence
    $W^s(f,C)\cap \Interior W$ is not an open subset of
    $\Interior W$.}  of $W$.  This situation corresponds to the
  behavior shown on the left side of \Cref{fig: stable subset is
    manifold if boundary suitable}.

  \medskip

  Consider now instead the Riemannian metric~$g_2$ of the form
  \begin{equation*}
    g_2(x,z) =
    \begin{pmatrix}
      1  &  -c x \\
      -c x & 1
    \end{pmatrix} \;,
  \end{equation*}
  where $c$ is a constant such that $|c| < 1$.  The gradient vector
  field is then $\nabla_2 f = -\frac{2x}{1-c^2x^2} \, (1, c x)$.

  The oriented trajectories of $\nabla_2 f$ are for $x<0$ parallel to
  the vector~$(1, c x)$, and for $x>0$ parallel to $(-1, -c x)$.  If
  $c >0$, then the stable set~$W^s(f;C)$ fills up the entire chart,
  and in particular it is a neat submanifold of $W$.  The considered
  situation is thus of the type shown on right side of \Cref{fig:
    stable subset is manifold if boundary suitable}.
\end{example}

We will study now the properties of the different stable and unstable
subsets.  Under additional technical assumptions (that are always
satisfied for Hamiltonian circle actions), we will be able to show
that all stable/unstable subsets of maximal dimension, that is, the
ones that correspond to local extrema or to the positive/negative
boundary components are all mutually disjoint open subsets of $W$.

For some time, we had hoped that we could always choose a Riemannian
metric that preserved the convexity of the boundary and such that all
stable and unstable subsets of a Morse-Bott function would be smooth
submanifolds.  The \Cref{example convex and concave boundary}.(b)
contradicts \Cref{topology simple if no codim 1 stable manifolds}
showing that our belief was wrong.

\medskip

We call a point of $W$ that does not lie in $\p W$, an \defin{interior
  point of $W$}.  Furthermore, if $A$ is a subset of $W$, we denote
the set of all points in $A$ that are interior points of $W$, that is,
$A\setminus \p W$ by $\Interior A$.  In particular,
$\Interior W = W \setminus \p W$.

\begin{proposition}\label{stable sets are submanifolds}
  Let $(W, g)$ be a compact Riemannian manifold with boundary, and let
  $f\colon W\to \RR$ be a Morse-Bott function.  Suppose that the
  boundary of $W$ is $\nabla$-convex, then the stable and unstable
  sets satisfy:
  \begin{itemize}
  \item[(a)] All subsets $W^s(f;\p_l^+ W)$ and $W^u(f;\p_l^- W)$ are
    open in $W$ and intersect $\p W$ transversely.
  \item[(b)] If $C_j \subset \Crit(f)$ is a component that may or may
    not intersect $\p W$, then it follows that the interior of the
    stable and unstable sets $\Interior W^s(f;C_j)$ and
    $\Interior W^u(f;C_j)$ are \emph{contained} in smooth submanifolds
    of dimension $\dim C_j + i^-(C_j)$ and $\dim C_j + i^+(C_j)$
    respectively.
  \item[(c)] If $C_j \subset \Crit(f)$ does not intersect $\p W$, then
    it follows that
    \begin{equation*}
      W^s(f;C_j) \quad\text{ and }\quad W^u(f;C_j)
    \end{equation*}
    are smooth submanifolds with boundary.  Their respective
    dimensions are $\dim C_j + i^-(C_j)$ and $\dim C_j + i^+(C_j)$,
    and they intersect $\p W$ transversely, and
    $\p W^s(f;C_j) \subset \p W$ and $\p W^u(f;C_j)\subset \p W$.
  \end{itemize}
\end{proposition}
\begin{proof}
  The convexity assumption imposes that any non-trivial gradient
  trajectory intersecting the boundary of $W$ necessarily does so
  transversely.  Thus if $p_0$ and $p_1$ are any two points in the
  interior of $W$ that lie on the same gradient trajectory~$\gamma$,
  say with $\gamma(0) = p_0$ and $\gamma(T) = p_1$, then it follows
  that none of the points $\gamma(t)$ with $t\in [0,T]$ can intersect
  $\p W$ so that $\gamma\bigl([0,T]\bigr)$ lies in $\Interior W$.  It
  follows that there are open neighborhoods~$U_0$ of $p_0$ and $U_1$
  of $p_1$ such that $\Phi^{\nabla f}_T$ restricts to a diffeomorphism
  between $U_0$ and $U_1$.

  \smallskip
  
  \textbf{(a)} To show that $W^s(f;\p_l^+ W)$ is open, note first that
  transversality is an open property so that $\p_l^+ W$ is an open
  subset of $\p W$.  Let $p_0$ be a point that lies in $\p_l^+ W$,
  then we can choose a small open neighborhood~$U$ of $p_0$ in $\p W$,
  and an $\epsilon > 0$ such that $(t,p)\mapsto \Phi^{\nabla f}_t(p)$
  defines a diffeomorphism of $(-\epsilon, 0]\times U$ onto a
  neighborhood of $p_0$ in $W$.  We have thus found an open
  neighborhood of $p_0$ that lies in $W^s(f;\p_l^+ W)$.

  If $p_0\in W^s(f;\p_l^+ W)$ is now an interior point of $W$, then
  let $\gamma$ be the gradient trajectory starting at $p_0$
  intersecting $\p_l^+ W$ at time $t = T$.  In order to avoid any
  technicalities due to the boundary, use that $\gamma(T)$ has a small
  open neighborhood~$U$ that lies in $W^s(f;\p_l^+ W)$.  If
  $T'\in (0,T)$ is chosen sufficiently close to $T$, then
  $p_1 = \gamma(T')$ will lie in $U$.  By our remark above, there are
  open neighborhoods~$U_0$ of $p_0$ and $U_1$ of $p_1$ such that
  $\Phi^{\nabla f}_{T'}$ is a diffeomorphism between $U_0$ and $U_1$.
  After possibly shrinking the size of $U_1$, we may assume that $U_1$
  lies in $U$.  Replace $U_0$ by $U_0 := \Phi^{\nabla f}_{-T'}(U_1)$,
  then by construction $U_0$ is a neighborhood of $p_0$ that lies in
  $W^s(f;\p_l^+ W)$ proving that $W^s(f;\p_l^+ W)$ is indeed an open
  subset.

  \medskip

  \textbf{(b)} In order to avoid some of the technicalities arising
  along the boundary, cap-off $W$ using \Cref{doubling Morse-Bott} to
  obtain a closed manifold~$W^{\mathrm{cap}}$ and a Morse-Bott
  function~$f^{\mathrm{cap}}$ that extends $f$ to all of
  $W^{\mathrm{cap}}$.  We can also extend $g$ to a metric
  $g^{\mathrm{cap}}$ on all of $W^{\mathrm{cap}}$.  Denote the cap-off
  of $C_j$ in $\Crit( f^{\mathrm{cap}})$ by $C_j^{\mathrm{cap}}$.
  Then it follows that $W^s( f^{\mathrm{cap}}; C_j^{\mathrm{cap}})$ is
  a smooth submanifold of dimension $\dim C_j + i^-(C_j)$, see
  \Cref{stable subset smooth on closed manifold}.  Since the
  intersection of any submanifold with an open set is still a
  submanifold, it follows that $\Interior W^s(f; C_j)$ is contained in
  the submanifold
  $W^s( f^{\mathrm{cap}}; C_j^{\mathrm{cap}}) \cap \Interior W$.
  
  \medskip
  
  \textbf{(c)} Let $C_j \subset \Crit(f)$ be a component of $\Crit(f)$
  that does not intersect $\p W$.  Capping-off $(W,g)$ and $f$ to
  $(W^{\mathrm{cap}}, g^{\mathrm{cap}})$ and $f^{\mathrm{cap}}$ as in
  (b), we obtain that $C_j^{\mathrm{cap}} = C_j$ and the stable subset
  $W^s(f^{\mathrm{cap}};C_j)$ in $W^{\mathrm{cap}}$ is a smooth
  submanifold of dimension $\dim C_j + i^-(C_j)$.

  This reproves that $\Interior W^s(f;C_j)$ is contained in a smooth
  submanifold, but to show that $W^s(f;C_j)$ itself is a smooth
  submanifold, we argue as follows: First we will show that
  $W^s( f^{\mathrm{cap}};C_j)$ and $W^s(f;C_j)$ agree on a
  sufficiently small neighborhood of $C_j$.  For this choose small
  open neighborhoods~$U_0$ and $U_1$ of $C_j$ as in \Cref{box around
    critical component}.(b) such that $U_0\subset U_1$ lie in
  $\Interior W$ and such that every gradient trajectory of
  $\nabla f^{\mathrm{cap}}$ that passes through $U_0$ and later
  escapes from $U_1$ can never again return to $U_0$.  It follows that
  every point~$q\in U_0$ lying in $W^s( f^{\mathrm{cap}};C_j)$,
  automatically also lies in $W^s(f;C_j)$ showing that
  $W^s( f^{\mathrm{cap}};C_j)\cap U_0 \subset W^s(f;C_j)$.

  If $p$ is now any point in $W^s(f;C_j)$, then we can choose $T> 0$
  such that $\gamma(T) = q \in W^s(f;C_j) \cap U_0$.  If
  $p \notin \p W$, then $\gamma(t)$ does not touch for $t\in [0,T]$
  the boundary of $W$ anywhere due to $\nabla$-convexity.  By the
  remark at the beginning of this proof, there are open
  neighborhoods~$U_p$ and $U_q$ of $p$ and $q$ respectively such that
  $\Phi_T^{\nabla f}$ is a diffeomorphism between $U_p$ and $U_q$.  We
  can shrink $U_q$ so that it lies in $U_0$, and since
  $W^s( f^{\mathrm{cap}};C_j)\cap U_q = W^s(f;C_j)\cap U_q$, it
  follows that the later is a smooth submanifold.

  We obtain from the invariance of $W^s(f;C_j)$ under the gradient
  flow that
  $W^s(f;C_j) \cap U_p = \Phi_{-T}^{\nabla f}\bigl(W^s(f;C_j)\cap
  U_q\bigr)$ so that $\Interior W^s(f;C_j)$ is globally a smooth
  submanifold.

  \smallskip
  
  If $p \in W^s(f;C_j)$ does lie in $\p W$, then note that
  $\nabla f(p)$ necessarily points by $\nabla$-convexity transversely
  into $W$.  This proves that $W^s( f^{\mathrm{cap}};C_j)$ intersects
  $\p W$ transversely at $p$, and
  $W^s( f^{\mathrm{cap}};C_j)\cap \p W$ is then in a neighborhood of
  $p$ a smooth hypersurface of $W^s( f^{\mathrm{cap}};C_j)$ that is
  transverse to $\nabla f$.  Let $\gamma$ be the gradient trajectory
  with $\gamma(0) = p$, and assume that
  $\gamma(T) = q \in W^s(f;C_j) \cap U_0$.  Note that by
  $\nabla$-convexity, $\gamma(t)$ does not intersect $\p W$ for any
  $t\in (0,T]$.  As above, we can choose neighborhoods $U_p$ of $p$
  and $U_q$ of $q$ such that
  $W^s( f^{\mathrm{cap}};C_j) \cap U_p = \Phi_{-T}^{\nabla
    f^{\mathrm{cap}}}\bigl(W^s(f;C_j)\cap U_q\bigr)$.

  The subset $W^s( f^{\mathrm{cap}};C_j) \cap U_p$ is split by the
  hypersurface $W^s( f^{\mathrm{cap}};C_j)\cap \p W$ into the part
  that lies in $W$ and the part that lies in the complement of $W$.
  The gradient trajectories of the points in the first half cannot
  touch $\p W$ by $\nabla$-convexity from the interior, and thus it
  follows that they all lie in $W^s(f;C_j)$.
  
  This shows that $W^s(f;C_j)$ is a smooth submanifold with boundary
  such that $\p W^s(f;C_j) \subset \p W$, and such that $W^s(f;C_j)$
  intersects $\p W$ transversely.
  
  \smallskip
  
  To prove the statements about the unstable sets, it suffices to
  invert the sign of $f$.
\end{proof}

Above we have shown that the stable set of a component of $\Crit(f)$
intersecting $\p W$ is always contained in a submanifold whose
dimension is given by the usual index formula.  In general though, it
does not need to be a submanifold itself.  The following lemma
provides a technical condition that solves this problem.

\begin{lemma}\label{lemma for stable sets with boundary}
  Let $(W', g')$ be a closed Riemannian manifold, and let
  $f'\colon W' \to \RR$ be a Morse-Bott function.  Suppose that $W$ is
  a compact subdomain with $\nabla$-convex boundary.

  Consider a component~$C'_j$ of $\Crit( f')$ that intersects $\p W$
  transversely, and denote the restriction of $ f'$ to $W$ by $f$, and
  the intersection $C'_j\cap W \subset \Crit(f)$ by $C_j$.

  \begin{itemize}
  \item If $C_j \cap \p W$ admits a neighborhood~$U^-$ in $\p W$ such
    that $\nabla f$ does not point anywhere along
    $W^s\bigl(\hat f; C'_j\bigr)\cap U^-$ transversely \emph{out of
      $W$}, see \Cref{fig: stable subset is manifold if boundary
      suitable}, then it follows that the interior
    $\Interior W^s(f; C_j)$ of the stable subset of $C_j$ in $W$ is a
    smooth submanifold of dimension $\dim C_j + i^-(C_j)$.
  \item Similarly, if $C'_j \cap \p W$ admits a neighborhood~$U^+$ in
    $\p W$ such that $\nabla f$ does not point anywhere along
    $W^u\bigl( f' ; C'_j\bigr)\cap U^+$ transversely \emph{into $W$},
    then it follows that the interior $\Interior W^u(f; C_j)$ of the
    unstable subset of $C_j$ in $W$ is a smooth submanifold of
    dimension $\dim C_j + i^+(C_j)$.
  \end{itemize}
\end{lemma}

\begin{figure}[htbp]
  \begin{center}
    \includegraphics[height=3.5cm,keepaspectratio]{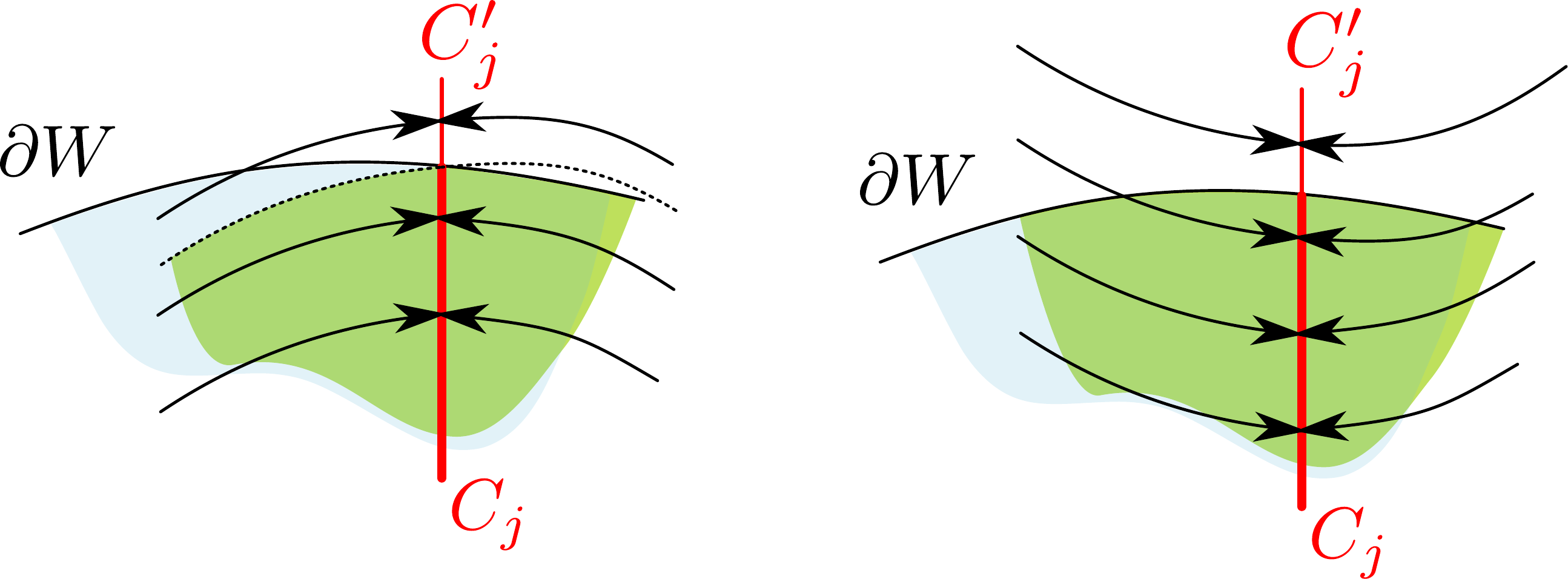}
  \end{center}  
  \caption{In the two pictures above, we have sketched in green color
    the stable set $\Interior W^s(f;C_j)$ of a
    component~$C_j \subset \Crit(f)$ that intersects $\p W$
    transversely.  According to \Cref{stable sets are
      submanifolds}.(b), $\Interior W^s(f;C_j)$ is always
    \emph{contained in} a submanifold, but in the left picture, we see
    that the stable set might be cut-off in an unfortunate way, while
    the picture on the right illustrates how the boundary condition
    stated in \Cref{lemma for stable sets with boundary} avoids this
    problem. \label{fig: stable subset is manifold if boundary
      suitable}}
\end{figure}
\begin{proof}
  We will only consider stable sets; to prove the statements about the
  unstable ones, simply invert the sign of $ f'$.  The stable set
  $W^s\bigl(f'; C'_j\bigr)$ in $W'$ is according to \Cref{stable
    subset smooth on closed manifold} a submanifold of dimension
  $\dim C'_j + i^-(C'_j)$.  We can choose arbitrarily small open
  neighborhoods~$U_0\subset U_1$ of $C'_j$ as in \Cref{box around
    critical component}.(b) such that
  \begin{itemize}
  \item a gradient trajectory of $\nabla  f'$ that passes through
    $U_0$ and later escapes from $U_1$ can never again return to
    $U_0$;
  \item the only critical points lying in $U_1$ are the ones in
    $C'_j$.
  \end{itemize}
  We can furthermore suppose that $U_1$ is so small that
  \begin{itemize}
  \item the intersection of $U_1$ with $\p W$ is contained in $U^-$.
  \end{itemize}

  We will first show that
  $\Interior W^s(f;C_j)\cap U_0 = W^s( f'; C'_j) \cap U_0$ so that
  $\Interior W^s(f;C_j)\cap U_0$ is a smooth submanifold.  The
  inclusion $W^s(f;C_j) \subset W^s( f'; C'_j)$ is obvious, so assume
  that $p$ is any point in $\Interior W^s( f'; C'_j) \cap U_0$.  By
  our assumption, the $\nabla f'$-trajectory through $p$ cannot escape
  from $U_1$, because otherwise it could never return to $U_0$ again,
  and $p$ would certainly not lie in the stable set of $C'_j$.  Since
  $U^-$ does not allow the trajectory to leave $U_1\cap \Interior W$
  either, we see that $p \in W^s(f; C_j)$ as desired.

  If $p$ is now any point in $\Interior W^s(f;C_j)$ and if
  $\gamma\subset W$ is its gradient trajectory, then it follows from
  \Cref{convergence gradient trajectory} that $\gamma$ converges to a
  point on $C_j$.  There is thus a time~$T > 0$ such that $\gamma(T)$
  lies for every $t \ge T$ in the neighborhood~$U_0$.  Furthermore, by
  the convexity of the boundary, $\gamma$ does not intersect $\p W$ on
  the interval $[0,T]$.

  There is a small neighborhood~$U_p$ of $p$ and a small
  neighborhood~$U_q$ of $q = \gamma(T)$ such that
  $U_q \subset U_0\cap \Interior W$ and such that $\Phi_T^{\nabla f}$
  restricts to a diffeomorphism between $U_p$ and $U_q$.  We have
  shown above that
  $\Interior W^s(f;C_j)\cap U_0 = \Interior W^s( f'; C'_j) \cap U_0$
  so that also $W^s(f;C_j)\cap U_q = W^s( f'; C'_j) \cap U_q$.  We
  obtain that
  $W^s(f;C_j)\cap U_p = \Phi_{-T}^{\nabla f}\bigl(W^s(f;C_j)\cap
  U_q\bigr)$ is a smooth submanifold.
\end{proof}

The previous lemma is in general not very practical for direct
applications, because to verify its conditions one would need to
compute the intersection of the boundary with the stable subsets.  The
situation simplifies significantly, if the gradient field does not
point \emph{anywhere} in the neighborhood of $\p C_j$ transversely out
of $W$.  For the main result of this section, \Cref{topology simple if
  no codim 1 stable manifolds}, it is even only necessary to show that
the stable subsets of local maxima and the unstable subsets of local
minima are open.

\begin{corollary}\label{stable set of max min open}
  Let $(W, g)$ be a compact Riemannian manifold with boundary, and let
  $f\colon W\to \RR$ be a Morse-Bott function such that $\p W$ is
  $\nabla$-convex.

  If $C_j$ is a component of $\Crit(f)$ that is a local maximum (a
  local minimum), and if either
  \begin{itemize}
  \item $\p C_j = \emptyset$, or
  \item if $\p C_j$ has a neighborhood~$U_j$ in $\p W$ such that
    $\nabla f$ does not point anywhere along $U_j$ transversely out of
    (into) $W$,
  \end{itemize}
  then it follows that $\Interior W^s(f;C_j)$ is an open subset of
  $\Interior W$ ($\Interior W^u(f;C_j)$ is an open subset of
  $\Interior W$).
\end{corollary}
\begin{proof}
  If $C_j$ does not intersect $\p W$, then it follows from
  \Cref{stable sets are submanifolds}.(c) that $\Interior W^s(f;C_j)$
  and $\Interior W^u(f;C_j)$ respectively are full dimensional
  submanifolds and thus open subsets.  Otherwise if
  $C_j\cap \p W\ne \emptyset$, take first the cap-off of $W$ and $f$
  as in \Cref{doubling Morse-Bott}, and then apply \Cref{lemma for
    stable sets with boundary} so that we obtain again the desired
  claim.
\end{proof}

\medskip

The assumptions about the intersection between local extrema and the
boundary of the manifold made in \Cref{stable set of max min open} and
in \Cref{topology simple if no codim 1 stable manifolds} below may
seem quite artificial, but they are automatically satisfied by
Hamiltonian circle actions so that the results developed in this
section apply to the symplectic manifolds we are interested in.

\begin{theorem}\label{topology simple if no codim 1 stable manifolds}
  Let $(W,g)$ be a connected compact Riemannian manifold with
  boundary, and let $f\colon W \to \RR$ be a Morse-Bott function such
  that $\p W$ is $\nabla$-convex.
  
  We assume that none of the components~$C_j \subset\Crit(f)$ has
  index $i^-(C_j) = 1$ ($i^+(C_j) = 1$).  If $C_j$ is a local maximum
  (local minimum), then we additionally suppose that $C_j$ is either
  closed, or that its boundary $\p C_j = C_j \cap \p W$ has a
  neighborhood $U_j$ in $\p W$ such that $\nabla f$ does not point
  anywhere along $U_j$ transversely out of $W$ (transversely into
  $W$).

  \smallskip

  Then we are in one of the following two situations:
  \begin{itemize}
  \item The set of critical points~$\Crit(f)$ has a unique
    component~$C_{\max}$ that is a local maximum (a unique
    component~$C_{\min}$ that is a local minimum), and $f$ is
    everywhere else on $W\setminus C_{\max}$ strictly smaller than on
    $C_{\max}$ (everywhere else on $W\setminus C_{\min}$ strictly
    larger than on $C_{\min}$), so that this local maximum (local
    minimum) is actually the global one.  The gradient
    field~$\nabla f$ does not point anywhere along $\p W$ transversely
    out of $W$ (into $W$).  Finally, every loop in $W$ can be
    homotoped to one that lies in $C_{\max}$ (that lies in
    $C_{\min}$).
  \item The subset~$\p^+ W$ ($\p^- W$) is non-empty and connected, and
    every loop in $W$ can be homotoped to one that lies in $\p^+ W$
    ($\p^- W$).  None of the components of $\Crit(f)$ is a local
    maximum (local minimum), and $f$ takes its global maximum (global
    minimum) on the boundary of $W$.
  \end{itemize}
\end{theorem}
\begin{proof}
  Let $W_0$ be the union of all the stable sets of boundary components
  of $W$ and of all the stable sets of components
  $C_j \subset \Crit (f)$ that are local maxima.  According to
  \Cref{stable sets partition}, the stable subsets partition $W$ so
  that the complement of $W_0$ is the union of all
  $\Interior W^s(f;C_k)$ for which $C_k\subset \Crit(f)$ is not a
  local maximum.

  With our assumptions it follows from \Cref{stable set of max min
    open} that $\Interior W^s(f; C_j)$ is open for every
  component~$C_j$ that is composed of local maxima.  By the convexity
  assumption, the flow line of a point $p\in \Interior W$ intersecting
  the boundary of $W$ will do so transversely so that
  $\Interior W^s(f;\p_k W)$ is the disjoint union of all
  $\Interior W^s(f;\p_l^+ W)$ with $\p_l^+W\subset \p_k W$, and
  furthermore each of the stable subsets~$\Interior W^s(f;\p_l^+ W)$
  is by \Cref{stable sets are submanifolds}.(a) open.

  Thus we see that $\Interior W_0$ decomposes into subsets that are
  open and pairwise disjoint.  We will show that $\Interior W_0$ is
  path-connected, so that all but one of the stable subsets composing
  $\Interior W_0$ have to be empty.

  \smallskip
  
  Let $p$ and $p'$ be any two points in the interior of $W_0$, and
  join them by a path~$\gamma$ in $\Interior W$.  We can perturb
  $\gamma$ in such a way that it will lie in $\Interior W_0$: if
  $C_k \subset \Crit(f)$ is a \emph{closed} component, recall that by
  \Cref{stable sets are submanifolds}.(c) the interior of its stable
  subset $W^s(f;C_k)$ is a smooth submanifold of dimension
  $\dim C_k + i^-(C_k)$; if $C_k$ is a component of $\Crit(f)$ with
  $\p C_k \ne \emptyset$, then its stable set $\Interior W^s(f;C_k)$
  does not need to be a smooth submanifold, but by \Cref{stable sets
    are submanifolds}.(b) it is nonetheless \emph{contained} in a
  smooth submanifold of dimension $\dim C_k + i^-(C_k)$.  It follows
  that the complement of $\Interior W_0$ in $\Interior W$ lies in the
  finite union of smooth submanifolds that are each at least of
  codimension~$2$.
  
  A generic perturbation of $\gamma$ will be transverse to all of
  these submanifolds (see for example
  \cite[Theorem~\S~3.2.5]{HirschDiffTopology}; it is not required that
  the submanifolds are closed).  By the dimension formula, the
  intersection between $\gamma$ and any of the stable subsets
  $W^s(f;C_k)$ in the complement of $W_0$ would be at most of
  dimension~$-1$, that is, the intersection has to be empty.  This
  shows that the perturbed path~$\gamma$ is contained in
  $\Interior W_0$ so that $\Interior W_0$ is indeed path-connected.
  
  We deduce that $\Interior W_0$ is only composed of a single stable
  set, either $\Interior W^s(f; \p_l^+ W)$ or $\Interior W^s(f; C_j)$
  for $C_j$ a local maximum --- all other potential stable sets
  composing $\Interior W_0$ need to be empty.  This proves that either
  there is a unique component of $\Crit(f)$ that is a local maximum or
  that $\partial^+W$ is non-empty and connected.

  \medskip

  In order to prove that every loop~$\gamma \subset W$ can be
  homotoped either to a loop in $\p^+ W$ or to a loop in $C_{\max}$,
  apply the same transversality argument as above to ensure that
  $\gamma$ lies in $\Interior W_0$.  Since $\Interior W_0$ is composed
  of a single stable set, it retracts via the gradient flow either to
  $\p^+W$ or to an arbitrarily small neighborhood of $C_{\max}$ thus
  proving the desired claim.

  \medskip
  
  Finally note that since $W$ is compact $f$ takes somewhere a
  maximum, either in the interior of $W$ or on its boundary.  If there
  is no component~$C_j$ in $\Crit(f)$ that is a local maximum, then it
  is clear that the maximum of $f$ has to lie on $\p W$.

  If there is on the other hand a component~$C_{\max}$, then $\p^+ W$
  is empty by what we have just shown.  The global maximum of $f$
  cannot lie at a point of $\p^-W$, because the gradient points along
  $\p^-W$ transversely into $W$ so that the function is necessarily
  increasing in inward direction.  Assume thus that the maximum lies
  at a point $p$ on $\p^0W$.  If $p$ is a regular point of $f$, then
  the level set of $p$ is locally a regular hypersurface that is
  transverse to $\p W$.  This implies that we find close to $p$ a
  point~$p'$ in $\Interior W$ lying on the same level set as $p$ such
  that the level set is also regular in $p'$.  Following the gradient
  flow from $p'$ the function~$f$ increases further so that the global
  maximum cannot lie at $p'$ and thus also not at $p$.

  We deduce that if the global maximum lies on $\p W$ it necessarily
  needs to lie at a point where $\nabla f$ vanishes, that is, it lies
  on a critical point of $f$.  With our definition of Morse-Bott
  function, we know that every component of $\Crit(f)$ is transverse
  to $\p W$ and thus it follows that $f$ takes its maximum on a
  component $C_j \subset \Crit(f)$.  This component needs to be a
  local maximum so that $C_j = C_{\max}$ as desired.
  
  The statements in parenthesis are easily deduced from the original
  ones by changing the sign of $f$.
\end{proof}

\subsection{Closed $1$-forms of Morse-Bott type}
\label{sec: local to global}

Consider a connected Riemannian manifold~$(W,g)$ with possibly
non-empty boundary, and let $\eta$ be a closed $1$-form on $W$.  We
can define a vector field~$X_\eta$ by the equation
\begin{equation*}
  g(X_\eta,\cdot) = \eta \;.
\end{equation*}
Locally, every closed $1$-form is exact so that $X_\eta$ is locally a
gradient vector field.

In this section we give a sufficient criterion for $\eta$ to be
globally exact that we will then be used in Section~\ref{sec:
  Hamiltonian circle action} to prove that every compact symplectic
$\SS^1$-manifold with contact type boundary is Hamiltonian
(\Cref{theorem: boundary+symplectic is Hamiltonian}).  The initial
motivation for this theorem was a result due to McDuff
\cite{McDuffFixedPointsHamiltonianGroupsActions} stating that a closed
symplectic $\SS^1$-manifold is Hamiltonian if the $1$-form obtained as
the contraction of the symplectic form with the infinitesimal
generator of the circle action has a local maximum or minimum.  We
also reprove this result.

\medskip

We will always assume in this section that $\eta$ is of
\defin{Morse-Bott type}, that is, the components of $\Crit(\eta)$ are
smooth submanifolds, and the local primitives of $\eta$ around the
critical points are of Morse-Bott type.  If $W$ has non-empty
boundary, then we assume additionally that the components of
$\Crit(\eta)$ intersect $\p W$ transversely.

The indices of a critical point of $\eta$ are simply the Morse-Bott
indices of the corresponding local primitive and we say that a
critical point of $\eta$ is a local maximum or local minimum if it is
a local maximum or minimum of the local primitive.

The main result we want to prove in this section is the following
theorem.

\begin{theorem}\label{theorem:exactness with boundary}
  Let $\eta$ be a closed $1$-form of Morse-Bott type on a compact
  connected Riemannian manifold~$(W,g)$.  We assume that
  \begin{itemize}
  \item none of the critical points of $\eta$ has indices $i^- = 1$ or
    $i^+ = 1$;
  \item $\p W$ is either empty or convex with respect to $X_\eta$;
  \item every critical point~$p$ that is a local maximum of $\eta$ and
    that lies in $\p W$ admits an open neighborhood~$U_p \subset \p W$
    in the boundary, such that $X_\eta$ does not point anywhere along
    $U_p$ out of $W$;
  \item every critical point~$p$ that is a local minimum of $\eta$ and
    that lies in $\p W$ admits an open neighborhood~$U_p \subset \p W$
    in the boundary, such that $X_\eta$ does not point anywhere along
    $U_p$ into $W$.
  \end{itemize}

  If at least one of the two following conditions hold, then $\eta$ is
  exact:
  \begin{itemize}
  \item [(a)] one of the components of $\Crit(\eta)$ is composed of
    local minima or local maxima;
  \item [(b)] there exists a boundary component~$\p_0 W$ of $W$ such
    that the restriction~$\restricted{\eta}{\p_0 W}$ is exact.
  \end{itemize}
\end{theorem}

To prove this theorem, we will lift $\eta$ to the smallest covering
space $\pi\colon \widetilde W\to W$ on which the lift of $\eta$ is
exact.  This way, we have a genuine Morse-Bott function on
$\widetilde W$ that satisfies all properties required in
\Cref{topology simple if no codim 1 stable manifolds} except of course
for the compactness of $\widetilde W$.  A uniqueness statement for
local extrema or for boundary components of $\widetilde W$ as in
\Cref{topology simple if no codim 1 stable manifolds} would force
$\widetilde W\to W$ to be a simple cover as otherwise every local
extremum or every boundary component in $W$ downstairs would lead to
several such components upstairs.  The main technical complication in
our proof is thus to deal with the non-compactness of $\widetilde W$.

\medskip

\textbf{Construction of the cover $\widetilde W\to W$:} Let $W$ be a
smooth compact connected manifold with possibly non-empty boundary,
and let $\eta$ be a closed $1$-form on $W$.  Analogously to the
construction of the universal cover, fix a point $x_0\in W$ and
consider the set of all piecewise smooth paths
$\Gamma = \bigl\{\gamma\colon [0,1]\to W\bigm|\; \gamma(0) =
x_0\bigr\}$.  We define an equivalence relation~''$\sim$'' on $\Gamma$
by saying $\gamma \sim \gamma'$ for $\gamma, \gamma' \in \Gamma$, if
and only if $\gamma(1) = \gamma'(1)$ and
$\int_\gamma \eta = \int_{\gamma'}\eta$.  Denote the space of
equivalence classes by $\widetilde W := \Gamma/\!\sim$, and note that
there is a natural surjective map $\pi\colon \widetilde W \to W$ given
by $\pi([\gamma]) = \gamma(1)$.

To construct a smooth structure on $\widetilde W$, let $U_x$ be any
open neighborhood of a point~$x\in W$ such that
$\restricted{\eta}{U_x}$ is exact.  We find for every
$\tilde x \in \widetilde W$ with $\pi(\tilde x) = x$, a natural lift
$\rho_{\tilde x} \colon U_x\to \widetilde W$ such that
$\rho_{\tilde x}(x) = \tilde x$ and
$\pi\circ \rho_{\tilde x} = \id_{U_x}$.  More precisely, if
$\gamma\in \Gamma$ is a path representing $\tilde x$, then choose for
any point~$x'\in U_x$, a path~$\psi$ in $U_x$ connecting $x$ to $x'$
and define $\rho_{\tilde x}(x')$ as the point in $\widetilde W$
represented by the concatenation $\gamma \cdot \psi$.  For any other
choice of path~$\psi' \subset U_x$ connecting $x$ to $x'$, we have
$\int_{\psi'} \eta = \int_\psi \eta$ so that
$\gamma \cdot \psi \sim \gamma \cdot \psi'$ which implies that
$\rho_{\tilde x}$ does not depend on the choice of $\psi$.  Similarly,
one can also see that the construction does not depend on the choice
of $\gamma$.

Furthermore if $\tilde x \ne \tilde x'$ are any two points lying in
the same fiber over $x$, then the images of $\rho_{\tilde x}$ and
$\rho_{\tilde x'}$ are disjoint, for otherwise, there would be a
point~$y\in U_x$ such that $\rho_{\tilde x}(y) = \rho_{\tilde x'}(y)$.
This way, if $\tilde x$ and $\tilde x'$ are represented by
paths~$\gamma$ and $\gamma'$ respectively, we obtain the points
$\rho_{\tilde x}(y)$ and $\rho_{\tilde x'}(y)$ by attaching to
$\gamma$ and $\gamma'$ respectively a path~$\psi\subset U_x$ from $x$
to $y$.  Since the integrals of $\eta$ over $\gamma\cdot \psi$ and
$\gamma'\cdot \psi$ agree by definition, the integrals over $\gamma$
and $\gamma'$ also have to agree, implying that
$\tilde x = \tilde x'$.

With the help of these maps, we can lift any local structure from $W$
to $\widetilde W$.  In particular, we can equip $\widetilde W$ with a
topology such that $\pi\colon \widetilde W \to W$ is a covering space,
and $\widetilde W$ carries a unique smooth structure coming from the
base.  We will always assume that $\widetilde W$ is equipped with the
pull-back metric $\tilde g = \pi^* g$.

Furthermore, $\widetilde W$ is path-connected: Denote by
$\tilde x_0\in\widetilde W$ the class of the constant path at $x_0$,
and let $\tilde x$ be a point in $\widetilde W$ that is represented by
a path~$\gamma$.  Consider the family of paths~$\gamma_s$ for
$s\in[0,1]$, given by
\begin{equation*}
  \gamma_s(t) := \begin{cases}
    \gamma(t),&  t\in[0,s],\\
    \gamma(s), & t\in[s,1].
  \end{cases}
\end{equation*}
For any fixed $s\in[0,1]$, $[\gamma_s]$ represents a point in
$\widetilde W$, and the map $s\mapsto [\gamma_s]$ is a path in
$\widetilde W$ that connects $\tilde x_0$ to $\tilde x = [\gamma]$.

\begin{lemma}\label{lemma lift of exact neighborhood}
  Let $U\subset W$ be a connected open set on which $\eta$ is exact.
  Then, it follows that $\pi^{-1}(U)\to U$ is a trivial cover, that
  is, $\pi^{-1}(U)$ is for every choice of $x\in U$ naturally
  diffeomorphic to the disjoint union
  $\sqcup_{\tilde x \in \pi^{-1}(x)} \{\tilde x\}\times U$.
\end{lemma}
\begin{proof}
  As explained above, we find for every $\tilde x\in \widetilde W$
  with $\pi(\tilde x) = x$ a natural lift~$\rho_{\tilde x}$.  The
  images of two lifts~$\rho_{\tilde x}$ and $\rho_{\tilde x'}$ for two
  different points $\pi(\tilde x) = \pi(\tilde x') = x$ are disjoint.

  This way we obtain
  $\pi^{-1}(U) = \sqcup_{\tilde x \in \pi^{-1}(x)} \rho_{\tilde
    x}(U)$.  Since the $\rho_{\tilde x}$ were used to lift the smooth
  structure from $W$ to $\widetilde W$, we obtain by definition that
  each $\rho_{\tilde x}$ is a diffeomorphism showing the desired
  statement.
\end{proof}

\medskip

\begin{proposition}\label{proposition:exact}
  The lift of $\eta$ to the covering~$\widetilde W$ is exact.
  Furthermore, if $\tilde f\colon \widetilde W\to \RR$ is a primitive
  of $\pi^*\eta$, it follows that two points $\tilde x, \tilde x'$
  lying in the same fiber of $\pi$ are equal if and only if
  $\tilde f(\tilde x) = \tilde f(\tilde x')$.
\end{proposition}
\begin{proof}
  Define a function $\tilde f\colon \widetilde W \to \RR$ by
  $\tilde f(\tilde x) := \int_\gamma \eta$, where $\gamma\in \Gamma$
  is any path representing the point $\tilde x \in \widetilde W$.  By
  our construction, $\tilde f$ is well-defined, and for every $x\in W$
  and every contractible open neighborhood~$U_x$ of $x$, we easily
  recognize that $\pi^*\eta = d\tilde f$ over $\pi^{-1}(U_x)$.  This
  proves that $\tilde f$ is a primitive of $\pi^*\eta$.

  \smallskip

  Let now $\tilde x$ and $\tilde x'$ be two points in $\widetilde W$
  lying in the same fiber over a point $x \in W$ such that
  $\tilde f(\tilde x) = \tilde f(\tilde x')$.  If $\gamma\in \Gamma$
  is a path representing $\tilde x$ and $\gamma'\in \Gamma$ is a path
  representing $\tilde x'$, then $x = \gamma(1) = \gamma'(1)$.  By our
  assumption $\tilde f(\tilde x) = \int_\gamma \eta$ and
  $\tilde f(\tilde x') = \int_{\gamma'} \eta$ are equal.  But this
  means precisely that $\gamma \sim \gamma'$ represent the same point
  in $\widetilde W$, that is, $\tilde x = \tilde x'$ as we wanted to
  show.
\end{proof}

\medskip

\textbf{Boundary convexity on $\widetilde W$:} Let $(W, g)$ be a
Riemannian manifold with possible non-empty compact boundary and let
$\eta$ be a closed $1$-form on $W$ with dual vector field~$X_\eta$
such that $g(X_\eta, \cdot) = \eta$.  Lift $\eta$ to the covering
$(\widetilde W, \tilde g)$ constructed above.  It follows by
\Cref{proposition:exact} that $\tilde \eta = \pi^*\eta$ admits a
primitive function $\tilde f\colon \widetilde W \to \RR$.  The
gradient vector field~$\nabla\tilde f$ with respect to the pull-back
metric~$\tilde g$ is then related to $X_\eta$ by
\begin{equation}\label{equation: gradient upstairs downstairs}
  D\pi \cdot \nabla \tilde f = X_\eta\circ \pi \;.
\end{equation}
For the corresponding flows we find
\begin{equation}\label{equation: flows commute with projection}
  \pi\circ\Phi_t^{\nabla \tilde f} = \Phi_t^{X_\eta}\circ \pi \;.
\end{equation}

\begin{lemma}\label{convexity upstairs}
  The boundary of $\widetilde W$ is
  $\p \widetilde W = \pi^{-1}(\p W)$, and the decomposition
  $\p W = \p^+ W \cup \p^- W \cup \p^0 W$ with respect to $X_\eta$
  lifts to the decomposition
  $\p \widetilde W = \p^+ \widetilde W \cup \p^- \widetilde W \cup
  \p^0 \widetilde W$ with respect to $\nabla \tilde f$, that is,
  \begin{equation*}
    \p^+\widetilde W = \pi^{-1}(\p^+W) \;, \quad
    \p^-\widetilde W = \pi^{-1}(\p^-W) \;,  \text{ and }
    \p^0\widetilde W = \pi^{-1}(\p^0W) \;.
  \end{equation*}
  Moreover, $\p W$ is (strongly) convex with respect to $X_\eta$, if
  and only if $\p\widetilde W$ is (strongly) convex with respect to
  $\nabla\tilde f$.
\end{lemma}
\begin{proof} 
  The claims in the lemma are all about local properties, and these
  are preserved because the covering is locally diffeomorphic to the
  base manifold.
\end{proof}

\bigskip

We are now ready to prove the main result of this section:

\begin{proof}[Proof of \Cref{theorem:exactness with boundary}]
  Let $\tilde f$ be a primitive of $\tilde\eta$ upstairs.  To show
  that $\eta$ is exact, we will prove that the covering
  $\pi\colon \widetilde W\to W$ is simple so that $\pi$ is a
  diffeomorphism.  Then $\tilde f\circ \pi^{-1}$ is a primitive of
  $\eta$ downstairs.

  \smallskip
  
  All local properties like being Morse-Bott etc.\ lift directly to
  $\widetilde W$.  It thus follows that $\tilde f$ is a Morse-Bott
  function with $\Crit(\tilde f) = \pi^{-1}\bigl(\Crit(\eta)\bigr)$,
  and none of the indices~$i^+$ or $i^-$ of a point in
  $\Crit(\tilde f)$ is equal to $1$.  As explained in \Cref{convexity
    upstairs}, the boundary $\p \widetilde W$ is convex with respect
  to $\nabla \tilde f$, and every critical point
  $\tilde p \in \Crit(\tilde f) \cap \p \widetilde W$ that is a local
  maximum admits a neighborhood~$U_{\tilde p}$ in $\p \widetilde W$
  such that $\nabla \tilde f$ does not point anywhere along
  $U_{\tilde p}$ transversely out of $\widetilde W$; every point
  $\tilde p \in \Crit(\tilde f) \cap \p \widetilde W$ that is a local
  minimum of $\tilde f$ admits a neighborhood~$U_{\tilde p}$ in
  $\p \widetilde W$ such that $\nabla \tilde f$ does not point
  anywhere along $U_{\tilde p}$ transversely into $\widetilde W$.

  \smallskip

  We will choose now values $a,b \in \RR$ with $a< b$ to work with a
  suitable subdomain
  $\widetilde W_{[a,b]} = \tilde f^{-1}\bigl([a,b]\bigr)$.  The choice
  of these $a$ and $b$ depend on whether we are in case~(a) or (b)
  below.

  \smallskip
  
  \textbf{(a)} Let $C\subset \Crit(\eta)$ be a component of critical
  points of $\eta$ consisting of local maxima.  Clearly
  $\restricted{\eta}{C}$ is exact, because $\eta$ vanishes along $C$.
  The restriction of $\eta$ to a tubular neighborhood~$U$ retracting
  to $C$ is then also exact.  If we assume that the covering map~$\pi$
  is not injective, then there are by \Cref{lemma lift of exact
    neighborhood} at least two components $\widetilde C_1$ and
  $\widetilde C_2$ of $\Crit(\tilde f)$ that project diffeomorphically
  onto $C$.  Furthermore, $\tilde f$ is constant on each of these
  components, but
  $\tilde f\bigl(\widetilde C_1\bigr) \ne \tilde f\bigl(\widetilde
  C_2\bigr)$.

  Since $\widetilde W$ is path connected, there is a
  path~$\tilde \gamma$ joining $\widetilde C_1$ to $\widetilde C_2$.
  We can choose the interval $[a,b]$ so large that the compact subset
  $\widetilde C_1 \cup \widetilde C_2 \cup \tilde \gamma$ is contained
  in $\tilde f^{-1}\bigl([a,b]\bigr)$, and by Sard's theorem, we can
  furthermore perturb $a$ and $b$ to be regular values both of the
  function~$\tilde f$ and of $\restricted{\tilde f}{\p \widetilde W}$
  so that $\tilde f^{-1}([a,b])$ is a smooth manifold with boundary
  and corners.  To simplify the notation we restrict to the connected
  component of $\tilde f^{-1}([a,b])$ that contains $\widetilde C_1$
  and $\widetilde C_2$ and denote it by $\widetilde W_{[a,b]}$.

  If there is a component of $\Crit(f)$ that is a local minimum, we
  could equally well apply all arguments in this proof by replacing
  $\eta$ first by $-\eta$.

  \smallskip

  \textbf{(b)} Let $\p_0 W\subset \p W$ be one of the connected
  components of the boundary on which $\eta$ is exact.  Assume that
  the covering map~$\pi$ is not injective.  Applying \Cref{lemma lift
    of exact neighborhood} to a collar neighborhood of $\p_0 W$, we
  can find at least two boundary components $\p_1\widetilde W$ and
  $\p_2 \widetilde W$ of $\p \widetilde W$ that project
  diffeomorphically onto $\p_0 W$.  In particular, it follows that
  $\p_1\widetilde W$ and $\p_2\widetilde W$ are compact.

  Furthermore, if follows from \Cref{convex boundary cannot be
    everywhere tangent to gradient} (the argument is local) that
  $\nabla \tilde f$ cannot be everywhere tangent to $\p_1\widetilde W$
  or to $\p_2\widetilde W$, and up to reversing the sign of $\eta$ if
  necessary, we may always assume that both boundary components have a
  point along which the gradient points outwards.
  
  Since $\widetilde W$ is connected, there is a path~$\tilde \gamma$
  joining $\p_1\widetilde W$ to $\p_2\widetilde W$.  Using that this
  path, that $\p_1\widetilde W$, and that $\p_2\widetilde W$ are
  compact and that $\tilde f$ is continuous, it follows that
  $\tilde f$ is bounded on
  $\p_1\widetilde W \cup \p_2\widetilde W \cup \tilde \gamma$, and in
  particular, we can find two real numbers~$a$ and $b$ with $a<b$ such
  that
  $\p_1\widetilde W \cup \p_2\widetilde W \cup \tilde \gamma \subset
  \tilde f^{-1}([a,b])$.  By slightly perturbing $a$ and $b$, we can
  again guarantee that $\tilde f^{-1}\bigl([a,b]\bigr)$ will be a
  smooth manifold with boundary and corners, and we denote the
  component of $\tilde f^{-1}\bigl([a,b]\bigr)$ containing
  $\p_1\widetilde W$ and $\p_2\widetilde W$ by $\widetilde W_{[a,b]}$.

  \medskip

  Note that in both cases, $\widetilde W_{[a,b]}$ does \emph{not} need
  to be compact.  However, even so, we show in \Cref{no gradient
    trajectory escapes to infinity} below that the situation in the
  cover is sufficiently tame so that none of the gradient trajectories
  of $\tilde f$ can escape to infinity.  This will be the key property
  that allows us to apply a strategy similar to the one used in the
  proof of \Cref{topology simple if no codim 1 stable manifolds} even
  though $\widetilde W_{[a,b]}$ may not be compact.

  \smallskip
  
  Denote $f^{-1}(a) \cap \widetilde W_{[a,b]}$ by $\widetilde W_a$ and
  $f^{-1}(b) \cap \widetilde W_{[a,b]}$ by $\widetilde W_b$.  Study
  now a gradient trajectory~$\tilde \gamma$ of $\widetilde f$ passing
  through an inner point of $\widetilde W_{[a,b]}$, and follow it for
  positive time inside $\widetilde W_{[a,b]}$ for as long as possible.
  By \Cref{no gradient trajectory escapes to infinity} below, the
  orbit~$\tilde \gamma$ is contained in a compact subset of
  $\widetilde W_{[a,b]}$ so that precisely one of the following
  statements will be true for $\tilde\gamma$:
  \begin{itemize}
  \item [(A)] $\tilde \gamma$ reaches in finite time $\widetilde W_b$,
    or some of the boundary components of
    $\p \widetilde W\cap \widetilde W_{[a,b]}$;
  \item [(B)] $\tilde \gamma(t)$ converges for $t\to \infty$ to a
    critical point of $\tilde f$ in $\widetilde W_{[a,b]}$ other than
    a local maximum;
  \item [(C)] $\tilde \gamma(t)$ converges for $t\to \infty$ to a
    local maximum of $\tilde f$ lying in $\widetilde W_{[a,b]}$.
  \end{itemize}

  We will partition the interior of $\widetilde W_{[a,b]}$ according
  the cases listed above, and study the properties of this
  decomposition.

  \smallskip
  
  For case~(A) note that the boundary of the
  domain~$\widetilde W_{[a,b]}$ is composed of
  $\p \widetilde W\cap \widetilde W_{[a,b]}$, of $\widetilde W_a$ and
  of $\widetilde W_b$.  If $\tilde \gamma$ intersects
  $\p \widetilde W_{[a,b]}$, it will necessarily do so transversely,
  because $\p \widetilde W$ is $\nabla$-convex, and because
  $\widetilde W_a$ and $\widetilde W_b$ are regular level sets of
  $\tilde f$.  Clearly, $\tilde \gamma$ cannot hit $\widetilde W_a$ or
  $\p^- \widetilde W$ in forward time, because $\nabla \tilde f$
  points along these boundary points into $\p \widetilde W_{[a,b]}$.
  This also excludes that $\tilde \gamma$ reaches any of the corners
  in $\p \widetilde W\cap \widetilde W_a$ or in
  $\p^- \widetilde W\cap \widetilde W_b$.  The remaining corners in
  $\p \widetilde W\cap \widetilde W_b$ are composed of
  $\p^+ \widetilde W\cap \widetilde W_b$, and
  $\p^0 \widetilde W\cap \widetilde W_b$.  Again by convexity,
  $\tilde \gamma$ cannot reach any point of $\p^0 \widetilde W$, and
  thus, $\tilde\gamma$ may only intersect corners given by
  $\p^+ \widetilde W\cap \widetilde W_b$ where $\nabla \tilde f$ is
  positively transverse both to $\p \widetilde W$ and to
  $\widetilde W_b$.

  Let $p \in \Interior\widetilde W_{[a,b]}$ be a point whose gradient
  trajectory~$\tilde \gamma$ satisfies (A).  It follows that all
  trajectories through nearby points also end up on
  $\p \widetilde W_{[a,b]}$: If $\tilde \gamma$ hits a point in
  $\p \widetilde W$ or in $\widetilde W_b$ that is not a corner point,
  then the claim can be easily deduced from a consideration as in the
  proof of \Cref{stable sets are submanifolds}.(a); to see that the
  claim is true if $\tilde \gamma$ ends up at a corner point
  $q \in \p^+ \widetilde W\cap \widetilde W_b$ consider
  $\widetilde W_{[a,b+\epsilon]}$, glue a collar to a neighborhood of
  $q$ in $\p \widetilde{W}$, and extend $\widetilde W_b$ and the
  gradient vector field to this collar.  Since $\nabla \tilde f$ hits
  both $\p^+ \widetilde W$ and $\widetilde W_b$ transversely, the
  gradient trajectory through any point close to $p$ will also hit
  both hypersurfaces transversely.  Depending on whether it reaches
  first $\p^+ \widetilde W$ or first $\widetilde W_b$ or both at the
  same time, the trajectory will exit from $\widetilde W_{[a,b]}$
  either through a point on $\p^+ \widetilde W$, through a point on
  $\widetilde W_b$ or through a corner point.  In any case, any
  trajectory passing through a point close to $p$ will hit the
  boundary of $\widetilde W_{[a,b]}$ in positive time.

  Thus the points in $\Interior \widetilde W_{[a,b]}$ satisfying (A)
  form an open subset.  Note that if not empty, $\p_1\widetilde W$ and
  $\p_2\widetilde W$ are disjoint and they are also isolated from the
  remaining points in $\p \widetilde W_{[a,b]}$.  In this case, we can
  further subpartition the points satisfying (A) into the open subsets
  of points whose trajectories end on $\p_1\widetilde W$, the points
  whose trajectories end on $\p_2\widetilde W$, and the ones ending on
  any of the remaining boundary points.  Since we are assuming that
  $\nabla \tilde f$ is somewhere positively transverse both to
  $\p_1\widetilde W$ and to $\p_2\widetilde W$, it follows that the
  stable subsets corresponding to the first two components are not
  empty.
    
  \smallskip

  A point satisfying (B) lies on the stable subset of a component of
  $\Crit(\tilde f) \cap \widetilde W_{[a,b]}$ that is not a local
  maximum.  Our aim here is to show that each such stable subset is
  contained in a smooth submanifold of dimension at least~$2$.  If
  $\widetilde W$ were compact and without corners, this would directly
  follow from \Cref{stable sets are submanifolds}.(b).  Nonetheless we
  will see below that such a statement also holds in our situation:
  Locally, we argue with the Hadamard-Perron Theorem
  \cite[Theorem~4.1]{HirshInvariantManifolds}.  If a component of
  $\Crit(\tilde f)$ is \emph{closed} then it admits a small
  neighborhood in which the local stable subset is a smooth
  submanifold with the desired properties.  By \Cref{box around
    critical component}.(b), there exist two neighborhoods of
  $\widetilde C_j$ such that every point of
  $W^s(\tilde f; \widetilde C_j)$ in the smaller one of the two
  neighborhoods lies on the local stable subset given by the
  Hadamard-Perron Theorem.  This shows that
  $W^s(\tilde f; \widetilde C_j)$ is close to $\widetilde C_j$ a
  smooth submanifold.

  On a global level, there are also no problems with the stable
  subsets if $\widetilde C_j$ is closed, because the boundary of
  $\widetilde W_{[a,b]}$ is $\nabla$-convex in the sense that no
  gradient trajectory can \emph{touch} the boundary from the interior.
  It follows that the gradient flow between two points lying in the
  interior of the manifold defines a diffeomorphism between their
  neighborhoods.  The structure of the stable subset agrees thus at
  every point of $\Interior\widetilde W_{[a,b]}$ with the manifold
  structure obtained locally around the critical points and thus
  $\Interior W^s(\tilde f; \widetilde C_j)$ is a submanifold of
  codimension at least~$2$.

  \smallskip

  If $\widetilde C_j$ is a component of $\Crit(\tilde f)$ that does
  intersect $\p\widetilde W$, then double first $\widetilde W_{[a,b]}$
  along the boundary components intersecting $\widetilde C_j$ using
  the method presented in the proof of \Cref{doubling Morse-Bott}.
  The potential problems due to the existence of corners are avoided
  by choosing a collar neighborhood that is tangent to
  $\widetilde W_a$ and $\widetilde W_b$.  Note that $\Crit(\tilde f)$
  does never intersect the boundaries~$\widetilde W_a$ and
  $\widetilde W_b$.  After this extension of $\widetilde W_{[a,b]}$,
  $\widetilde C_j$ embeds into a closed component.

  Furthermore note that because $\widetilde C_j$ is compact, we only
  need to double $\widetilde W_{[a,b]}$ along finitely many boundary
  components, and we can furthermore assume that any boundary
  component corresponding to $\p \widetilde W$ or to its doubling is
  $\nabla$-convex, and that $\tilde f$ is regular along
  $\widetilde W_a$, $\widetilde W_b$, and along their doubles.  We can
  apply the same steps as for closed components of $\Crit(\tilde f)$,
  and it follows that the stable subset inside the doubled domain is a
  submanifold.  When reducing back to $\widetilde W_{[a,b]}$, we
  retain then that $\Interior W^s(\tilde f; \widetilde C_j)$ is
  \emph{contained} is a submanifold of codimension at least~$2$.

  \smallskip
  
  One easily convinces oneself that the points corresponding to
  situation~(C) form an open subset.  The argument that the stable set
  of every component of $\Crit(\tilde f)$ that is a local maximum is
  contained in a full dimensional submanifold is identical to the
  proof given above for (B).  By \Cref{stable set of max min open} it
  follows that such stable subsets are really submanifolds.  Since
  none of the trajectories can touch the boundary from the inside, the
  presence of corners does not affect the validity of the claim.

  Note that if there are two connected components~$\widetilde C_1$ and
  $\widetilde C_2$ as described in case~(a), then we can further
  subpartition all points satisfying (C) into the subsets
  $W^s(\tilde f;\widetilde C_1)$ and $W^s(\tilde f;\widetilde C_2)$
  and the stable subsets of any remaining local maximum.  Each of the
  parts in the decomposition is open.

  \medskip

  Remove now all points that satisfy property~(B) from
  $\Interior\widetilde W_{[a,b]}$, and denote the complement by
  $\Interior \widetilde W_{[a,b]}^0$.  Since
  $\Interior \widetilde W_{[a,b]}^0$ is the union of all points that
  verify properties~(A) and (C), $\Interior \widetilde W_{[a,b]}^0$ is
  open, and we will show now that it is also path connected: Choose
  any two points~$p$ and $p'$ in $\Interior \widetilde W_{[a,b]}^0$
  and connect them inside the larger set
  $\Interior \widetilde W_{[a,b]}$ with a path~$\gamma$.  Note that
  $\Crit(\tilde f)$ may decompose in $\Interior \widetilde W_{[a,b]}$
  into infinitely many components.  Nonetheless $\gamma$ is compact
  and every gradient trajectory is contained by \Cref{no gradient
    trajectory escapes to infinity} in a ball of uniform radius.  This
  implies that there is a compact set~$K \subset \widetilde W_{[a,b]}$
  such that $\gamma$ can only encounter the stable subsets of the
  components of $\Crit(\tilde f)$ that intersect $K$.  In particular
  it follows from this that $\gamma$ only intersects a finite number
  of different stable subsets.

  Since every stable set $W^s(\tilde f;\widetilde C_k)$ corresponding
  to a component~$\widetilde C_k$ that is not a local maximum, is
  contained in a submanifold of codimension at least~$2$, we can
  assume after a perturbation that $\gamma$ does not intersect any of
  these stable sets.  This implies in particular that $\gamma$ lies in
  $\Interior \widetilde W_{[a,b]}^0$, proving as desired that
  $\Interior \widetilde W_{[a,b]}^0$ is path connected.
  
  Having shown that $\Interior \widetilde W_{[a,b]}^0$ is a path
  connected set that is partitioned by the open subsets listed above
  for (A) and (C), it follows that only one of these open subsets may
  be non-empty.  This leads to a contradiction both in situation~(a)
  where $\p_1\widetilde W$ and $\p_2\widetilde W$ and in
  situation~(b), where $\widetilde C_1$ and $\widetilde C_2$ are not
  empty.  As desired we obtain that the cover $\widetilde W \to W$ has
  to be simple.
\end{proof}


\begin{proposition}\label{no gradient trajectory escapes to infinity}
  Let $(W,g)$ be a compact manifold that might have boundary, and that
  carries a closed $1$-form~$\eta$ of Morse-Bott type.  Let
  $\pi \colon \widetilde W\to W$ be the minimal cover such that
  $\pi^*\eta$ is an exact $1$-form with primitive~$\tilde f$, and let
  $a < b$ be regular values both of $\tilde f$ and of
  $\restricted{\tilde f}{\p \widetilde W}$.
  
  Then there exists a constant~$R > 0$ such that every gradient
  trajectory of $\tilde f$ in
  $\widetilde W_{[a,b]} = \tilde f^{-1}([a,b])$ with respect to the
  pull-back metric~$\tilde g := \pi^*g$ lies in a ball of radius~$R$.
  Furthermore, every gradient trajectory is of finite length.
\end{proposition}
\begin{proof}
  Before considering a gradient trajectory, let us first study the
  manifolds~$W$ and $\widetilde W$ more in detail.  Since the base
  manifold~$W$ is compact, $\Crit(\eta)$ has only finitely many
  components~$C_1, \dotsc, C_N$.

  Clearly $\eta$ vanishes along $C_j$, so that
  $\restricted{\eta}{C_j}$ is trivially exact.  We can find for every
  $C_j$ a tubular neighborhood and a function~$f_j$ defined on this
  neighborhood such that $X_\eta$ is the gradient of $f_j$.  By
  \Cref{lemma lift of exact neighborhood}, it follows that the lift of
  $\pi^{-1}(C_j)$ is a disjoint union of components of
  $\Crit(\tilde f)$ and each of these components projects via $\pi$
  diffeomorphically onto $C_j$.  In particular, every component of
  $\Crit(\tilde f)$ in $\widetilde W$ upstairs is compact.

  More precisely, there exists for every component
  $C_j\subset \Crit(f)$ an $R_j>0$ such that any component
  $\widetilde{C}_j \subset \Crit(\tilde f)$ covering $C_j$ is
  contained in a ball of radius smaller than $R_j$.  To find a
  suitable radius for one such $\widetilde{C}_j$, it suffices to
  combine that $\widetilde{C}_j$ is compact with an exhaustion of
  $\widetilde W$ by open balls of increasing size.

  To show that every component~$\widetilde{C}_j'$ of $\Crit(\tilde f)$
  covering $C_j$ also fits into a ball of radius~$R_j$, let
  $\tilde p_0'$ and $\tilde p_1'$ be two points in $\widetilde{C}_j'$.
  Project $\tilde p_0'$ and $\tilde p_1'$ down to $C_j\subset W$ and
  then lift them to points~$\tilde p_0$ and $\tilde p_1$ in
  $\widetilde{C}_j$.  There is then a path~$\tilde \psi$ with
  $\tilde \psi(0) = \tilde p_0$ and $\tilde \psi(1) = \tilde p_1$ that
  is of length less than $R_j$.  Project now this path to $W$, and
  lift it to a new path $\tilde \psi'$ such that
  $\tilde \psi'(0) = \tilde p_0'$.  The lifted path has the same
  length as the initial one, and it just remains to convince oneself
  that its end point $\tilde \psi'(1)$ is $\tilde p_1'$.
    
  By \Cref{proposition:exact}, it suffices to verify that
  $\tilde f\bigl(\tilde \psi'(1)\bigr) = \tilde f(\tilde p_1')$.
  Since $\tilde f(\tilde p_0') = \tilde f(\tilde p_1')$, we can as
  well just check that
  $\tilde f\bigl(\tilde \psi'(1)\bigr) - \tilde f\bigl(\tilde
  \psi'(0)\bigr) = \int_{\tilde \psi'} d\tilde f$ vanishes.  This is
  true, because
  $\int_{\tilde \psi'} d\tilde f = \int_{\pi\circ \tilde \psi'} \eta =
  \int_{\tilde \psi} d\tilde f = \tilde f (\tilde p_1) - \tilde f
  (\tilde p_0) = 0$.

  Define $R_0 = \max\{R_1,\dotsc,R_N\}$.

  \medskip

  Choose for every component~$C_j$ neighborhoods
  $U_{j,0} \subset U_{j,1}$ as described in \Cref{function increases
    to leave box around critical component} that are sufficiently
  small so that the restriction of $\eta$ to $U_{j,1}$ is exact and
  such that no two $U_{i,1}$ and $U_{j,1}$ intersect for $i\ne j$.
  Using that the components of $\Crit(f)$ are compact, we can slightly
  enlarge $R_0$ (and shrink the size of $U_{j,1}$) so that every
  component of $\pi^{-1}\bigl(U_{j,1}\bigr)$ fits into a ball of
  size~$R_0$.

  We will now start working in $\widetilde W_{[a,b]}$.  Denote by
  $\widetilde U_0$ the union in $\widetilde W_{[a,b]}$ of the small
  neighborhoods $\pi^{-1}(U_{j,0})\cap \widetilde W_{[a,b]}$ for
  $j=1,\dotsc,N$, and by $\widetilde U_1$ the union of all larger
  neighborhoods $\pi^{-1}(U_{j,1})\cap \widetilde W_{[a,b]}$ for
  $j=1,\dotsc,N$.

  Let us now study a gradient trajectory~$\tilde\gamma$ in
  $\widetilde W_{[a,b]}$ for positive time (for negative time, apply
  the same reasoning to $-\tilde f$).  Let $I$ be the maximal interval
  on which $\tilde \gamma$ is defined, so that $I=[0,T_{\max})$ or
  $I=[0,T_{\max}]$.  We will cover $I$ by two subsets~$A$ and $B$
  where
  \begin{itemize}
  \item $A$ is the union of all intervals in
    $\tilde \gamma^{-1}\bigl(\widetilde U_1\bigr)$ that contain at
    least one point that is mapped into the smaller
    neighborhood~$\widetilde U_0$;
  \item $B$ is the union of all intervals in
    $\tilde \gamma^{-1}\bigl(\widetilde W_{[a,b]}\setminus \widetilde
    U_0\bigr)$ that each contain at least one point that is mapped
    into the complement of the larger neighborhood~$\widetilde U_1$.
  \end{itemize}

  Clearly, if $\tilde\gamma(t)$ lies in $\widetilde U_0$ then it
  follows that $t$ lies in $A$; if $\tilde\gamma(t)$ lies in the
  complement of $\widetilde U_1$, then $t$ lies in $B$.  The only
  remaining points are those for which $\tilde\gamma(t)$ lies in
  $\widetilde U_1\setminus \widetilde U_0$: if such a $t$ does neither
  lie in $A$ nor in $B$, it follows that we are in the particular case
  where $\tilde\gamma(I)$ lies entirely in one of the components of
  $\widetilde U_1\setminus \widetilde U_0$.  Since each such component
  lies in a ball of radius~$R_0$, it follows as desired that
  $\tilde \gamma$ is also contained in this ball.

  If we exclude this particular situation, we find that $A$ and $B$
  cover together all of $I$.  Below we will show that the number of
  components of $A$ is bounded by some constant~$N'$ that is
  independent of the choice of $\tilde \gamma$.  This obviously
  implies that the restriction of $\tilde \gamma$ to each component of
  $A$ is trapped in a ball of radius~$R_0$.

  The components of $A$ and $B$ alternate so that $B$ may have at most
  $N'+1$ components.  We will then show that the restriction of
  $\tilde \gamma$ to every component~$B_l$ of $B$ has uniformly
  bounded length~$\ell_B$ so that the total length of
  $\tilde \gamma(B)$ is bounded.  Together these facts prove that
  $\tilde \gamma$ lies in a ball of uniformly bounded
  radius~$R = R_0N' + \ell_B\, (N'+1)$.

  \medskip

  We will now show that there is a uniform upper bound on the number
  of times a gradient trajectory of $\tilde f$ can move into the
  vicinity of $\Crit(\tilde f)$ and then again again out of it.  Let
  $f_j$ be the local primitive of $\eta$ on $U_{j,1}$ introduced
  above.  By \Cref{box around critical component}.(b), it follows that
  an $X_\eta$-orbit that enters one of the components of $U_{j,0}$
  either
  \begin{itemize}
  \item is trapped inside $U_{j,1}$, and hits $\p W$ or it eventually
    accumulates at $C_j$,
  \item or it leaves $U_{j,1}$ and the value of $f_j$ increases on the
    way out of $U_{j,1}$ by more than some constant~$\epsilon_j > 0$,
    see \Cref{function increases to leave box around critical
      component}.
  \end{itemize}
  Since there are only finitely many components~$C_j$, denote the
  minimum of the $\epsilon_1,\dotsc,\epsilon_N$ by $\epsilon_0$.  The
  gradient trajectories of $\tilde f$ inherit these properties with
  respect to the components of the lifted neighborhoods
  $\widetilde U_0$ and $\widetilde U_1$ as can be easily deduced from
  \Cref{lemma lift of exact neighborhood}.

  This allows us to show that a gradient trajectory~$\tilde\gamma$
  cannot pass infinitely often from the smaller neighborhood
  $\widetilde U_0$ to the complement of the larger neighborhood
  $\widetilde U_1$.  The reason is simply that every time
  $\tilde\gamma$ crosses one of the neighborhoods~$\widetilde U_0$ and
  then escapes from the larger neighborhood~$\widetilde U_1$, the
  function~$\tilde f$ will increase by more than $\epsilon_0 > 0$.
  Since the values of the function~$\tilde f$ on
  $\widetilde W_{[a,b]}$ lie all in the interval $[a,b]$, none of the
  gradient trajectories in $\widetilde W_{[a,b]}$ can enter
  $\widetilde U_0$ and then leave $\widetilde U_1$ more than
  $N' = \bigl\lceil (b-a)/\epsilon_0\bigr\rceil + 1$ times.  In
  particular, $A$ cannot have more than $N'$ components, and $B$
  cannot have more than $N'+1$ components.
  
  \medskip
  
  Let us now study the restriction of $\tilde \gamma$ to a
  component~$B_0$ of $B$, that is, $\tilde \gamma(B_0)$ does not
  intersect the smaller neighborhood~$\widetilde U_0$.  Rescale
  $\nabla \tilde f$ on the complement of $\widetilde U_0$ to be of the
  form $Z = \frac{\nabla\tilde f}{\|\nabla\tilde f\|^2}$, and let
  $\widetilde \psi$ be the $Z$-trajectory that agrees up to
  parametrization with $\restricted{\tilde \gamma}{B_0}$.
  
  There is a $k > 0$ bounding $\|X_\eta\|$ on the compact set
  $W\setminus \cup_jU_{j,0}$ from below.  As a consequence,
  $\|Z\| < K$ with $K = 1/k$ on the complement of $\widetilde U_0$.
  Furthermore we see from
  \begin{equation*}
    \frac{d}{dt} \tilde f \bigl(\widetilde\psi(t)\bigr) = d\tilde f(Z) = 1
  \end{equation*}
  that for every $t\in\RR$ for which $\widetilde\psi(t)$ is defined
  \begin{equation*}
    \tilde f \bigl(\widetilde\psi(t)\bigr) = 
    \tilde f \bigl(\widetilde\psi(0)\bigr) + t \;.
  \end{equation*}
  Clearly, the trajectory~$\widetilde \psi(t)$ can certainly not exist
  for times larger than $b-a$, because $\widetilde \psi(t)$ will
  either have hit $\widetilde U_0$, $\p \widetilde W$ or
  $\widetilde W_b$ before.  On the complement of $\widetilde U_0$,
  $\|Z\|$ is bounded by $K$ proving that $\widetilde \psi$ and thus
  also $\restricted{\tilde \gamma}{B_0}$ are paths of finite length
  less than $K\,(b-a)$.

  This completes the proof that $\tilde \gamma$ lies in a ball of
  radius $R = R_0N' + K\,(b-a)\, (N'+1)$.

  \medskip

  It still remains to show that the length of $\tilde \gamma$ is
  bounded.  Let $\gamma$ be the projection of $\tilde \gamma$ to $W$,
  and note that length of $\gamma$ and the one of $\tilde \gamma$
  agree.  Since $\|\gamma'\| = \|X_\eta\|$ is bounded on the compact
  manifold~$W$, it is clear that the only trajectories that could
  possibly have infinite length are those defined for all
  $t\in [0,\infty)$.  Suppose thus from now on that $\tilde \gamma$ is
  of this type.
  
  By what we proved above, $\tilde \gamma$ lies in a ball of
  radius~$R$.  By slightly varying the radius of the ball, and cutting
  off its complement, we can suppose that $\tilde \gamma$ lies in a
  compact domain with boundary and corners to which we can apply then
  the doubling trick in \Cref{doubling Morse-Bott} (the corners do not
  pose a problem as explained in the proof of \Cref{theorem:exactness
    with boundary}).  Once in this situation, we obtain the desired
  result by \Cref{length bound for certain gradient curves}.
\end{proof}

\subsection{Connectedness of level sets of a Morse-Bott function on a
  manifold with cylindrical ends}
\label{sec: Morse-Bott with cylindrical ends}

In the previous section we generalized several classical results about
Morse-Bott functions on closed manifolds to Morse-Bott functions on
compact manifolds with convex boundary.

By a classical result, every level set of a Hamiltonian function
generating a circle action on a closed symplectic manifold is either
connected or empty.  This statement, which is one of the key steps for
the proof of the Atiyah--Guillemin-Sternberg convexity theorem for
Hamiltonian torus actions turns out to be false for manifolds with
boundary, even assuming that the boundary is convex, see
\Cref{disconnected level sets}.  To solve this minor technical
problem, we attach cylindrical ends to our manifold.  This will lead
us in this section to \Cref{connected level set for cylindrical end}.

\smallskip

The strategy to show that the level sets of the function~$f$ are
connected is to consider the flow of the vector field
$Z = \frac{1}{\norm{\nabla f}^2}\, \nabla f$.  For closed manifolds,
every point that lies in a level set~$c_0$ is transported in time~$t$
to the level set~$c_0 + t$ unless the point lies in one of the stable
or unstable manifolds of critical points with value between $c_0$ and
$c_0 + t$.  In our situation where the manifold has cylindrical ends
and is hence not compact, the main technical difficulty will be to
show that the trajectories of $Z$ can never escape in finite time to
infinity.

\bigskip

Consider a non-compact manifold~$\widehat{W}$ containing a compact
domain~$W$ with non-empty boundary~$V = \p W$ such that $\widehat{W}$
decomposes as
\begin{equation*}
  \widehat{W} = W \cup_V [0,\infty) \times V \;.
\end{equation*}
We call $[0,\infty) \times V$ the \defin{cylindrical ends} of the
manifold.

\begin{definition}\label{definition function adapted to cylindrical end}
  Let $\hat f\colon \widehat{W}\to \RR$ be a smooth function on a
  manifold
  \begin{equation*}
    \widehat{W} = W \cup [0,\infty) \times V
  \end{equation*}
  with cylindrical ends.  We say that $\hat f$ is \defin{adapted to
    the cylindrical end $[0,\infty) \times V$}, if the restriction of
  $\hat f$ to the cylindrical end is of the product form
  \begin{equation*}
    \restricted{\hat f}{[0,\infty) \times V} (s,p) =
    e^s \, f_V(p) + c\;,
  \end{equation*}
  where $c$ is locally constant (so that it is constant on each of the
  connected components of $V$), and $f_V\colon V \to \RR$ is a smooth
  function on $V$.
\end{definition}

\medskip

We will from now on always assume in this section if not stated
otherwise, that we have chosen a Riemannian metric~$g$ on
$\widehat{W}$ that restricts on $[0,\infty) \times V$ to
\begin{equation*}
  \restricted{g}{[0,\infty) \times V} = e^s\cdot \bigl(ds^2 \oplus g_V\bigr) \;,
\end{equation*}
where $g_V$ is the restriction of $g$ to $V = \p W$.

\begin{remark}\label{geodesically complete}
  A Riemannian metric of this form is geodesically complete.
\end{remark}

Note that if $f_V$ is somewhere on $V$ strictly larger than $0$, then
$\hat f$ is unbounded from above, if $f_V$ is somewhere on $V$
strictly smaller than $0$, then $\hat f$ is unbounded from below.

\smallskip

Let $\hat f$ be a function that is adapted to the cylindrical ends of
$\widehat{W}$ so that it restricts on the ends to $e^s \, f_V(p) + c$,
and let $g$ be a Riemannian metric such that
$\restricted{g}{[0,\infty) \times V} = e^s\cdot \bigl(ds^2 \oplus
g_V\bigr)$.

One easily verifies that the gradient of $\hat f$ simplifies on the
cylindrical ends to
\begin{equation}\label{eq: gradient in cylindrical end}
  \nabla \hat f(s,p) =  f_V(p)\, \partial_s + \nabla f_V(p) \;,
\end{equation}
where $\nabla f_V$ is the gradient vector field of $f_V$ on $V$ with
respect to the metric~$g_V$.

To determine a trajectory $\gamma(t) = \bigl(s(t),p(t)\bigr)$ of
$\nabla \hat f$, integrate first the gradient flow of $f_V$ on $V$ to
find $p(t)$.  The $s(t)$-component is obtained in a second step by
solving $s(t) = \int_0^t f_V(p(\tau))\, d\tau + s_0$.

We easily see from \Cref{eq: gradient in cylindrical end} that a point
$(s,p) \in [0,\infty)\times V$ in the cylindrical end is a critical
point of $\hat f$ if and only if $p \in \Crit(f_V)$ and $f_V(p) = 0$.
In particular, the critical set of $\hat f$ in the cylindrical ends is
thus invariant under $s$-translations.

Note that $\hat f$ is a Morse-Bott function if and only if the
restriction of $\hat f$ to the compact domain~$W$ is a Morse-Bott
function (so that $\Crit(\hat f)$ is transverse to $\p W$).  In this
case it follows that all critical points of $f_V$ lying in
$f_V^{-1}(0)$ are also of Morse-Bott type (even though $f_V$ itself is
usually not).

\medskip

The main result we want to prove in this section is the following
theorem.

\begin{theorem}\label{connected level set for cylindrical end}
  Let $\widehat{W}$ be a manifold with cylindrical ends, and let
  $\hat f$ be a Morse-Bott function on $\widehat{W}$ that is adapted
  to the cylindrical ends and that does not have any critical points
  of index~$i^- = 1$ or $i^+ = 1$.

  Then it follows that the level sets of $\hat f$ are either connected
  or empty.
\end{theorem}

We break up the proof into several steps.  For this we assume from now
on that $\widehat{W}$ is a manifold with cylindrical ends, that
$\hat f$ is a smooth function on $\widehat{W}$ that is adapted to the
cylindrical ends $[0,\infty)\times V$ so that it restricts to
$\hat f(s,p) = e^s \, f_V(p) + c$ for $(s,p) \in [0,\infty)\times V$,
and that $g$ is a Riemannian metric on $\widehat{W}$ that agrees on
$[0,\infty)\times V$ with
$\restricted{g}{[0,\infty) \times V} = e^s\,\bigl(ds^2 \oplus
g_V\bigr)$.

\begin{lemma}\label{cylindrical end convex with respect to gradient}
  \begin{itemize}
  \item [(a)] Let $(s_0,p_0)$ be a point in the cylindrical end at
    which $\nabla \hat f$ has strictly positive (strictly negative)
    $\partial_s$-component.  Then it follows that the gradient
    trajectory~$\gamma$ through this point has for all $t\ge 0$
    strictly increasing $s$-coordinate (strictly decreasing
    $s$-coordinate for all $t\le 0$), and in particular, $\gamma(t)$
    stays for every $t\ge 0$ (for every $t\le 0$) in the cylindrical
    end.  Furthermore $\gamma(t)$ does not encounter any critical
    point, and tends to $s = + \infty$ as $t$ increases towards
    $t=+\infty$ (as $t$ decreases towards $t=-\infty$).
  \item [(b)] The truncation of $\widehat{W}$ along any hypersurface
    $\{s_0\}\times \p W$ in the cylindrical end is a compact
    domain~$\widehat{W}_{\le s_0}$ whose boundary is strongly convex
    with respect to $\nabla \hat f$.
  \end{itemize}
\end{lemma}
\begin{proof}
  \textbf{(a)} The $\partial_s$-component of $\nabla \hat f$ at a
  point~$(s,p)$ in the cylindrical end is $f_V(p)$, see \Cref{eq:
    gradient in cylindrical end}.  If $(s_0,p_0)$ is a point at which
  $\nabla \hat f$ is positively transverse to the level set
  $\{s=s_0\}$, then $f_V(p_0) > 0$, and since
  $\lie{\nabla \hat f} f_V = \norm{\nabla f_V}^2 \ge 0$, it follows
  that $f_V$ increases along the flow line from that moment on.  In
  particular, the trajectory cannot hit any critical point in forward
  direction, since $f_V$ does not decrease.

  Furthermore, the $s$-coordinate of the trajectory continues
  increasing so that it eventually hits $s = +\infty$ for
  $t\to +\infty$ (since $f_V$ is bounded, the trajectory cannot reach
  $s= +\infty$ in finite time).  For the claim in parenthesis, just
  invert the sign of $\hat f$.

  \smallskip
  
  \textbf{(b)} Clearly $\nabla \hat f$ is only tangent to the boundary
  of the domain $\widehat{W}_{\le s_0}$ along the subset
  $\p^0 \widehat{W}_{\le s_0} = \{(s,p)|\; s=s_0, \; f_V(p)=0\}$.
  According to \Cref{eq: gradient in cylindrical end} and \Cref{def:
    strongly convex}, the boundary of $\widehat{W}_{\le s_0}$ is
  strongly convex with respect to $\nabla \hat f$, because at every
  point $(s_0,p) \in \p^0 \widehat{W}_{\le s_0}$, either
  $\nabla f_V(p) = 0$ so that $\nabla \hat f(s_0,p) = 0$, or
  $df_V(\nabla f_V) = \norm{\nabla f_V(s_0,p)}^2$ is strictly
  positive.
\end{proof}

As already explained, we will use a rescaled gradient flow in the
proof of \Cref{connected level set for cylindrical end} to compare the
different level sets of $\hat f$.  The following lemma shows that the
trajectories of this flow do not escape in \emph{finite} time through
the cylindrical end.

\begin{lemma}\label{lemma: Z flow complete}
  Consider the vector field
  \begin{equation*}
    Z := \frac{1}{\norm{\nabla\hat f}^2}\cdot \nabla\hat f
  \end{equation*}
  defined  on $\widehat{W}\setminus \Crit(\hat f)$, let $x$ be any
  point in the domain of $Z$, and let $\gamma$ be the $Z$-trajectory
  through $x$.  Then:
  \begin{itemize}
  \item If $x$ lies in the stable set $W^s(\hat f; \widehat{C}_j)$ of
    one of the connected components~$\widehat{C}_j$ of
    $\Crit(\hat f)$, then $\gamma$ is defined for time $[0,T_x)$ with
    $T_x = \hat f\bigl(\widehat{C}_j\bigr) - \hat f(x)$, and $\gamma$
    extends to a continuous map on $[0,T_x]$ such that $\gamma(T_x)$
    is a point in $\widehat{C}_j$.
  \item If $x$ does not lie in the stable set of any of the components
    of $\Crit(\hat f)$, then $\gamma$ is defined for all $t > 0$ and
    $\gamma(t)$ tends for $t\to \infty$ towards $s = +\infty$ in the
    cylindrical end of $\widehat{W}$.
  \end{itemize}
\end{lemma}
\begin{proof}
  The vector fields~$Z$ and $\nabla \hat f$ are conformal on
  $\widehat{W}\setminus \Crit(\hat f)$ so that their trajectories
  agree up to reparametrization.  This implies that any $Z$-trajectory
  that is confined in a compact domain of $\widehat{W}$, converges by
  \Cref{convergence gradient trajectory} to some point in
  $\Crit(\hat f)$ (this is the only time in this proof that we use
  that $\hat f$ is Morse-Bott).
  
  Since $\lie{Z}\hat f = 1$, a $Z$-trajectory moves any point in the
  level set $\{\hat f = c\}$ (whenever defined) in time $t$ to the
  level set $\{\hat f = c + t \}$.  Thus it follows that if the
  trajectory through $x$ converges to a critical point~$x'$, the
  trajectory will only exist up to time
  $T_x := \hat f(x') - \hat f(x)$.

  \medskip

  If the $Z$-orbit is not confined in a compact domain, it will
  necessarily move at a certain moment into the cylindrical end of
  $\widehat{W}$ crossing one of the $s$-level sets transversely in
  positive direction.  We know from \Cref{cylindrical end convex with
    respect to gradient}.(a) that the underlying gradient trajectory
  continues to move from this point on in the cylindrical end upwards
  towards $s = +\infty$.  This implies that the $Z$-trajectory escapes
  through the cylindrical ends of $\widehat{W}$.  Note though that to
  prove our claim, we also need to show that the trajectory needs
  \emph{infinite} time to reach $s = +\infty$.

  We will show below that the norm of $Z$ is bounded along any
  $Z$-trajectory that moves up towards $s = + \infty$.  Since
  $\widehat{W}$ is by \Cref{geodesically complete} geodesically
  complete, the $Z$-trajectory cannot move in time~$t > 0$ further
  than distance~$K\,t$, where $\norm{Z} < K$.  In particular, none of
  the $Z$-trajectories can escape through the cylindrical end of
  $\widehat{W}$ in finite time.
    
  Recall from \Cref{eq: gradient in cylindrical end} that
  $ \nabla \hat f = f_V\,\partial_s + \nabla f_V$ on the cylindrical
  end.  The norm of $Z$ on the cylindrical end is thus bounded by
  \begin{equation*}
    \norm{Z} = \frac{1}{\norm{\nabla \hat f}}
    =  \frac{e^{-s/2}}{\sqrt{f_V^2 + \norm{\nabla f_V}_V^2}}
    \le  \frac{e^{-s/2}}{\abs{f_V}} \;.
  \end{equation*}
  The function~$f_V$ is monotonously increasing along the
  $Z$-trajectories because
  $df_V(Z) = \frac{\norm{\nabla f_V}^2}{\norm{\nabla \hat f}^2} \ge
  0$.  Thus, if $\gamma(t) = \bigl(s(t), p(t)\bigr)$ is a
  $Z$-trajectory in $[0,\infty)\times V$ such that $s'(0) > 0$, then
  $\norm{Z(\gamma(t))} \le K$ with $K = \frac{1}{\abs{f_V(p(0))}}$
  along this trajectory.
\end{proof}

In \Cref{stable sets are submanifolds}.(b) we showed for Morse-Bott
functions on compact domains with $\nabla$-convex boundary that the
interior of the stable and unstable subsets of the critical points are
always \emph{contained} in finite dimensional smooth submanifolds, but
we did not claim that the stable and unstable subsets are actually
smooth submanifolds themselves.  For cylindrical Morse-Bott functions,
we obtain the following stronger result.

\begin{proposition}\label{stable and unstable manifolds in cylindrical
    completion}
  Let $\widehat{C}_j$ be a connected component of $\Crit(\hat f)$.

  The stable and unstable sets are smooth submanifolds of dimension
  $\dim \widehat{C}_j + i^-(\widehat{C}_j)$ and
  $\dim \widehat{C}_j + i^+(\widehat{C}_j)$ respectively, where
  $i^-(\widehat{C}_j)$ and $i^+(\widehat{C}_j)$ are the indices of
  $\widehat{C}_j$.  The level sets $\{s_0\}\times \p W$ in the
  cylindrical end intersect the stable and unstable subsets
  transversely.
\end{proposition}

We have seen in \Cref{example convex and concave boundary}.(b) that
there exist compact domains with boundary for which the boundary is
$\nabla$-convex with respect to some Morse-Bott function and some
Riemannian metric, but for which the stable and unstable subsets are
not smooth submanifolds.  The situation complicates even even further
because whether the stable or unstable subsets are smooth submanifolds
depends also on the choice of the Riemannian metric, see \Cref{example
  gradient along boundary depends on metric}.  It might be difficult
to decide if such a metric exists.

If we can find a boundary collar on which the Morse-Bott function
resembles the product form of a cylindrical end, the situation
simplifies:

\begin{corollary}
  Let $f\colon W \to \RR$ be a Morse-Bott function on a compact
  manifold with boundary.  Assume that $W$ has a boundary collar of
  the form $(-\epsilon,0]\times \p W$ on which $f$ restricts to
  $f(s,p) = e^s \, f_V(p) + C$, where $f_V$ is a function on
  $V : = \p W$, and $C\in \RR$ is a constant.

  If we choose a Riemannian metric that agrees on the boundary collar
  with $e^s\,\bigl(ds^2 \oplus g_V\bigr)$, then it follows that
  $W^s(f; C_j)$ and $W^u(f; C_j)$ are smooth proper submanifolds.
\end{corollary}

\begin{proof}[Proof of \Cref{stable and unstable manifolds in cylindrical
    completion}]
  Truncate $\widehat{W}$ at any hypersurface $\{s_0\}\times \p W$ to
  obtain a compact domain~$\widehat{W}_{\le s_0}$ whose boundary is
  convex with respect to $\nabla \hat f$, see \Cref{cylindrical end
    convex with respect to gradient}.(b).

  For a connected component~$\widehat{C}_j$ of $\Crit(\hat f)$ that
  does not intersect the boundary of $\widehat{W}_{\le s_0}$, it
  follows from \Cref{stable sets are submanifolds}.(c) that the
  interior of the stable subset of $\widehat{C}_j$ in
  $\widehat{W}_{\le s_0}$ is a smooth submanifold that we denote by
  $W^{\mathrm{stable}}_{< s_0}$.  We show now that
  \begin{equation*}
    W^s\bigl(\hat f; \widehat{C}_j\bigr) \cap \widehat{W}_{< s_0} =
    W^{\mathrm{stable}}_{< s_0} \;.
  \end{equation*}

  Clearly
  $W^{\mathrm{stable}}_{< s_0} \subset W^s\bigl(\hat f;
  \widehat{C}_j\bigr)$.  On the other hand, let
  $p\in W^s\bigl(\hat f; \widehat{C}_j\bigr)$ be a point that lies in
  $\widehat{W}_{< s_0}$.  If the gradient trajectory of $p$ intersects
  for positive time the height level $\{s_0\}\times \p W$, then it
  needs to do so transversely, because $\widehat{W}_{\le s_0}$ has
  $\nabla$-convex boundary.  Since we have shown in \Cref{cylindrical
    end convex with respect to gradient}.(a) that any gradient
  trajectory intersecting a height level $\{s\}\times \p W$
  transversely in outward direction continues towards $s = +\infty$,
  it does not converge a point in $\widehat{C}_j$.  It follows that
  the trajectory through $p$ stays inside $\widehat{W}_{< s_0}$
  proving as desired that $p\in W^{\mathrm{stable}}_{< s_0}$.

  Every $W^{\mathrm{stable}}_{<s}$ with $s\in \NN$ embeds smoothly as
  an open subset in $W^{\mathrm{stable}}_{<s+1}$.  This way,
  $W^s\bigl(\hat f; \widehat{C}_j\bigr) = \cup_{s\in\NN}
  W^{\mathrm{stable}}_{<s}$ is an abstract smooth manifold with the
  structure it obtains as a direct limit.  It is also clear that
  $W^s\bigl(\hat f; \widehat{C}_j\bigr)$ is injectively immersed into
  $\widehat{W}$ so that it only remains to show that it is smoothly
  embedded, that is, its topology agrees with the topology induced as
  a subset of $\widehat{W}$.

  A subset $U \subset W^s\bigl(\hat f; \widehat{C}_j\bigr)$ is an open
  subset, if and only if every $U_s := U\cap W^{\mathrm{stable}}_{<s}$
  for $s\in \NN$ is an open subset of $W^{\mathrm{stable}}_{<s}$.
  Since $W^{\mathrm{stable}}_{<s}$ is smoothly embedded in
  $\widehat{W}_{<s}$, we find an open subset
  $U_s' \subset \widehat{W}_{<s}$ such that
  $U_s = U_s' \cap W^{\mathrm{stable}}_{<s}$.  Consider now
  $U' := \cup_{s\in \NN} U_s'$, which is an open subset of
  $\widehat{W}$, then it follows that
  $U' \cap W^s\bigl(\hat f; \widehat{C}_j\bigr) = U$ so that $U$ is
  also an open subset with respect to the subset topology.  This shows
  that the stable subset $W^s\bigl(\hat f; \widehat{C}_j\bigr)$ is a
  smooth submanifold.  It intersects every hypersurface
  $\{s_0\}\times \p W$ by \Cref{cylindrical end convex with respect to
    gradient}.(a) transversely.

  \medskip

  Assume now that $\widehat{C}_j$ is a connected component of
  $\Crit(\hat f)$ that does intersect the hypersurface
  $\{s_0\}\times \p W$.  We want to use \Cref{lemma for stable sets
    with boundary} to show that the stable subset of
  $\widehat{C}_j\cap \widehat{W}_{\le s_0}$ in $\widehat{W}_{\le s_0}$
  is a smooth submanifold.  For this we first choose $s_1 > s_0$ and
  then we cap off the domain~$\widehat{W}_{\le s_1}$ to obtain a
  closed manifold that we will denote by
  $\widehat{W}_{\le s_1}^{\mathrm{cap}}$.  According to \Cref{doubling
    Morse-Bott}, this poses no problem, and we also obtain that
  $\widehat{W}_{\le s_1}^{\mathrm{cap}}$ carries a Morse-Bott
  function~$\hat f^{\mathrm{cap}}$ that agrees on
  $\widehat{W}_{\le s_1}$ with $\hat f$.  Choose now a Riemannian
  metric on $\widehat{W}_{\le s_1}^{\mathrm{cap}}$ that agrees on the
  subdomain~$\widehat{W}_{\le s_1}$ with the metric on $\widehat{W}$,
  and let $\widehat{C}_j^{\mathrm{cap}}$ be the component of
  $\Crit \bigl(\hat f^{\mathrm{cap}}\bigr)$ that contains
  $\widehat{C}_j\cap \widehat{W}_{\le s_1}$.

  There is a neighborhood of the boundary of $\widehat{W}_{\le s_0}$
  in $\widehat{W}_{\le s_1}^{\mathrm{cap}}$ such that
  $\hat f^{\mathrm{cap}}$ and the Riemannian metric look like $\hat f$
  and the metric on a neighborhood of $\{s_0\}\times \p W$ in the
  cylindrical end.  In particular it follows that
  $\widehat{W}_{\le s_0}$ has $\nabla$-convex boundary in
  $\widehat{W}_{\le s_1}^{\mathrm{cap}}$.
  
  We want to apply \Cref{lemma for stable sets with boundary} to
  $\widehat{W}_{\le s_0}$ in $\widehat{W}_{\le s_1}^{\mathrm{cap}}$.
  Recall that $\widehat{C}_j\cap \bigl(\{s_0\} \times V\bigr)$ with
  $V = \p W$ lies in the subset where both $f_V$ and $\nabla f_V$
  vanish.  It is easy to see that the stable subset of a point
  $(s_0,p_0) \in \widehat{C}_j$ in the cylindrical end necessarily
  projects under the map
  $(-\epsilon,\epsilon)\times V \to V, \;(s,p) \mapsto p$ onto the
  stable subset of $p_0$ with respect to $\nabla f_V$ in $V$.  Since
  the function $f_V$ increases along $\nabla f_V$, and since $f_V$
  vanishes at $p_0$, we obtain that $f_V \le 0$ along the stable
  subset of $p_0$.

  This in turn implies that $\nabla \hat f$ points along
  $W^s\bigl(\hat f; \widehat{C}_j\bigr)$ \emph{into}
  $\widehat{W}_{\le s_0}$ allowing us to use \Cref{lemma for stable
    sets with boundary}.  The interior of the stable manifold of
  $\widehat{C}_j\cap \widehat{W}_{\le s_0}$ in $\widehat{W}_{\le s_0}$
  is a smooth submanifold.

  Since this applies to any $s_0$, we can show that
  $W^s\bigl(\hat f; \widehat{C}_j\bigr)$ is a smooth submanifold of
  $\widehat{W}$, proceeding exactly as above by taking the union over
  all $W^{\mathrm{stable}}_{< s}$ with $s\in \NN$.
  
  Finally note that $W^s\bigl(\hat f; \widehat{C}_j\bigr)$ intersects
  every height level $\{s_0\} \times \p W$ transversely: if
  $(s_0,p_0) \in W^s\bigl(\hat f; \widehat{C}_j\bigr)$ is a point
  where $f_V(p_0) \ne 0$, then $\nabla \hat f$ will be transverse to
  the height level of the cylindrical end; if $f_V(p_0) = 0$, then
  $(s_0,p_0)$ has to lie in $\Crit(\hat f)$, because if $(s_0,p_0)$
  were a regular point, the gradient trajectory of $(s_0,p_0)$ would
  continue by \Cref{cylindrical end convex with respect to gradient}
  to $s = +\infty$ as $t$ increases so that
  $(s_0,p_0) \notin W^s\bigl(\hat f; \widehat{C}_j\bigr)$.  Finally,
  because
  $\Crit(\hat f) \cap W^s\bigl(\hat f; \widehat{C}_j\bigr) =
  \widehat{C}_j$ and because $\widehat{C}_j$ is transverse to the
  $s$-level sets in the cylindrical end, it follows that the stable
  set is also transverse to $\{s\} \times \p W$ at points where
  $f_V = 0$.
\end{proof}

We can now begin studying if the level sets in \Cref{connected level
  set for cylindrical end} are connected.  We first consider level
sets containing a local extremum separately:

\begin{lemma}\label{min connected in cylindrical completion if no index 1}
  In the situation of \Cref{connected level set for cylindrical end},
  assume that there exists a $c_0\in \RR$ such that $\hat f^{-1}(c_0)$
  contains a critical point that is a local minimum or a local
  maximum.

  Then it follows that $\hat f^{-1}(c_0)$ is a component of
  $\Crit(\widehat H)$ and $c_0$ is the absolute minimum/maximum of
  $\hat f$.
\end{lemma}
\begin{proof}
  Let $p\in \hat f^{-1}(c_0)$ be a local minimum.  Truncate
  $\widehat{W}$ at a sufficiently large height level~$s_0$ in the
  cylindrical end so that $p$ lies in the interior of the compact
  domain~$\widehat{W}_{\le s_0}$.
  
  Since $\widehat{W}_{\le s_0}$ is a connected compact manifold with
  $\nabla$-convex boundary, we can apply \Cref{topology simple if no
    codim 1 stable manifolds}.  It follows that there is only one
  component of $\Crit(\hat f)\cap \widehat{W}_{\le s_0}$ that is
  composed of local minima, and the restriction of $\hat f$ to
  $\widehat{W}_{\le s_0}$ is everywhere else on
  $\widehat{W}_{\le s_0}$ strictly larger than $c_0$.  This argument
  remains valid for any $s\ge s_0$, thus proving the lemma.  For a
  local maximum, simply invert the sign of $\hat f$.
\end{proof}

\begin{lemma}\label{regular level sets connected}
  In the situation of \Cref{connected level set for cylindrical end},
  every \emph{regular} level set of $\hat f$ is either empty or
  connected.
\end{lemma}
\begin{proof}
  Since every component of $\Crit(\hat f)$ intersects $W$,
  $\Crit(\hat f)$ may only have finitely many components.  The stable
  and unstable sets of any component of $\Crit(\hat f)$ that is
  neither a maximum nor a minimum are by \Cref{stable and unstable
    manifolds in cylindrical completion} smooth submanifolds that are
  of codimension at least $2$ in $\widehat{W}$.  The set of critical
  points~$\Crit(\hat f)$ is also a finite collection of submanifolds
  of codimension at least $2$ in $\widehat{W}$.

  Let $U$ be the open set obtained from $\widehat{W}$ by removing
  $\Crit(\hat f)$ and all stable and unstable submanifolds of critical
  points that are not local extrema.  All these subsets are of
  codimension at least~$2$ in $\widehat{W}$, thus it follows that $U$
  is connected.

  \medskip

  Assume now that $c_0$ were a regular value of $\hat f$ with a
  non-empty and non connected preimage $N_0 := \hat f^{-1}(c_0)$.  To
  simplify the notation we replace $\hat f$ by $\hat f -c_0$ to assume
  that $c_0 = 0$.  The stable and unstable submanifolds of critical
  points intersect every regular level set transversely so that it is
  easy to convince oneself that $N_0\cap U$ is a codimension~$1$
  submanifold of $U$ that will also be disconnected.

  \smallskip
  
  We have shown in \Cref{lemma: Z flow complete} that the forward
  trajectories of $Z$ are either complete or they converge to critical
  points.  Consider the restriction of the vector field~$Z$ to the
  subset~$U$.  Having removed the smaller dimensional stable and
  unstable manifolds, the flow of $Z$ through a point
  $p \in N_0 \cap U$ is defined up to time $T_+ = c_{\max}$ if
  $\hat f$ has a maximum, or up to $T_+ = \infty$ otherwise.
  Similarly, the flow of $Z$ in backward time is defined up to
  $T_- = c_{\min}$ if $\hat f$ has a minimum, or up to $T_- = -\infty$
  otherwise.  The flow of $Z$ defines thus a diffeomorphism
  \begin{equation*}
    \Phi\colon (T_-,T_+) \times (N_0\cap U) \to U \;,
  \end{equation*}
  and if $N_0\cap U$ were disconnected, so would be
  $U = \Phi\bigl((T_-,T_+) \times (N_0\cap U)\bigr)$ giving a
  contradiction to the fact that $U$ is connected.  It follows that
  every regular level set of $\hat f$ needs to be connected.
\end{proof}

\medskip

\begin{proof}[Proof of \Cref{connected level set for cylindrical end}]
  Let $c\in \RR$ be a number such that the level set
  $N_c := \hat f^{-1}(c)$ is non-empty.  If $c$ is a regular value or
  if $c$ is the maximum or minimum of $\hat f$, then it follows by
  \Cref{regular level sets connected,min connected in cylindrical
    completion if no index 1} that $N_c$ is connected.  Assume thus
  that $c$ is a critical value that is neither a maximum nor a
  minimum.

  Choose any two points~$p_0, p_1\in N_c$.  The aim is to show that
  these points can be connected to each other by a path in $N_c$.  If
  $p_0$ or $p_1$ are critical points of $\hat f$, we connect them
  first with a smooth path in $N_c$ to a regular point: if $p_0$ lies
  for example in one of the components~$\widehat{C}_j$ of
  $\Crit(\hat f)$, choose a Morse-Bott chart as in
  \cite{BanyagaMorseBott} centered at $p_0$ with coordinates
  \begin{equation*}
    \bigl(x_1,\dotsc,x_{i^-}; y_1,\dotsc,y_{i^+};
    z_1,\dotsc,z_d\bigr)
  \end{equation*}
  where $i^- = i^-(\widehat{C}_j)$, $i^+ = i^+(\widehat{C}_j)$, and
  $d = \dim \widehat{C}_j$, such that $\hat f$ takes the form
  \begin{equation*}
    \hat f\bigl(x_1,\dotsc,x_{i^-}; y_1,\dotsc,y_{i^+};
    z_1,\dotsc,z_d\bigr) = c +  \sum_{i=1}^{i^+} y_i^2 - \sum_{i=1}^{i^-} x_i^2 \;.
  \end{equation*}
  Since we are assuming that $c = \hat f(p_0)$ is neither a local
  maximum or minimum, both $i^-(\widehat{C}_j)$ and
  $i^+(\widehat{C}_j)$ are different from $0$, and the path
  $\gamma(t) = \bigl(t,0,\dotsc,0;t,0,\dotsc,0;0,\dotsc,0\bigr)$ for
  $t\in [0,\epsilon]$ lies in the level set~$N_c$ and connects $p_0$
  to a regular point of $\hat f$.

  \smallskip
  
  Suppose thus from now on that both $p_0$ and $p_1$ are regular
  points of $\hat f$.  The function~$\hat f$ has only
  finitely many critical values, because all of them agree by
  \Cref{cylindrical end convex with respect to gradient} with those of
  $\restricted{\hat f}{W}$.  For any sufficiently small choice of
  $\delta > 0$, the interval~$[c,c+\delta]$ lies in the image of
  $\hat f$ and does not contain any critical values of
  $\hat f$ except for $c$.

  \smallskip
  
  We will follow the classical strategy to move $p_0$ and $p_1$ along
  the trajectories of the vector field~$Z$ in \Cref{lemma: Z flow
    complete} to points~$p_0'$ and $p_1'$ in the regular level
  set~$N_{c+\delta} := \hat f^{-1}(c+\delta)$.  By \Cref{regular level
    sets connected}, $N_{c+\delta}$ is connected and thus we can join
  $p_0'$ and $p_1'$ with a path~$\gamma$ in~$N_{c+\delta}$.

  The aim is to translate this path along the flow of $-Z$ into the
  initial level set~$N_c$.  For this strategy to work, we need to make
  sure that $\gamma$ does not meet any of the critical points of
  $\hat f$ as we push it along the vector field~$-Z$.  The only
  critical value in $[c,c+\delta]$ is $c$.

  To avoid any technical complication with respect to the limit of the
  flow of $-Z$, we simply use that all the unstable subsets
  $W^u(\hat f; \widehat{C}_j)$ of any component
  $\widehat{C}_j \subset \Crit(\hat f)$ lying in $N_c$ are smooth
  submanifolds of codimension at least $2$ in $\widehat{W}$, see
  \Cref{stable and unstable manifolds in cylindrical completion}, and
  they intersect $N_{c+\delta}$ transversely.  Thus there is no
  problem in perturbing the path~$\gamma$ inside $N_{c+\delta}$
  keeping the end-points fixed, so that $\gamma$ avoids all unstable
  submanifolds $W^u(\hat f; \widehat{C}_j)$ for
  $\widehat{C}_j \subset N_c$.
  
  After perturbing $\gamma$ , we can move it with the time~$\delta$
  flow of $-Z$ back to the initial level set~$N_c$ where it joins
  $p_0$ and $p_1$ proving as desired that $N_c$ is connected.
\end{proof}

\section{Hamiltonian $G$-manifolds with contact type boundary}
\label{sec: hamiltonian with contact type boundary}

\subsection{Definitions and preliminaries}\label{definitions and
  preliminaries}

In this section, we briefly give several definitions and technical
results about Hamiltonian group actions that are mostly well-known.
We recommend the reader to jump directly to \Cref{sec: hamiltonian
  morse-bott and convex boundary} and only consult \Cref{definitions
  and preliminaries} below when looking for a reference.

\medskip

The foremost tool when working with compact group actions on manifolds
consists in averaging certain sections in tensor bundles using the
corresponding Haar measure of the group.  This way, we can for example
easily obtain Riemannian metrics for which $G$ acts by isometries, and
with these techniques, we can also obtain invariant Liouville vector
fields.

\begin{proposition}
  Let $(W,\omega)$ be a symplectic manifold with contact type
  boundary~$V = \p W$, and let $G$ be a compact Lie group that acts on
  $W$ via symplectomorphisms.

  Then there is a $G$-invariant Liouville vector field~$Y$ in a
  neighborhood of $V$ that induces an invariant contact
  structure~$\xi$ on the boundary.  The choice of this contact
  structure is unique up to equivariant contactomorphisms.
\end{proposition}
\begin{proof}
  Let $Y'$ be any Liouville vector field on a neighborhood of $V$.  We
  can average $Y'$ (after possibly decreasing the neighborhood of
  $\p W$) and define
  \begin{equation*}
    Y := \int_G (g^*Y')  \, dg \;.
  \end{equation*}
  It is easy to see that $Y$ is still a Liouville vector field on a
  neighborhood of $V$ pointing outwards (every $g\in G$ respects the
  boundary coorientation of $V$, because to change the coorientation,
  the collar neighborhood of $V$ would have to be flipped by $g$ to
  the ``other side'' of $V$, that is, outside $W$).

  We obtain this way a $G$-invariant Liouville
  form~$\lambda_Y = \iota_Y \omega$ on a neighborhood of $V$ that
  induces an invariant contact structure~$\xi$ and an invariant
  contact form~$\alpha := \restricted{(\lambda_Y)}{TV}$ on $V$.
  
  Furthermore, if $Y_1$ and $Y_2$ are both $G$-invariant Liouville
  vector fields, we can linearly interpolate between them to see with
  Gray stability that the corresponding contact structures~$\xi_1$ and
  $\xi_2$ are $G$-equivariantly isotopic.
\end{proof}

Thus we can and we will from now on always assume that if a compact
Lie group~$G$ acts on a symplectic manifold~$(W,\omega)$ with contact
type boundary, then the contact structure and the Liouville vector
field are also $G$-invariant.

\begin{lemma}\label{lemma: hamiltonian function through liouville
    form}
  Let $(W,\omega)$ be a symplectic manifold, and let $\lambda$ be a
  local primitive of $\omega$ that is defined on some open subset
  $U\subset W$.
  
  \begin{itemize}
  \item [(a)] If $X$ is a symplectic vector field on $U$ that
    preserves $\lambda$, then $H_\lambda = \lambda(X)$ is a
    Hamiltonian function for $X$.
  \item [(b)] If $G$ is a Lie group that acts on $U$ preserving
    $\lambda$, then we obtain a $G$-equivariant moment map by setting
    \begin{equation*}
      \langle \mu_\lambda(p), X\rangle := \lambda_p(X_U)
    \end{equation*}
    for every $X\in \gfrak$ and every $p\in U$.
  \end{itemize}
\end{lemma}

\begin{proof}
  \textbf{(a)} By definition we need to show that
  $\iota_X \omega = -dH_\lambda$ which follows directly from Cartan's
  formula
  \begin{equation*}
    -dH_\lambda = -d\iota_X\lambda = - \lie{X} \lambda + \iota_X d\lambda \;,
  \end{equation*}
  because $\lie{X} \lambda = 0$.
  
  \smallskip

  \textbf{(b)} The map~$\mu_\lambda$ is clearly pointwise linear and
  it only remains to show its $G$-equivariance.  Let $g$ be any
  element of $G$ and let $X$ be an element in the Lie
  algebra~$\gfrak$.  Then we compute
  \begin{equation}
    \bigl(\Ad_gX\bigr)_U (p) =  \restricted{\frac{d}{dt}}{t=0} \bigl(\exp (t
    \Ad_gX)\bigr)\, p =  
    \restricted{\frac{d}{dt}}{t=0} \bigl(g \exp (tX) g^{-1}\bigr)\, p =
    Dg \, X_U\bigl(g^{-1} p\bigr) \label{eq: Ad X is Dg X}
  \end{equation}
  which confirms the desired property of the moment map
  \begin{equation*}
    \begin{split}
      \bigl\langle \mu_\lambda\bigl(gp\bigr), X\bigr\rangle &=
      \lambda_{gp} \bigl(X_U(gp)\bigr) =
      \lambda_{gp} \bigl(Dg\, Dg^{-1}\, X_U(gp)\bigr) \\
      &= \bigl(g^*\lambda\bigr)_p \bigl(Dg^{-1}\, X_U(gp)\bigr) =
      \lambda_p \Bigl(\bigl(\Ad_{g^{-1}} X\bigr)_U (p)\Bigr) =
      \bigl\langle \mu_\lambda\bigl(p\bigr), \Ad_{g^{-1}}
      X\bigr\rangle
    \end{split}
  \end{equation*}
  that is
  $\mu_\lambda\bigl(gp\bigr) = \Ad_g^* \bigl(\mu_\lambda(p)\bigr)$.
\end{proof}

\begin{corollary}\label{cor: exactness implies Hamiltonian}
  Symplectic actions of compact Lie groups on exact symplectic
  manifolds are always Hamiltonian.
\end{corollary}

\begin{remark}\label{remark: hamiltonian function through liouville
    form}
  Clearly if $X$ is a globally defined Hamiltonian vector field with
  Hamiltonian function~$H\colon W\to \RR$, then it follows in the
  setup of the preceding lemma that $\restricted{H}{U}$ and
  $\lambda(X)$ agree up to addition of a constant.

  If $G$ acts symplectically on $(W,\omega)$ with moment map
  $\mu\colon W\to \gfrak^*$, then it follows in the setup of the
  preceding lemma that
  \begin{equation*}
    \restricted{\mu}{U} = \mu_\lambda + \nu_0
  \end{equation*}
  for a covector~$\nu_0 \in \gfrak^*$.
\end{remark}

By \Cref{lemma: hamiltonian function through liouville form}, it is
obvious that a symplectic action of a compact Lie group is Hamiltonian
when restricted to a collar neighborhood of a contact type boundary.
In \Cref{sec: Hamiltonian circle action} we will show that such
actions are in fact even globally Hamiltonian.

The difference between weakly Hamiltonian and Hamiltonian $G$-actions
can be detected on any $G$-invariant neighborhood.

\begin{lemma}\label{lemma: if weakly Ham and Ham on subset then Ham everywhere}
  Le $(W,\omega)$ be a connected symplectic $G$-manifold.  Assume that
  the $G$-action is weakly Hamiltonian and that there is a
  $G$-invariant connected open subset $U\subset W$ such that the
  restriction of the $G$-action to $U$ is Hamiltonian with moment map
  $\mu_U\colon U\to \gfrak^*$.

  Then it follows that the action is Hamiltonian on all of $W$.
\end{lemma}
\begin{proof}
  Choose a basis $X_1,\dotsc,X_k$ for the Lie algebra of $G$.  Then we
  find for every $X_j$ a Hamiltonian function~$H_j$ on $W$ that is
  unique up to addition of a constant.  Choose these constants in such
  a way that every $H_j$ agrees on $U$ with
  $\langle\mu_U, X_j\rangle$.

  We then define a moment map~$\mu$ on all of $W$ by setting
  \begin{equation*}
    \bigl\langle \mu, a_1 X_1 + \dotsm + a_k X_k\bigr\rangle
    = a_1 H_1 + \dotsm + a_k H_k
  \end{equation*}
  for every $a_1,\dotsc,a_k\in \RR$.  It only remains to show that
  $\mu$ is $G$-equivariant.

  We essentially follow the strategy explained in
  \cite[Section~5.2]{McDuffSalamonIntro}.  Choose an $X\in \gfrak$,
  and a $g\in G$, and denote $\Ad(g^{-1}) X$ by $X'$.  The
  infinitesimal generator corresponding to $X'$ is given by
  $X_W'(p) = Dg^{-1} \, X_W(gp)$, as seen in \Cref{eq: Ad X is Dg X}.

  Compare now the function
  $p\mapsto \bigl\langle\mu(p), \Ad(g^{-1}) X\bigr\rangle$ to the
  function $p\mapsto \langle\mu(gp), X\rangle$.  The moment map is
  $G$-equivariant if and only if both functions are equal.  The
  Hamiltonian vector field of the first function is by definition
  $X'_W$, the vector field corresponding to the second function is
  $Dg^{-1} \, X_W\circ g$, because
  \begin{equation*}
    \omega\bigl(Dg^{-1} X_W\circ g, \cdot\bigr) =
    \omega\bigl(Dg\,Dg^{-1} X_W \circ g, Dg\,\cdot\bigr) =
    \omega\bigl(X_W \circ g, Dg\,\cdot\bigr) =
    g^*\bigl(\omega(X_W,\,\cdot)\bigr) =
    - g^*d\langle\mu, X\rangle \;.
  \end{equation*}
  As we have shown above, both Hamiltonian vector fields are
  identical, so that the respective Hamiltonian functions only differ
  by a constant.  For every $X\in \gfrak$, there is thus a
  constant~$c_X$ such that
  \begin{equation*}
    \langle\mu(gx), X\rangle =
    \bigl\langle\Ad^*_{g} \bigl(\mu(x)\bigr),X \bigr\rangle + c_X \;.
  \end{equation*}
  To prove \eqref{eq:coadjoint equivariance}, it only remains to show
  that $c_X = 0$ for every $X\in \gfrak$.
  
  Recall that the restriction of the moment map~$\mu$ to $U$ is equal
  to the $G$-equivariant map~$\mu_U$ so that $c_X = 0$ on $U$, but
  since $c_X$ is a constant it vanishes then on all of $W$ showing
  that \eqref{eq:coadjoint equivariance} holds everywhere.
\end{proof}

\bigskip

We describe now a normal form for the boundary collar of a symplectic
manifold with contact type boundary.  The result is well-known but we
give nonetheless a sketch of the construction to show that the model
respects the group action.

\begin{lemma}\label{collar neighborhood}
  Let $(W,\omega)$ be a symplectic manifold, and let $X$ be a
  Hamiltonian vector field on $W$ with Hamiltonian function
  $H\colon W\to \RR$.  Assume that $W$ has contact type
  boundary~$V = \p W$ and that there is a Liouville field~$Y$ defined
  in a neighborhood of $V$ that commutes with $X$, inducing an
  $X$-invariant Liouville form~$\lambda$, and a contact
  form~$\alpha := \restricted{\lambda}{TV}$.

  It then follows that $W$ admits a boundary collar that is
  diffeomorphic to $(-\epsilon,0]\times V$ such that
  \begin{itemize}
  \item $V$ is naturally identified with $\{0\}\times V$, and the
    Liouville field corresponds to $\partial_s$ where $s$ is the
    coordinate on $(-\epsilon,0]$;
  \item the Liouville form is diffeomorphic to $e^s\alpha$, and the
    symplectic structure~$\omega$ is symplectomorphic to
    $d(e^s\alpha)$;
  \item there is a function $f\colon V\to \RR$ and a vector
    field~$X_V$ on $V$ such that $X$ simplifies on the collar
    neighborhood to $X(s,p) = f(p)\, \partial_s + X_V(p)$ for any
    $(s,p) \in (-\epsilon,0]\times V$;
  \item the Hamiltonian function~$H$ restricts on the collar to
    $H(s,p) = e^s\, \alpha_p(X_V) + c$ where $c$ is a constant.
  \end{itemize}
\end{lemma}
\begin{proof}
  Using Leibniz formula
  $L_X({\iota_Y\omega})=(L_X\omega)(Y,\cdot)+\omega(L_XY,\cdot)$ one
  can easily see that the Liouville form $\lambda = \iota_Y\omega$ is
  $X$-invariant.  We use the flow of the Liouville vector field~$Y$ to
  obtain a collar neighborhood
  \begin{equation*}
    (-\epsilon, 0] \times V \to W, \quad (s,p) \mapsto \Phi^Y_s(p) \;.
  \end{equation*}
  With this identification, $Y$ corresponds to $\partial_s$.

  From $\iota_Y \omega = \lambda$ we deduce that
  $\iota_{Y} \lambda = 0$ and $\lie{Y} \lambda = \lambda$.  It then
  follows that the pull-back of $\lambda$ to the collar neighborhood
  does not have any $ds$-terms, and since it agrees with
  $\restricted{\lambda}{\{0\}\times V} = \alpha$ along $V$, we obtain
  $\lambda = e^s\, \alpha$ as desired.

  The vector field~$X$ takes on the collar model the form
  $X(s,p) = f(s,p)\, \partial_s + X_V(s,p)$.  Here $f$ is some
  function and $X_V$ is a vector field on the collar neighborhood that
  is tangent to the slices $\{s\}\times V$.  Since $[Y,X] = 0$, it
  follows that neither $f$ nor $X_V$ depend on the $s$-coordinate
  giving us the desired form for $X$.

  The Hamiltonian function is then obtained by combining that
  $\lambda = e^s\, \alpha$ with \Cref{lemma: hamiltonian function
    through liouville form}.
\end{proof}

A direct corollary of \Cref{collar neighborhood} is that if a Lie
group~$G$ acts symplectically on a manifold $(W,\omega)$ with contact
type boundary~$V = \p W$, then we can choose an invariant Liouville
form~$\lambda$ and an invariant contact form
$\alpha = \restricted{\lambda}{TV}$ such that the boundary collar of
$W$ is $G$-equivariantly diffeomorphic to $(-\epsilon, 0] \times V$
with $\lambda = e^s\alpha$, and such that the $G$-action agrees with
$g\cdot (s,p) = (s, gp)$ for any $g\in G$ and any
$(s,p) \in (-\epsilon,0]\times V$.  The natural moment
map~$\mu_\lambda$ associated to the $G$-action and the Liouville
form~$\lambda$ (see \Cref{lemma: hamiltonian function through
  liouville form}) simplifies in this neighborhood to
\begin{equation}\label{eq: moment map}
  \bigl\langle\mu_\lambda(s,p), X\bigr\rangle = e^s \alpha_p(X_V) \;,
\end{equation}
for every $X\in \gfrak$, and every $(s,p)\in (-\epsilon,0]\times V$.
Here $X_V$ denotes the infinitesimal generator associated to the
restriction of the $G$-action to $V$.

\begin{definition}\label{def: cylindrical end}
  Let $(W,\omega)$ be a symplectic manifold with contact type
  boundary~$V = \p W$.  Choose a collar neighborhood of the form
  $\bigl((-\epsilon,0]\times V, d(e^s\,\alpha)\bigr)$ as explained in
  \Cref{collar neighborhood}.
  
  We attach a \defin{cylindrical end} to $(W,\omega)$ by defining the
  open symplectic manifold
  \begin{equation}
    (W,\omega)\sqcup \bigl((-\epsilon,\infty)\times V,\;
    d(e^s\alpha)\bigr)
    \label{gluing cylindrical end}
  \end{equation}
  and gluing both parts smoothly to each other using the obvious
  identification along the collar neighborhood
  $(-\epsilon,0]\times V$.  Denote the resulting manifold by
  $\bigl(\widehat{W}, \widehat{\omega}\bigr)$.
\end{definition}
  
If $(W,\omega)$ comes with a Hamiltonian $G$-action with moment
map~$\mu$ and if the boundary collar of \Cref{collar neighborhood} has
been obtained using an invariant Liouville vector field, then the
action simplifies on this collar for any point
$(s,p)\in (-\epsilon, 0]\times V$ to
\begin{equation*}
  g\cdot (s,p) = (s,gp)
\end{equation*}
and there is an element~$\nu_0 \in \gfrak^*$ such that the moment map
is according to \eqref{eq: moment map} and \Cref{remark: hamiltonian
  function through liouville form} equal to
\begin{equation*}
  \bigl\langle\mu(s,p), X\bigr\rangle =
  e^s\, \alpha_p(X_V) + \langle\nu_0,X\rangle
\end{equation*}
for every $X\in \gfrak$.

We can thus extend the action and the moment map in a straightforward
way to the cylindrical end such that the cylindrical completion
$(\widehat{W},\widehat{\omega})$ will also be a Hamiltonian
$G$-manifold.

\begin{remark}\label{symplectic completion independent}
  The result of attaching a cylindrical end to a Hamiltonian
  $G$-manifold is up to $G$-equivariant symplectomorphisms independent
  of the choice of the boundary collar chosen according to
  \Cref{collar neighborhood}.
\end{remark}

For the sake of completeness, we give a proof of this remark in
\Cref{sec: proof symplectic completion independent}.

\subsubsection{Almost complex structures}
\label{sec:almost complex structures}

It is well-known that for a symplectic action of a compact Lie
group~$G$, the components of the fixed point set~$\Fix(G)$ are
isolated symplectic submanifolds.  Moreover, if $G = \SS^1$, then it
follows that the corresponding (local) Hamiltonian function has
critical points of Morse-Bott type along $\Fix(\SS^1)$, and the
indices of all critical points are even.

The proof of these facts relies on a local argument \cite{Frankel}
using a compatible $G$-invariant almost complex structure.  Recall
that a compatible almost complex structure~$J$ on a symplectic
manifold~$(W,\omega)$ is a bundle endomorphism $J\colon TW \to TW$
such that $J^2 = - \id_{TW}$ and such that
\begin{itemize}
\item $\omega(Jv,Jw) = \omega(v,w)$ for every $x\in W$ and for every
 $v,w \in T_xW$, and
\item $\omega(v,Jv) \ge 0$ for every $x\in W$ and for every
 $v \in T_xW$ with equality if and only if $v = 0$.
\end{itemize}

A compatible almost complex structure~$J$ defines a Riemannian metric
\begin{equation}
 \langle v, w \rangle := \omega(v,Jw) \;,
 \label{eq: riemannian metric given by almost complex structure}
\end{equation}
and we can easily verify the following relations between a Hamiltonian
vector field~$X_H$ associated to a function $H\colon W \to \RR$ and
the gradient vector field~$\nabla H$
\begin{equation}
 X_H = J\cdot \nabla H \quad \text{ and }\quad
 \nabla H = - J\cdot X_H \;,
 \label{eq: relation gradient and Hamiltonian vector field}
\end{equation}
because
$dH(v) = \langle \nabla H, v\rangle = \omega(\nabla H, J v) =
\omega(J\nabla H, J^2 v) = \omega(-J\nabla H, v)$ agrees with the
definition of the Hamiltonian vector field~$X_H$.

\smallskip

Slightly less direct as for a metric, it is nonetheless possible to
average a compatible~$J$ over a compact Lie group.  The fixed points
of the group are then almost complex submanifolds and the normal
bundle of the fixed points splits into complex subbundles showing that
the indices of the critical points of a Hamiltonian circle action are
all even.

\smallskip

It also follows directly from \eqref{eq: relation gradient and
  Hamiltonian vector field} that every Hamiltonian $\SS^1$-manifold
with an invariant compatible almost complex structure~$J$ is foliated
by $J$-holomorphic cylinders (that are singular at the fixed points).
The cylinders are obtained by taking the union of the $\SS^1$-orbits
of all points lying along a chosen gradient trajectory of the
Hamiltonian function.

\medskip

In the context of Hamiltonian $G$-manifolds with contact type
boundary, it would be natural to impose additional conditions on the
almost complex structures to make them compatible with the contact
type boundary (or even to cylindrical ends).  This way, the boundary
would be the regular level set of a pluri-subharmonic function.  For
the purposes of this article, such additional conditions are not
necessary, and we do not give any further details.

\subsection{Hamiltonian $\SS^1$-manifolds with contact type
  boundary}\label{sec: hamiltonian morse-bott and convex boundary}

Using Morse-Bott techniques, Atiyah \cite{AtiyahConvexity} and
independently Guillemin-Sternberg \cite{GuilleminSternbergConvexity}
analyzed the moment map of a Hamiltonian $\SS^1$-action and deduced
many important results.  With their method it follows for example
easily that the moment map of such an action on a closed connected
manifold must have a unique component of local minima and a unique
component of local maxima.

The argument is as follows: The Hamiltonian function is Morse-Bott and
the stable manifolds of all components of critical points are even
dimensional.  Consider the subset~$U$ of all stable manifolds of
codimension~$0$, that is, the union of stable subsets that correspond
to components of local maxima.  The codimension~$0$ stable manifolds
are open and they partition $U$.  Since we have only excluded a
collection of finitely many submanifolds of codimension at least two,
it follows that $U$ is connected.  This implies that there is exactly
one codimension~$0$ stable manifold proving also that there is exactly
one component of local maxima.

\smallskip

In this section, we will generalize this result to Hamiltonian
manifolds with contact type boundary.  We begin this section by
showing that such manifolds fit into the more general framework of
Morse-Bott functions with convex boundary described in \Cref{sec:
  morse-bott with convex boundary}.  In particular, the Hamiltonian
functions are Morse-Bott, the boundary is convex with respect to their
gradient vector fields, and the components of critical points
intersecting the boundary have a suitable neighborhood so that we can
apply \Cref{topology simple if no codim 1 stable manifolds} to such a
Hamiltonian manifold.

\begin{lemma}\label{boundary is convex with respect to gradient}
  Let $(W,\omega)$ be a compact symplectic manifold with boundary.
  Suppose that there is boundary component $V \subset \p W$ such that
  \begin{itemize}
  \item $X$ and $Y$ are vector field that are defined on a
    neighborhood of $V$, such that $Y$ is a Liouville field, and $X$
    is Hamiltonian;
  \item $X$ and $Y$ commute, $X$ is tangent to $V$, and $Y$ is
    positively transverse to $V$ so that $V$ is of (convex) contact
    type;
  \item $H = \lambda(X)$ is the Hamiltonian function of $X$ associated
    via \Cref{lemma: hamiltonian function through liouville form} to
    the Liouville form~$\lambda := \iota_Y\omega$.
  \end{itemize}
  
  Then there exists a Riemannian metric~$g$ on $W$ that can be chosen
  arbitrarily away from $V$ such that
  \begin{itemize}
  \item [(a)] $V$ is convex with respect to the corresponding gradient
    field~$\nabla H$;
  \item [(b)] we have
    \begin{equation*}
      \begin{aligned}
        V^+ &:= \bigl\{ p\in V\bigm|\, \text{$\nabla H(p)$ points out
          of $W$}\bigr\} &= \{p\in V|\; H(p) > 0\} \;, \\
        V^0 &:= \bigl\{ p\in V\bigm|\, \nabla H(p)\in T_p V\bigr\} &= \{p\in V|\; H(p) = 0\} \;,  \\
        V^- &:= \bigl\{ p\in V\bigm|\, \text{$\nabla H(p)$ points into
          $W$}\bigr\} &= \{p\in V|\; H(p) < 0\} \;;
      \end{aligned}
    \end{equation*}
  \item [(c)] if $C$ is a component of $\Crit(H)$ such that $C\cap V$
    is composed of local minima (local maxima) of $\restricted{H}{V}$,
    then there exists a neighborhood~$U_C \subset V$ of $C\cap V$ such
    that $\nabla H$ does not point anywhere on $U_C$ transversely into
    $W$ (transversely out of $W$).
  \end{itemize}
\end{lemma}
\begin{proof}
  By our assumption, $X$ restricts to a vector field~$X_V$ on $V$.
  Taking the boundary collar given by \Cref{collar neighborhood}, we
  see then that $X$ simplifies on $(-\epsilon,0] \times V$ to
  $X(s,p) = X_V(p)$, and the Hamiltonian function~$H$ can be written
  as $H(s,p) = e^s\, \alpha_p(X_V)$ for
  $(s,p) \in (-\epsilon,0] \times V$.

  Attach to $W$ a cylindrical end as described in \Cref{def:
    cylindrical end}, and denote the resulting manifold by
  $(\widehat{W}, \hat \omega)$.  The Hamiltonian function~$H$ extends
  naturally to a function on $\widehat{W}$ that is adapted to the
  cylindrical end in the sense of \Cref{sec: Morse-Bott with
    cylindrical ends}, and that we denote for simplicity also by $H$.

  Choose a Riemannian metric~$g$ on $\widehat{W}$ that agrees with
  \begin{equation*}
    \restricted{g}{[0,\infty) \times V} = e^s\cdot \bigl(ds^2 \oplus g_V\bigr) \;,
  \end{equation*}
  on $[0,\infty) \times V$, where $g_V$ is any metric on $V$, then it
  follows from \Cref{cylindrical end convex with respect to gradient},
  that the truncation of $\widehat{W}$ at any level set~$s = s_0$ in
  the cylindrical end gives a manifold~$\widehat{W}_{\le s_0}$ whose
  boundary is convex with respect to $\nabla H$.  In particular it
  follows that $V\subset \p W$ is convex with respect to the gradient
  field of $H$ proving (a).
  
  \medskip

  Recall now from \eqref{eq: gradient in cylindrical end} that
  $\nabla H$ is given with respect to $g$ by
  \begin{equation*}
    \nabla H(s,p) =  H_V(p)\, \partial_s + \nabla H_V(p) \;,
  \end{equation*}
  on the cylindrical end.  Here $H_V = \restricted{H}{V}$.  This shows
  immediately our claim~(b) that the sign of $H_V = H$ determines
  whether $\nabla H$ points into or out of the manifold.

  \smallskip
  
  For (c), use that a point $p\in V$ is only in $\Crit(H)$ if both
  $H_V(p)$ and $\nabla H_V(p)$ vanish.  If $p$ is additionally a local
  minimum of $H_V$, then it follows that $H_V \ge 0$ on a neighborhood
  of $p$ so that $\nabla H$ does not point anywhere on this
  neighborhood transversely into $W$.  For the statement in
  parenthesis, it suffices to invert the sign of $H$.
\end{proof}

The lemma above applied directly to compact Hamiltonian
$\SS^1$-manifolds with convex contact type boundary, showing that the
convex contact type boundary of such a manifold is also
$\nabla$-convex.

\begin{remark}
  Jean-Yves Welschinger pointed out to us that the convexity of $\p W$
  with respect to $\nabla H$ can also be obtained as a reformulation
  of the fact that non-constant $J$-holomorphic curves cannot touch
  the $J$-convex boundary from the inside.  In our situation, it is
  always possible to choose an almost complex structure making the
  boundary $J$-convex.  As stated in \Cref{sec:almost complex
    structures}, it follows that $X$ and $\nabla H = - J X$ span a
  $J$-holomorphic foliation.  Every leaf of this foliation touching
  $\p W$ from the inside needs thus to be a point.

  This approach is probably the one preferred by symplectic
  topologists, but note that it does not replace any of the arguments
  of \Cref{sec: morse-bott with convex boundary} as it only gives a
  different proof for the $\nabla$-convexity of the boundary.
  Additionally further effort would be needed to prove the
  requirements of find \Cref{lemma for stable sets with boundary} to
  also be able to treat $\SS^1$-manifolds having fixed points in the
  boundary.
\end{remark}

\begin{theorem_fixed_points_boundary_components}
  Let $(W,\omega)$ be a connected compact Hamiltonian $\SS^1$-manifold
  with convex contact type boundary, and let $H\colon W\to \RR$ be an
  associated Hamiltonian function.

  \medskip
  
  \textbf{(a)} Then it follows that the set of critical points
  $\Crit(H)$ is equal to the set of fixed points $\Fix(\SS^1)$ of the
  circle action.  These decompose into finitely many connected
  components
  \begin{equation*}
    \Fix(\SS^1) = \bigsqcup_j C_j \;,
  \end{equation*}
  that intersects $\p W$ transversely.  A component~$C_j$ of
  $\Fix(\SS^1)$ is a closed symplectic submanifold if and only if it
  does not intersect the boundary of $W$, otherwise $C_j$ is a compact
  symplectic manifold with convex contact type
  boundary~$\p C_j = C_j \cap \p W$.

  \medskip
  
  \textbf{(b)} The choice of an $\SS^1$-invariant Liouville vector
  field along the boundary defines an invariant contact
  structure~$\xi$ on $\p W$.  Decompose the boundary of $W$ into the
  three subsets
  \begin{align*}
    \xi^+ &= \bigl\{p\in \p W\bigm|
            \text{ $X(p)$ is positively transverse to $\xi_p$}\bigr\} \;, \\
    \xi^- &= \bigl\{p\in \p W\bigm|
            \text{ $X(p)$ is negatively transverse to $\xi_p$}\bigr\} \text{ and }\\
    \xi^0 &= \bigl\{p\in \p W\bigm| \; X(p) \in \xi_p\bigr\} \;.
  \end{align*}
  If $H' = \lambda(X)$ is the Hamiltonian function of $X$ associated
  via \Cref{lemma: hamiltonian function through liouville form} to the
  Liouville form~$\lambda := \iota_Y\omega$, then
  $\xi^+ = \{H' > 0\}$, $\xi^- = \{H' < 0\}$, $\xi^0 = \{H' = 0\}$.

  The closed subset~$\xi^0$ is composed of all fixed points in
  $\Fix(\SS^1) \cap \p W$ and all non-trivial isotropic $\SS^1$-orbits
  in $\p W$.  The fixed points form a finite collection of closed
  contact submanifolds.  The non-trivial isotropic orbits form a
  finite union of cooriented disjoint hypersurfaces in $\p W$ that
  separate $\xi^-$ on one side from $\xi^+$ on the other side, so that
  $\xi^0$ is nowhere dense in $\p W$.

  The hypersurfaces of non-trivial isotropic orbits do not need to be
  compact as their closure may contain fixed points.  If this is the
  case, we can also not expect that the closure of the hypersurface is
  a smooth submanifold.

  \medskip
  
  \textbf{(c)} The function~$H$ is Morse-Bott, and all the Morse-Bott
  indices of the different components of $\Crit(H)$ are even.
  
  Let $C_j$ be a fixed point component that intersects $\p W$.  It
  follows that $\p C_j = C_j \cap \p W$ necessarily lies in $\xi^0$.
  It is surrounded in $\p W$ by $\xi^+$, if and only if $C_j$ is a
  local minimum of $H$.  Similarly, $\p C_j$ is surrounded by $\xi^-$,
  if and only if $C_j$ is a local maximum of $H$.

  \smallskip
  
  All conclusions of \Cref{topology simple if no codim 1 stable
    manifolds} can be applied to $W$ and $H$ so that one of the
  following mutually exclusive statements holds:
  \begin{itemize}
  \item The set of critical points~$\Crit(H)$ has a unique
    component~$C_{\max}$ that is a local maximum (a unique
    component~$C_{\min}$ that is a local minimum) and $H$ is
    everywhere else on $W\setminus C_{\max}$ strictly smaller than on
    $C_{\max}$ (everywhere else on $W\setminus C_{\min}$ strictly
    larger than on $C_{\min}$), so that this local maximum (local
    minimum) is actually the global one.

    The subset~$\xi^+$ is empty ($\xi^-$ is empty).  The inclusion
    $C_{\max} \hookrightarrow W$ ($C_{\min} \hookrightarrow W$)
    induces a \emph{surjective} homomorphism
    $\pi_1(C_{\max}) \to \pi_1(W)$ ($\pi_1(C_{\min}) \to \pi_1(W)$).
  \item None of the components of $\Crit(H)$ is a local maximum (local
    minimum), and $H$ takes its global maximum on $\xi^+ \subset \p W$
    (its global minimum on $\xi^-$).

    The subset~$\xi^+$ ($\xi^-$) is open, non-empty and connected.
    The inclusion $\xi^+ \hookrightarrow W$ (or
    $\xi^- \hookrightarrow W$) induces a \emph{surjective}
    homomorphism $\pi_1(\xi^+) \to \pi_1(W)$ (or
    $\pi_1(\xi^-) \to \pi_1(W)$).
  \end{itemize}
\end{theorem_fixed_points_boundary_components}
\begin{proof}
  Let $Y$ be an $\SS^1$-invariant Liouville vector field defined in
  the neighborhood of the boundary pointing traversely out of $W$, and
  let $\lambda = \iota_Y \omega$ be the associated Liouville form.
  Choose the metric whose existence is guaranteed by \Cref{boundary is
    convex with respect to gradient} so that the boundary of $W$ is
  convex with respect to the gradient vector field~$\nabla H$.
  
  \medskip

  (a) Since $\iota_X\omega = - dH$, it is clear that $X$ only vanishes
  at the critical points of $H$ so that $\Crit(H) = \Fix(\SS^1)$.  For
  any choice of an $\SS^1$-invariant compatible almost complex
  structure, it is not hard to see that the fixed point set will be an
  almost complex submanifold.  In particular, it will thus be
  symplectic.  Furthermore since the Liouville vector field is tangent
  to $\Fix(\SS^1)$, the fixed point set is transverse to $\p W$, and
  $\Fix(\SS^1)\cap \p W$ is the contact type boundary of the
  components of $\Fix(\SS^1)$.
  
  \medskip

  (b) Recall that we can decompose $\p W$ with respect to the gradient
  of $H$ into $\p^0 W$, $\p^+W$, and $\p^-W$, the subsets of the
  boundary along which $\nabla H$ is tangent to $\p W$, or points
  transversely out of or into $W$.  As shown in \Cref{boundary is
    convex with respect to gradient}.(b), we can also characterize
  $\p^0 W$, $\p^+W$, and $\p^-W$ by the sign of $H'$.  Using that
  $H' = \lambda(X) = \alpha(X)$, verifies directly that $\xi^+$ agrees
  with $\p^+W$, $\xi^-$ agrees with $\p^-W$, and $\xi^0$ agrees with
  $\p^0W$.

  \smallskip

  Let $\alpha$ be the contact form
  $\alpha = \restricted{\lambda}{T\p W}$.  Along
  $\xi^0 = \bigl\{p \in \p W\bigm|\, H'(p) = 0\bigr\}$, we compute
  $\restricted{dH'}{T\p W} = d\bigl(\alpha(X)\bigr) = \lie{X} \alpha -
  \iota_X d\alpha = - \iota_X d\alpha$.  If $p \in \xi^0$ is not a
  fixed point so that $X(p) \ne 0$, then it follows that $dH_p'\ne 0$,
  because $d\alpha$ is a symplectic form on $\ker \alpha$.  This
  implies that $\xi^0 = (H')^{-1}(0)$ is a regular hypersurface close
  to $p$ separating $\xi^-$ on one side from $\xi^+$ on the other
  side, defining a coorientation.

  We already explained in (a) that
  $\Fix(\SS^1)\cap \p W \subset \xi^0$ are the contact type boundaries
  of the fixed point components of $W$.  That $\xi^0$ is usually not a
  collection of smooth closed submanifolds is shown in \Cref{circle
    action on 6-ball}.

  \medskip

  (c) With the choice of an $\SS^1$-invariant almost complex
  structure, the proof that the Hamiltonian function of a circle
  action is Morse-Bott reduces to a local study \cite{Frankel} that
  applies equally well to critical points on the boundary by
  considering the collar neighborhood from \Cref{collar neighborhood}.
  Being a self-adjoint operator, the Hessian of $H$ only has real
  eigenvalues.  The spaces of its eigenvectors for positive/negative
  eigenvalues are complex, and thus it follows that the
  indices~$i^-(C_j)$ and $i^+(C_j)$ for every component
  $C_j \subset \Crit(H)$ are even.  In particular, we see that
  $i^-(C_j)$ and $i^+(C_j)$ are always different from $1$.

  Let us now study the intersection between a local extremum of $H$
  and the boundary.  Assume that $C_j \subset \Fix(\SS^1)$ is a
  component that intersects $\p W$, then $H'$ clearly vanishes along
  $\p C_j = C_j \cap \p W$ so that $\p C_j\subset \xi^0$.  If $\p C_j$
  is surrounded in $\p W$ by $\xi^+$, then there is a neighborhood of
  $\p C_j$ in $\p W$ where $H'$ is strictly positive except for
  $\p C_j$ itself.  This implies that $H'$ and $\restricted{H}{\p W}$
  (which agrees up to the addition of a constant with $H'$) have a
  local minimum along $\p C_j$, and by \Cref{indices of boundary crit
    in relation to normal index}, $C_j$ is also a local minimum of $H$
  in $W$.

  Conversely, if the points in $C_j \subset \Fix(\SS^1)$ are local
  minima of $H$, and if $C_j$ intersects $\p W$, then it follows that
  $\p C_j$ will be composed of local minima for $\restricted{H}{\p W}$
  and thus also for $H'$.  This implies that $\p C_j$ is surrounded by
  $\xi^+$.
  
  The argument for $\xi^-$ and local maxima is identical after
  changing the sign of $H$.

  \smallskip

  To finish the proof of (c), we only need to apply \Cref{topology
    simple if no codim 1 stable manifolds}.  All necessary conditions
  are satisfied: In particular, $\xi^+$ agrees with $\p^+ W$, and
  $\xi^-$ with $\p^- W$ as we have shown in (b).  This implies that if
  $C_j\subset \Crit (H)$ is a local maximum of $H$ that intersects
  $\p W$, then $\p C_j$ will be surrounded in $\p W$ by $\p^- W$ as
  required in \Cref{topology simple if no codim 1 stable manifolds}.
  The remaining claims follow then directly from \Cref{topology simple
    if no codim 1 stable manifolds}.
\end{proof}

\begin{example}\label{circle action on 6-ball}
  Let $W$ be the unit ball in $\CC^3$ and define a circle action on
  $W$ by
  \begin{equation*}
    e^{i\varphi} \, (z_0, z_1,z_2) :=
    \bigl(z_0, e^{i\varphi}\,z_1, e^{i k\varphi}\, z_2\bigr) \;,
  \end{equation*}
  where $k$ is either $+1$ or $-1$.
  
  The action preserves both the standard symplectic form
  $\omega_0 = \frac{i}{2}\, \sum_{j=1}^3 dz_j\wedge d\bar z_j$, and
  the contact structure~$\xi$ given on $\p W$ as the kernel of the
  $1$-form
  $\lambda = \frac{i}{4}\, \sum_{j=1}^3 ( z_j\, d\bar z_j - \bar z_j\,
  dz_j)$.  The Hamiltonian function for the circle action is
  \begin{equation*}
    H(z_0,z_1,z_2) = \tfrac{1}{2}\,\abs{z_1}^2
    + \tfrac{k}{2}\,\abs{z_2}^2 \; ,
  \end{equation*}
  and the set of all fixed points is
  $\Fix(\SS^1) = \bigl\{(z_0,0,0) |\, \abs{z_0} \le 1\bigr\}$.

  \begin{itemize}
  \item If $k = 1$, then $H\ge 0$ on all of $W$, and $\Fix(\SS^1)$ is
    the global minimum of $H$.  The $\SS^1$-orbits are everywhere
    along $\p W$ positively transverse to the contact structure,
    except of course at the fixed point set.  Thus
    $\xi^- = \emptyset$, and
    $\xi^0 = \bigl\{(e^{i\vartheta}, 0, 0)\bigr\}$ is an isolated
    circle inside $\xi^+ = \p W \setminus \xi^0$.
  \item If $k = -1$, then the boundary decomposes into the two
    connected subsets $\xi^- = \bigl\{\abs{z_1} < \abs{z_2}\bigr\}$
    and $\xi^+ = \bigl\{\abs{z_1} > \abs{z_2}\bigr\}$ along which the
    $\SS^1$-orbits are either negatively or positively transverse to
    the contact structure.  These two subsets are separated by
    $\xi^0$.

    In this case, $\xi^0$ is composed of the loop
    $\bigl\{(e^{i\vartheta}, 0,0)\bigr\}$ that lies in
    $\Fix(\SS^1) \cap \p W$, and of the hypersurface
    \begin{equation*}
      \Bigl\{\bigl(z_0, \tfrac{r}{\sqrt{2}} \, e^{i\varphi_1},
      \tfrac{r}{\sqrt{2}}\,e^{i\varphi_2}\bigr)\Bigm|
      r = \sqrt{1 - \abs{z_0}^2}, \text{ and } \abs{z_0} < 1  \Bigr\} \;.
    \end{equation*}
    that is composed of (nontrivial) isotropic orbits.  This
    hypersurface is thus diffeomorphic to $\DD^2 \times \TT^2$ with
    the anti-diagonal $\SS^1$-action on the $\TT^2$-factor.
    
    To understand the shape of the subset~$\xi^0$, we can take
    $\overline{\DD}^2 \times \TT^2$ and collapse the tori lying over
    the boundary of the disk, that is, the neighborhood of
    $\bigl\{(e^{i\vartheta}, 0,0)\bigr\}$ in $\xi_0$ looks like a
    quotient space
    $\bigl((1-\epsilon, 1]\times \SS^1\bigr) \times \TT^2/\sim$ by the
    equivalence relation
    \begin{equation*}
      \bigl(1,e^{i\vartheta}; e^{i\varphi_1}, e^{i\varphi_2}\bigr) \sim
      \bigl(1,e^{i\vartheta}; e^{i\varphi_1'}, e^{i\varphi_2'}\bigr) \;.
    \end{equation*}

    To see that $\xi_0$ is not a manifold, we prefer to identify
    $\bigl((1-\epsilon, 1]\times \SS^1\bigr) \times \TT^2$ with
    $\SS^1\times \bigl((1-\epsilon, 1]\times \TT^2\bigr)$, that is, we
    think about this model as an $\SS^1$-family of thickened tori.
    The $\SS^1$-factor is not relevant, so we will ignore it from now
    on.  We collapse $\{1\}\times \TT^2$ in two steps: In the first
    one we quotient $(1-\epsilon, 1]\times \TT^2$ by
    $\bigl(1, e^{i\varphi_1}, e^{i\varphi_2}\bigr) \sim \bigl(1,
    e^{i\varphi_1'}, e^{i\varphi_2}\bigr)$, which allows us to
    identify $(1-\epsilon, 1]\times \TT^2/\sim$ smoothly with the
    solid torus $\DD^2_{<\epsilon} \times \SS^1$.  In the second step
    we collapse the core $\{0\}\times \SS^1$ of the solid torus to a
    point.

    In order to see that this quotient space is not a manifold,
    consider an arbitrarily small neighborhood of the core in the
    solid torus.  Then after removing the core, the neighborhood will
    always contain an incompressible torus.  When passing to the
    quotient space, the core reduces to a point, and thus we have just
    proven that none of the neighborhoods of this point is simply
    connected after removing the point.  This behavior is in stark
    contrast to that of a genuine $3$-dimensional manifold, where one
    can take a ball around a point and remove its center, to find a
    simply connected neighborhood.
  \end{itemize}
\end{example}

\begin{corollary}
  Let $(W,\omega)$ be a connected compact Hamiltonian $\SS^1$-manifold
  with convex contact type boundary.  If any of the fixed point
  components has codimension~$2$ then it is necessarily consists
  either of local maxima or minima of the Hamiltonian function.
\end{corollary}
\begin{proof}
  Since all indices are even, if $C_j\subset \Fix(\SS^1)$ is of
  codimension~$2$, we either have $i^+(C_j) = 2$ and $i^-(C_j) = 0$ so
  that $C_j$ is a local minimum, or $i^+(C_j) = 0$ and $i^-(C_j) = 2$
  so that $C_j$ is a local maximum.
\end{proof}

\begin{corollary}\label{at most two boundary components}
  Let $(W,\omega)$ be a connected compact Hamiltonian $\SS^1$-manifold
  with convex contact type boundary, then $W$ can have at most two
  boundary components.
\end{corollary}
\begin{proof}
  Choose an invariant Liouville field close to the boundary of $W$ to
  define a contact structure on $\p W$.  \Cref{theorem: fixed points
    and boundary components of hamiltonian S1-manifolds} shows that
  $\p W$ can be partitioned into $\xi^-\sqcup \xi^0 \sqcup \xi^+$,
  where $\xi^-$ and $\xi^+$ are open subsets that are either connected
  or empty, and $\xi^0$ is a thin subset.

  Thus, either $\p W = \emptyset$, or only one of $\xi^-$ and $\xi^+$
  is non-empty and $\p W$ is the closure of this subset and thus
  connected, or if both $\xi^-$ and $\xi^+$ are non-empty, then either
  both of them lie in the same boundary component so that $\p W$ is
  connected, or each of them lies in a different component of $\p W$
  so that there are two boundary components.
\end{proof}

\begin{remark}
  As we have just proven, the number of boundary components for
  Hamiltonian $\SS^1$-manifolds with contact type boundary is at most
  two.  Except for dimension~$2$, where the cylinder
  $\SS^1\times [-1,1]$ is an obvious example, we do not know of any
  other example with disconnected boundary.  We show in \Cref{sec:
    dimension 4} that no such example exists in dimension~$4$, but it
  is unknown to us if there are examples in dimension$>4$.
  
  Note that even ignoring any circle action, it is already far from
  trivial to find symplectic manifolds with disconnected contact
  boundary, see for example \cite{McDuff_contactType,
    Geiges_disconnected, WeakFillabilityHigherDimension}.
\end{remark}

We will now do a case by case analysis of the different scenarios that
can occur depending on the existence of local extrema and the number
of boundary components.

\begin{corollary}[Empty boundary]\label{max and min and no boundary}
  Let $(W,\omega)$ be a \emph{closed} connected Hamiltonian
  $\SS^1$-manifold with Hamiltonian function $H\colon W\to \RR$.

  Clearly, $H$ has a global minimum and a global maximum that will be
  attained on a subset~$C_{\min}$ and $C_{\max}$ respectively.  It is
  well-known \cite{AtiyahConvexity, GuilleminSternbergConvexity} that
  these subsets are connected components of $\Crit(H)$, and that no
  other component of $\Crit(H)$ can be a local minimum or maximum.
  Furthermore, the natural homomorphisms
  $\pi_1(C_{\min}) \to \pi_1(W)$ and $\pi_1(C_{\max}) \to \pi_1(W)$
  are both surjective.
\end{corollary}
\begin{proof}
  The statement follows directly from \Cref{theorem: fixed points and
    boundary components of hamiltonian S1-manifolds}.
\end{proof}

\begin{corollary}[Non-empty boundary and a local minimum]\label{local
    min and boundary}
  Let $(W,\omega)$ be a connected compact Hamiltonian $\SS^1$-manifold
  with contact type boundary, and let $H\colon W\to \RR$ be a
  Hamiltonian function for the $\SS^1$-action.  Assume that $\Crit(H)$
  has a component~$C_{\min}$ that is composed of local minima.

  \smallskip
  
  Then it follows that $H$ takes its \emph{global} minimum on
  $C_{\min}$, and that $H$ is everywhere else on $W\setminus C_{\min}$
  strictly larger.  None of the other components of $\Crit(H)$ can be
  a local minimum or maximum.  The boundary of $W$ is connected, and
  the only component of $\Crit(H)$ that may intersect $\p W$ is
  $C_{\min}$.

  Choose an invariant Liouville field close to the boundary of $W$ to
  define a contact structure~$\xi$ on $\p W$.  All $\SS^1$-orbits in
  the boundary are positively transverse to the contact structure,
  except for the fixed points in $C_{\min}\cap \p W$.  If $C_{\min}$
  does not intersect $\p W$, then $(\p W, \xi)$ is contactomorphic to
  the prequantization of a symplectic orbifold.

  We can choose $H$ to agree in a neighborhood of $\p W$ with the
  natural Hamiltonian function induced by the Liouville form.

  Furthermore, the natural homomorphisms
  $\pi_1(C_{\min}) \to \pi_1(W)$ and $\pi_1(\xi^+) \to \pi_1(W)$ are
  surjective.
\end{corollary}
\begin{proof}
  Decompose $\p W$ with respect to an invariant contact structure into
  $\p W = \xi^-\sqcup \xi^0 \sqcup \xi^+$.  It follows from
  \Cref{theorem: fixed points and boundary components of hamiltonian
    S1-manifolds} that $C_{\min}$ is connected and $\xi^-$ is empty.
  Furthermore, since $\p W\ne \emptyset$ and $\xi^0$ is not dense in
  $\p W$, we deduce that $\xi^+$ is non-empty and connected.  This
  implies that $C_{\max} = \emptyset$.  Finally, since $\p W$ is the
  closure of $\xi^+$, the boundary is also connected.

  The Hamiltonian function~$H$ is unique up to addition of a constant.
  Here $\p W$ is connected, thus we can assume that $H$ agrees close
  to the boundary with $\lambda(X)$, where $X$ is the infinitesimal
  generator of the circle action and $\lambda$ is the Liouville form.

  Recall that $\xi^0$ is by \Cref{theorem: fixed points and boundary
    components of hamiltonian S1-manifolds}.(b) the union of boundary
  fixed points and of non-trivial isotropic orbits.  The non-trivial
  orbits lie in hypersurfaces that touch on one side $\xi^-$ and on
  the other one $\xi^+$, but since $\xi^-$ is empty, there cannot be
  any non-trivial isotropic orbits, and thus
  $\xi^0 = \Crit(H)\cap \p W$.  Obviously, every fixed point component
  in $\p W$ will be surrounded by $\xi^+$ so that $\xi^0$ can only be
  $\xi^0 = C_{\min} \cap \p W$ as we wanted to show.  If $C_{\min}$
  does not intersect the boundary, then the $\SS^1$-orbits in $\p W$
  are everywhere positively transverse to $\xi$ and $(\p W, \xi)$ is
  the prequantization of a symplectic orbifold.

  It was shown in \Cref{theorem: fixed points and boundary components
    of hamiltonian S1-manifolds}.(c) that $H$ takes its global minimum
  on $C_{\min}$ and is everywhere else strictly larger than on
  $C_{\min}$.
\end{proof}

The corresponding corollary about Hamiltonian $\SS^1$-manifolds with
non-empty contact type boundary and a fixed point component that is a
local maximum of the Hamiltonian function, can be deduced from the
previous corollary simply by inverting the orientation of the
$\SS^1$-orbits and the sign of the Hamiltonian function.  Note that if
there is no fixed point on the boundary, then the boundary will in
this case be contactomorphic to a prequantization with inverted circle
action.

\begin{corollary}[No local extrema]\label{no extrema but boundary}
  Let $(W,\omega)$ be a connected compact Hamiltonian $\SS^1$-manifold
  with contact type boundary, and let $H\colon W\to \RR$ be a
  Hamiltonian function for the $\SS^1$-action.  Assume that none of
  the components of $\Crit(H)$ is a local minimum or maximum.  Choose
  on a neighborhood of the boundary of $W$ an invariant Liouville
  field to define a contact structure~$\xi$ on $\p W$.
  
  \smallskip

  Then it follows that $W$ has either one or two boundary components.
  \begin{itemize}
  \item Assume that $\p W$ has two components: Then, all fixed points
    lie in the interior of $W$, and the $\SS^1$-orbits are everywhere
    along the boundary transverse to the contact structure.  On one of
    the boundary components the $\SS^1$-orbits will be positively
    transverse to the contact structure, so that the boundary
    component will be the prequantization of a symplectic orbifold; on
    the other boundary component the $\SS^1$-orbits will be negatively
    transverse to the contact structure, and the boundary component
    will be the prequantization of a symplectic orbifold for the
    inverted circle action.
  \item If $\p W$ is connected, then it decomposes as
    $\p W = \xi^-\sqcup \xi^0 \sqcup \xi^+$, where $\xi^-$ and $\xi^+$
    are non-empty and connected and $\xi^0$ necessarily contains
    non-trivial isotropic orbits, but also possibly fixed points all
    of which need to lie in the closure of the isotropic orbits.  We
    can choose $H$ to agree in a neighborhood of $\p W$ with the
    natural Hamiltonian function induced by the Liouville form so that
    $H>0$ on $\xi^+$ and $H< 0$ on $\xi^-$.
  \end{itemize}

  Furthermore, the natural homomorphisms $\pi_1(\xi^-) \to \pi_1(W)$
  and $\pi_1(\xi^+) \to \pi_1(W)$ are in both cases surjective.
\end{corollary}
\begin{proof}
  As explained in \Cref{at most two boundary components}, the
  boundary~$\p W$ can have at most two components.  If $\p W$ is
  disconnected, then $\xi^+$ and $\xi^-$ need to lie in different
  components.  This implies that $\xi^0$ does not contain any
  isotropic orbits, because these lie always in hypersurfaces that
  have $\xi^+$ on one and $\xi^-$ on the other side.  If there were
  any fixed points in the boundary, these would necessarily either be
  surrounded only by $\xi^+$ or only by $\xi^-$.  According to
  \Cref{theorem: fixed points and boundary components of hamiltonian
    S1-manifolds}.(c) such a fixed point component is a local maximum
  or a local minimum which is both in contradiction to our
  assumptions.  It follows that the circle action is everywhere along
  the boundary transverse to $\xi$.

  \smallskip

  If $\p W$ is connected, then we can choose $H$ to agree close to
  $\p W$ with the Hamiltonian function induced by the Liouville field,
  and it is clear that $\xi^+$ and $\xi^-$ can then be recovered from
  the sign of $H$.  Clearly $\xi^0$ cannot be empty, because $H$ needs
  to vanish somewhere on the boundary.  Furthermore, since
  $\Crit(H)\cap \p W$ is a finite union of contact submanifolds of
  codimension at least~$2$, there need to be non-trivial isotropic
  orbits in $\xi^0$.  By our assumption, $\Crit(H)$ does not contain
  any local extrema and so it follows that all boundary fixed points
  need to lie in the closure of the set of non-trivial isotropic
  orbits.

  \smallskip
  
  The statement about the fundamental groups follows directly from
  \Cref{theorem: fixed points and boundary components of hamiltonian
    S1-manifolds}.(c).
\end{proof}

It is easy to show that a connected compact Hamiltonian
$\TT^k$-manifold with $k\ge 2$ has never disconnected boundary.

\begin{theorem_rank_two_connected_boundary}
  Let $G$ be a compact Lie group of rank at least~$2$.  If
  $(W,\omega)$ is a connected compact Hamiltonian $G$-manifold with
  convex contact type boundary, then it follows that $\p W$ cannot be
  disconnected.
\end{theorem_rank_two_connected_boundary}
\begin{proof}
  Let $X,Y\in \gfrak$ be two elements such that $S_X := \exp (\RR X)$
  and $S_Y := \exp (\RR Y)$ are isomorphic to $\SS^1$, and such that
  $X,Y$ generate together a $2$-torus in $G$.  If $W$ has two boundary
  components, then it follows by \Cref{no extrema but boundary} that
  the orbits of $S_X$ and $S_Y$ are everywhere along the boundary
  transverse to the contact structure.

  Let $\alpha$ be a contact form for the contact structure, and define
  a smooth function $f\colon \p W \to \RR$ by
  $f := -\alpha(X_W) / \alpha(Y_W)$.  It follows that
  $X_W(p) + f(p) Y_W(p)$ lies inside the contact structure.  If there
  is a point $p_0\in \p W$ such that $c = f(p_0)$ is rational, then
  $Z := X + c Y \in \gfrak$ generates a circle~$S_Z$ in $G$.  The
  $S_Z$-action cannot be trivial, because $G$ acts effectively, so
  that $(W,\omega)$ is a Hamiltonian $S_Z$-manifold with disconnected
  boundary.  The $S_Z$-action has either a fixed point in
  $p_0\in \p W$ or the orbit through $p_0$ is tangent to the contact
  structure.  Both possibilities contradict \Cref{no extrema but
    boundary}, and we obtain that $f$ can never take a rational value.
  
  The continuity of $f$ implies then that $f$ has to be constant on
  each of the boundary components of $W$.  If we denote the value of
  $f$ on one of the boundary components by $c$, it follows that
  $Z = X + c Y$ induces the vector field~$Z_W$ that lies along one of
  the boundary components in the contact structure.  Using that
  $\alpha(Z_W) = 0$, and $\lie{Z_W}\alpha = 0$ it follows that
  $\iota_{Z_W}d\alpha = 0$ so that $Z_W$ vanishes along the considered
  boundary component.  Since $\exp (\RR Z)$ is dense in the $2$-torus
  generated by $X$ and $Y$, we deduce that all points in the
  considered boundary component are fixed points of $\exp (\RR X)$,
  but this is again a contradiction to \Cref{no extrema but boundary}.

  This shows that $W$ cannot have disconnected boundary.
\end{proof}

\begin{remark}
  According to Lutz \cite{Lutz_invariantes}, the $3$-dimensional
  closed contact manifolds with a free circle action can be understood
  by considering their orbit spaces, which are closed surfaces, and
  marking on them the domains over which the circle action is
  positively or negatively transverse to the contact structure.  (See
  \cite{S1-contact_dim3} for a general classification of
  $3$-dimensional closed contact $\SS^1$-manifolds).

  This decomposition corresponds to the decomposition of the contact
  boundary into $\xi^+$ and $\xi^-$, and we see in particular that all
  examples of Lutz, where $\xi^+$ or $\xi^-$ has more than one
  component cannot be contact type boundaries of a Hamiltonian
  $\SS^1$-manifold.
    
  Of course, it is already well-known that such contact manifolds are
  in many cases not even contact type boundary of general symplectic
  manifolds without any type of $\SS^1$-symmetry, see for example
  \cite{WendlCobordisms} for dimension~$3$ or
  \cite{WeakFillabilityHigherDimension} for higher dimensions.
\end{remark}


\subsection{From symplectic to Hamiltonian group action} \label{sec:
  Hamiltonian circle action}

As mentioned in \Cref{cor: exactness implies Hamiltonian}, a
symplectic action of a compact Lie groups on an exact symplectic
manifolds $(W,\omega = d\lambda)$ is always Hamiltonian.  This covers
in particular cotangent bundles which are probably the most important
example from classical mechanics.

On \emph{closed} manifolds, it is easy to find non-Hamiltonian
symplectic actions: consider for example the $2$-torus with the
standard rotation.  A necessary condition for a symplectic circle
action on a closed manifold to be Hamiltonian is the existence of
fixed points.  In dimension~$4$, it was shown by McDuff
\cite{McDuffFixedPointsHamiltonianGroupsActions} that the existence of
a fixed point is sufficient for a symplectic circle action to be
Hamiltonian.  On the other hand, she also constructed closed
symplectic $\SS^1$-manifolds in any dimension$\ge 6$ that do have
fixed points but that are nonetheless non Hamiltonian.

\smallskip

In contrast to the closed case, a symplectic $G$-manifold with contact
type boundary is always Hamiltonian.  We will first prove this result
for symplectic circle actions and then generalize it to actions of
general compact Lie groups.  McDuff had shown that a closed
\cite{McDuffFixedPointsHamiltonianGroupsActions} symplectic manifold
with a symplectic circle action is Hamiltonian if and only if it has a
fixed point that looks like a maximum or minimum of a local
Hamiltonian function.  Her result was the main motivation for our
question, and we briefly reprove it here.  (In \Cref{sec: dimension
  4}, we also reprove that a symplectic circle action on a closed
$4$-dimensional manifold is Hamiltonian if and only if it has fixed
points.)

\begin{theorem}\label{theorem: boundary+symplectic is Hamiltonian}
  Every symplectic circle action on a compact symplectic manifold with
  non-empty contact type boundary is Hamiltonian.
\end{theorem}
\begin{proof}
  Let $(W,\omega)$ be a symplectic $\SS^1$-manifold, and denote the
  infinitesimal generator of the circle action by $X$.  The $1$-form
  $\eta := \iota_X\omega$ is closed, and we will deduce from
  \Cref{theorem:exactness with boundary} that $\eta$ is exact so that
  the circle action is Hamiltonian.  We only need to verify that all
  the conditions of the theorem are satisfied.

  Note that $\eta$ restricts on a small neighborhood of the critical
  points to an exact $1$-form so that it has a local primitive, and
  this primitive is a Morse-Bott function with only even indices,
  because the corresponding proof for the Hamiltonian case is purely
  local.

  Let $Y$ be an $\SS^1$-invariant Liouville vector field on a
  neighborhood of $\p W$.  Denote the corresponding Liouville form by
  $\lambda_Y$, then it follows by \Cref{lemma: hamiltonian function
    through liouville form} that $f := \lambda_Y(X)$ is a primitive
  for $\eta$ on the neighborhood of $\p W$.

  Choosing a Riemannian metric as in \Cref{boundary is convex with
    respect to gradient}, it follows that the boundary of $W$ is
  convex with respect to the gradient of the local primitive~$f$.
  According to \Cref{boundary is convex with respect to gradient}.(c),
  any component of $\Crit(\eta)\cap \p W$ that is composed of local
  minima / local maxima admits a neighborhood~$U \subset \p W$ such
  that the gradient field does never point along $U$ traversely into
  $W$ / transversely out of $W$.  All conditions of
  \Cref{theorem:exactness with boundary} are satisfied, and it follows
  that $\eta$ is exact proving that the action is Hamiltonian.
\end{proof}

\begin{theorem_symplectic_action_always_hamiltonian}
  Let $G$ be a compact Lie group that acts symplectically on a compact
  symplectic manifold with non-empty contact type boundary.  Then it
  follows that the action is Hamiltonian.
\end{theorem_symplectic_action_always_hamiltonian}
\begin{proof}
  Let $G$ be a compact Lie group that acts symplectically on
  $(W,\omega)$, a symplectic manifold with a contact type boundary.
  We first show that this action is weakly Hamiltonian, that is, for
  every $X\in\gfrak$ there is a Hamiltonian function
  $H_X\colon W\to\RR$ such that $\iota_{X_W}\omega = -dH_X$.

  Take any $X\in\gfrak$.  If the flow of $X$ generates a circle
  action, then the claim follows directly from the previous theorem.
  In general however, the flow of $X$ is not cyclic and only generates
  an $\RR$-action.  Nonetheless, the closure of
  $\exp\bigl(\RR X\bigr)$ inside $G$ is a connected abelian subgroup,
  that is, a torus (\cite[Chapter~4]{BrockerDieck}).  This torus is
  the product of $k$ circles, each of which generates a symplectic
  circle action which again is Hamiltonian by \Cref{theorem:
    boundary+symplectic is Hamiltonian}.  Since $X_W$ is a linear
  combination of these $k$ Hamiltonian vector fields, it follows that
  $X$ also admits a Hamiltonian function.

  \smallskip
  
  This shows that the $G$-action is weakly Hamiltonian.  That the
  action is actually Hamiltonian, that is, it admits a $G$-equivariant
  moment map, follows then from \Cref{lemma: if weakly Ham and Ham on
    subset then Ham everywhere} by using that the Liouville form close
  to one of the boundary components defines a local moment map that is
  $G$-equivariant, see \Cref{lemma: hamiltonian function through
    liouville form}.(b).
\end{proof}

\Cref{theorem:exactness with boundary} can also be applied to obtain
the following result due to McDuff.  Our proof is not essentially
different, but much more detailed.  Note that every symplectic action
admits on a neighborhood of a fixed point a local Hamiltonian
function.

\begin{theorem}[McDuff,
  \cite{McDuffFixedPointsHamiltonianGroupsActions}]\label{theorem:
    maximum+symplectic is Hamiltonian}
  A symplectic circle action on a closed manifold is Hamiltonian if
  and only if it has a fixed point that is a local maximum or minimum
  of a locally defined Hamiltonian function.
\end{theorem}
\begin{proof}
  Let $X_W$ be the infinitesimal generator of the circle action.  Its
  contraction with the symplectic form yields a closed $1$-form, and
  it is easy to convince oneself that only minor modifications to the
  proof of \Cref{theorem: boundary+symplectic is Hamiltonian} are
  necessary to apply \Cref{theorem:exactness with boundary} also in
  this case.
\end{proof}

\subsection{Symplectic and Hamiltonian $\SS^1$-manifolds in
  dimension~$4$}\label{sec: dimension 4}

In this section, we will prove two results concerning manifolds of
dimension~$4$.  The first one is well-known and originally due to
McDuff and states that a closed \emph{symplectic} $\SS^1$-manifold of
dimension~$4$ is Hamiltonian if and only if the action has fixed
points.  Our proof follows roughly the same line as the one in
\cite{McDuffFixedPointsHamiltonianGroupsActions}, but is a more
``modular''.

The second result states that a $4$-dimensional compact Hamiltonian
$\SS^1$-manifolds cannot have disconnected contact type boundary.  We
do not know the answer to the following question.

\begin{question}
  Do there exist connected compact Hamiltonian $\SS^1$-manifolds of
  dimension$\ge 6$ that have disconnected contact type boundary?
\end{question}

Assume that $W$ is an $\SS^1$-manifold without fixed points.  A
\defin{(generalized) connection $1$-form~$A$} on $W$ is an
$\SS^1$-invariant $1$-form such that $A(X_W) = 1$ where $X_W$ denotes
the infinitesimal generator of the circle action.

The set of connection $1$-forms on a fixed point free $\SS^1$-manifold
is convex, and in particular every such manifold does admit a
connection $1$-form: Simply take any $1$-form~$\beta$ satisfying
$\beta(X_W) = 1$, and average it over the group action.  Similarly,
one can also see that if a connection form is already defined on a
compact $\SS^1$-invariant subset then we can extend it to a global
connection form.

\medskip

The proofs in this section are based on Stokes' theorem and the
following lemma.

\begin{lemma}\label{top wedge curvature form vanishes}
  Let $A$ be a generalized connection $1$-form on an
  $\SS^1$-manifold~$W$ that has no fixed points.
  \begin{itemize}
  \item If $W$ is $(2n)$-dimensional, then it follows that the
    product~$(dA)^n$ vanishes everywhere.
  \item If $W$ is $(2n+1)$-dimensional, then it follows that
    $A\wedge (dA)^n$ represents a cohomology class that is independent
    of the choice of the connection $1$-form.
  \end{itemize}
\end{lemma}
\begin{proof}
  In the even dimensional case, choose a basis $v_1,\dotsc,v_{2n}$ for
  a tangent space~$T_xW$ such that $v_1 = X_W(x)$.  Using that
  $\iota_{X_W}dA = \lie{X_W} A - d\iota_{X_W}A = 0$, it follows that
  $(dA)^n(v_1,\dotsc,v_{2n}) = 0$.

  \smallskip

  In the odd dimensional case, choose a second connection
  $1$-form~$A'$.  We can write $A' = A + \beta$ where $\beta$ is an
  $\SS^1$-invariant $1$-form such that $\iota_{X_W} \beta = 0$.  With
  the Cartan formula it follows that $\iota_{X_W} d\beta = 0$, so that
  $\beta\wedge (d\beta)^k \wedge (dA)^{n-k}$ vanishes for any choice
  of $k \le n$.

  The Leibniz rule allows us to write (for $k$ at least $1$)
  \begin{equation*}
    \begin{split}
      d \bigl(A\wedge \beta\wedge (d\beta)^{k-1} \wedge (dA)^{n-k}
      \bigr) &= dA\wedge \beta\wedge (d\beta)^{k-1} \wedge (dA)^{n-k}
      -
      A\wedge(d\beta)^k \wedge (dA)^{n-k} \\
      &= - A\wedge (d\beta)^k \wedge (dA)^{n-k} \;.
    \end{split}
  \end{equation*}
  It follows that $A'\wedge (dA')^n = A\wedge (dA)^n + d\eta$, where
  $\eta$ is a $(2n)$-form as desired.
\end{proof}

Consider for any two positive integers~$k_1,k_2$ with
$\gcd(k_1,k_2) = 1$, the linear action on $\CC^2$ given by
\begin{equation*}
  e^{i\theta}\cdot (z_1,z_2) =
  \bigl(e^{ik_1\theta}\, z_1,  e^{-ik_2\theta}z_2\bigr)
\end{equation*}
for every $(z_1,z_2) \in \CC^2$ and $e^{i\theta} \in \SS^1$.

The restriction of
\begin{equation*}
  \lambda = \frac{1}{k_1}\,\bigl(x_1\,dy_1 - y_1\, dx_1\bigr)
  - \frac{1}{k_2}\, \bigl(x_2\,dy_2 - y_2\,  dx_2\bigr)
\end{equation*}
to the unit sphere is a connection $1$-form.  Note that $d\lambda$ is
a symplectic form that induces the \emph{negative} orientation on
$\CC^2$!  We easily compute
\begin{equation*}
  \int_{\SS^3} \lambda\wedge d\lambda =
  -\frac{4}{k_1k_2} \int_{\DD^4} dx_1\wedge dy_1\wedge dx_2\wedge dy_2 =
  -\frac{4}{k_1k_2}\, \operatorname{Vol}(\DD^4) < 0 \;,
\end{equation*}
and by \Cref{top wedge curvature form vanishes} it follows that every
connection $1$-form on $\CC^2\setminus\{\0\}$ integrates on the
$3$-sphere to this value.
  
\medskip

Let $W$ be an oriented $4$-dimensional $\SS^1$-manifold.  Then we call
a fixed point~$x\in W$ a \defin{fixed point of mixed weights} if there
exist positive integers~$k_1,k_2$ and an \emph{orientation preserving}
diffeomorphism from a neighborhood of $x$ onto the open unit ball in
$\CC^2$ with the circle action given above.

\begin{proposition}\label{not only index 2 fixed points in 4 manifolds}
  \begin{itemize}
  \item [(a)] A $4$-dimensional \emph{closed} oriented
    $\SS^1$-manifold is either fixed point free or it must have at
    least one fixed point that does not have mixed weights.
  \item [(b)] Let $W$ be a $4$-dimensional \emph{compact} oriented
    $\SS^1$-manifold with boundary.  Assume that none of the fixed
    points lies on the boundary, and in case that there are interior
    fixed points assume that they all have mixed weights.  If $A$ is
    any connection $1$-form on $\p W$, then it follows that
    $\int_{\p W} A\wedge dA$ cannot be positive.
  \end{itemize}
\end{proposition}
\begin{proof}
  \textbf{(a)} Assume that $W$ is a closed oriented $\SS^1$-manifold
  with $N$ fixed points, and assume that all of them have mixed
  weights.  Remove around each fixed point an open subset that
  corresponds to the unit ball around the origin in $\CC^2$ in the
  model neighborhood, and denote the complement of these balls by
  $\mathring{W}$.

  Choose now a connection $1$-form~$A$ on $\mathring{W}$.  Then we
  find with Stokes' theorem and \Cref{top wedge curvature form
    vanishes} that
  \begin{equation*}
    \int_{\p \mathring{W}} A\wedge dA =
    \int_{\mathring{W}} dA\wedge dA = 0 \;.
  \end{equation*}
  On the other hand, we explained above that
  \begin{equation*}
    \int_{\p \mathring{W}} A\wedge dA = + \operatorname{Vol}(\DD^4)\cdot 
    \sum_{j=1}^N \frac{1}{k_{j,1} k_{j,2}} > 0 \;,
  \end{equation*}
  where $(k_{j,1},-k_{j,2})$ are the weights of the $j$-th fixed
  point.  We have used here that the boundary orientation of
  $\p \mathring{W}$ is the opposite as the boundary orientation
  induced from the model neighborhood in $\CC^2$.  It follows that $W$
  needs either to be fixed point free or that there is a fixed point
  that does not have mixed weights.

  \medskip

  \textbf{(b)} Assume now that $W$ is a compact $4$-manifold with
  boundary.  Remove around each fixed point an open subset that is
  equivariantly diffeomorphic to the unit ball around the origin in
  $\CC^2$, and denote the remaining part of $W$ by $\mathring{W}$.

  Choose a connection $1$-form~$A$ on $\mathring{W}$.  Then we obtain
  again via Stokes' theorem and \Cref{top wedge curvature form
    vanishes} that
  \begin{equation*}
    \int_{\p \mathring{W}} A\wedge dA =
    \int_{\mathring{W}} dA\wedge dA = 0 \;.
  \end{equation*}
  Since the integral over the boundary of $\mathring{W}$ has to
  vanish, but the contribution from the boundary sphere around each
  fixed point is strictly positive, we deduce that
  $ \int_{\p W} A\wedge dA$ needs to be $0$ (if there are no fixed
  points) or strictly negative otherwise.
\end{proof}

\begin{theorem}[McDuff]
  A symplectic $\SS^1$-action on a closed connected $4$-manifold is
  Hamiltonian if and only if it has fixed points.
\end{theorem}
\begin{proof}
  If the action is Hamiltonian, then there are clearly fixed points.

  Assume now that the symplectic $\SS^1$-manifold has fixed points.
  From the linearization around a fixed point we can read off that
  every fixed point in dimension~$4$ that is neither a local minimum
  nor a local maximum needs to be an isolated fixed point with mixed
  weights.  Since \Cref{not only index 2 fixed points in 4
    manifolds}.(a) excludes that all fixed points have mixed weights,
  it follows that there will either be a local maximum or minimum.
  
  By \Cref{theorem: maximum+symplectic is Hamiltonian} every closed
  symplectic $\SS^1$-manifold is Hamiltonian if it has a fixed point
  that is a local maximum or local minimum.
\end{proof}

\begin{theorem_boundary_connected_dim4}
  A $4$-dimensional compact Hamiltonian $\SS^1$-manifold cannot have
  disconnected contact type boundary.
\end{theorem_boundary_connected_dim4}
\begin{proof}
  Let $(W,\omega)$ be a $4$-dimensional Hamiltonian $\SS^1$-manifold
  with disconnected contact type boundary.  We know from \Cref{no
    extrema but boundary} that $W$ can only have two boundary
  components.  Furthermore, all fixed points have to lie in the
  interior of $W$ and none of them can be a local maximum or a local
  minimum.  This implies in dimension~$4$ that there are only isolated
  fixed points that have mixed weights.  Thus we can apply part~(b) of
  \Cref{not only index 2 fixed points in 4 manifolds}.
  
  Note that the $\SS^1$-action is by \Cref{no extrema but
    boundary} positively transverse to the contact structure on one of
  the boundary components and negatively transverse to the contact
  structure on the other boundary component.  Denote the first type of
  boundary by $(\p_+ W, \xi_+)$ and the second one by
  $(\p_- W, \xi_-)$.  We can find contact forms~$\alpha_+$ and
  $\alpha_-$ for $\xi_+$ and $\xi_-$ respectively that are rescaled in
  such a way that $\alpha_+$ and $-\alpha_-$ are connection $1$-forms
  on $\p W$.

  Remember that by the definition of convex filling,
  $\alpha_+\wedge d\alpha_+$ and $\alpha_-\wedge d\alpha_-$ will be
  positive volume forms on $\p W$ so that
  $\int_{\p_+ W} \alpha_+\wedge d\alpha_+ >0$, and
  $\int_{\p_- W} (-\alpha_-)\wedge (-d\alpha_-) = \int_{\p_- W}
  \alpha_-\wedge d\alpha_- >0$.  Thus we have a contradiction to
  \Cref{not only index 2 fixed points in 4 manifolds}.(b).
\end{proof}

\subsection{Level sets of a Hamiltonian function generating a circle
  action are connected}\label{sec: hamiltonian circle manifold with
  cylindrical ends has connected level sets}


In the previous sections we have shown that many of the Morse-Bott
techniques used for Hamiltonian circle actions continue to work (with
some modifications) for compact symplectic manifolds with contact type
boundaries.  Unfortunately though, it is not true for such manifolds
that all level sets of the Hamiltonian of the $\SS^1$-action are
either connected or empty.  To generalize this classical result due to
\cite{AtiyahConvexity, GuilleminSternbergConvexity} to our case, we
will first attach cylindrical ends to the symplectic manifold.

The reason is that even though the topological properties of a
symplectic manifold are quite stable under perturbations of the
boundary, geometric properties are sensitive to such perturbations,
see the example below.

\begin{example}\label{disconnected level sets}
  Take the $\SS^1$-action on $\CC^2$ defined by
  $e^{i\varphi}\, (z_1,z_2) = \bigl(e^{i\varphi}\,z_1
  ,e^{i\varphi}\,z_2)$ with the standard symplectic structure and with
  the Hamiltonian function
  $H(z_1,z_2) = \frac{1}{2}\, \bigl(\abs{z_1}^2 + \abs{z_2}^2\bigr)$.

  We can consider the closed unit ball
  $W = \bigl\{(z_1,z_2)\in \CC^2\bigm|\; \abs{z_1}^2 + \abs{z_2}^2 \le
  1 \bigr\}$, which is a Hamiltonian $\SS^1$-manifold with contact
  type boundary.  The level sets of $H$ are concentric codimension~$1$
  spheres of varying radii, and thus clearly connected.

  It is easy to perturb the boundary of $W$ to split some of the level
  sets of $H$ into several components.  For example, choose a cut-off
  function $\rho\colon \RR \to [0,1]$ with compact domain in the
  interval $(-\delta, + \delta)$ for $\delta \ll 1$.  Then it follows
  that the compact domain
  \begin{equation*}
    W_\epsilon := \bigl\{(z_1,z_2)\in \CC^2\bigm|\;
    \abs{z_1}^2 + \abs{z_2}^2 \le 1 +
    \epsilon\, \bigl(\rho(\abs{z_1}^2) + \rho(\abs{z_2}^2)\bigr) \bigr\}
  \end{equation*}
  is for $\epsilon>0$ a deformation of the round ball~$W$ that has
  only changed in a neighborhood of the two orbits $\SS^1\times \{0\}$
  and $\{0\}\times \SS^1$.  The boundary of the Hamiltonian
  $\SS^1$-manifold~$W_\epsilon$ is transverse to the Liouville vector
  field and is thus of contact type.

  \begin{figure}[htbp]
    \centering \includegraphics[height=4cm,
    keepaspectratio]{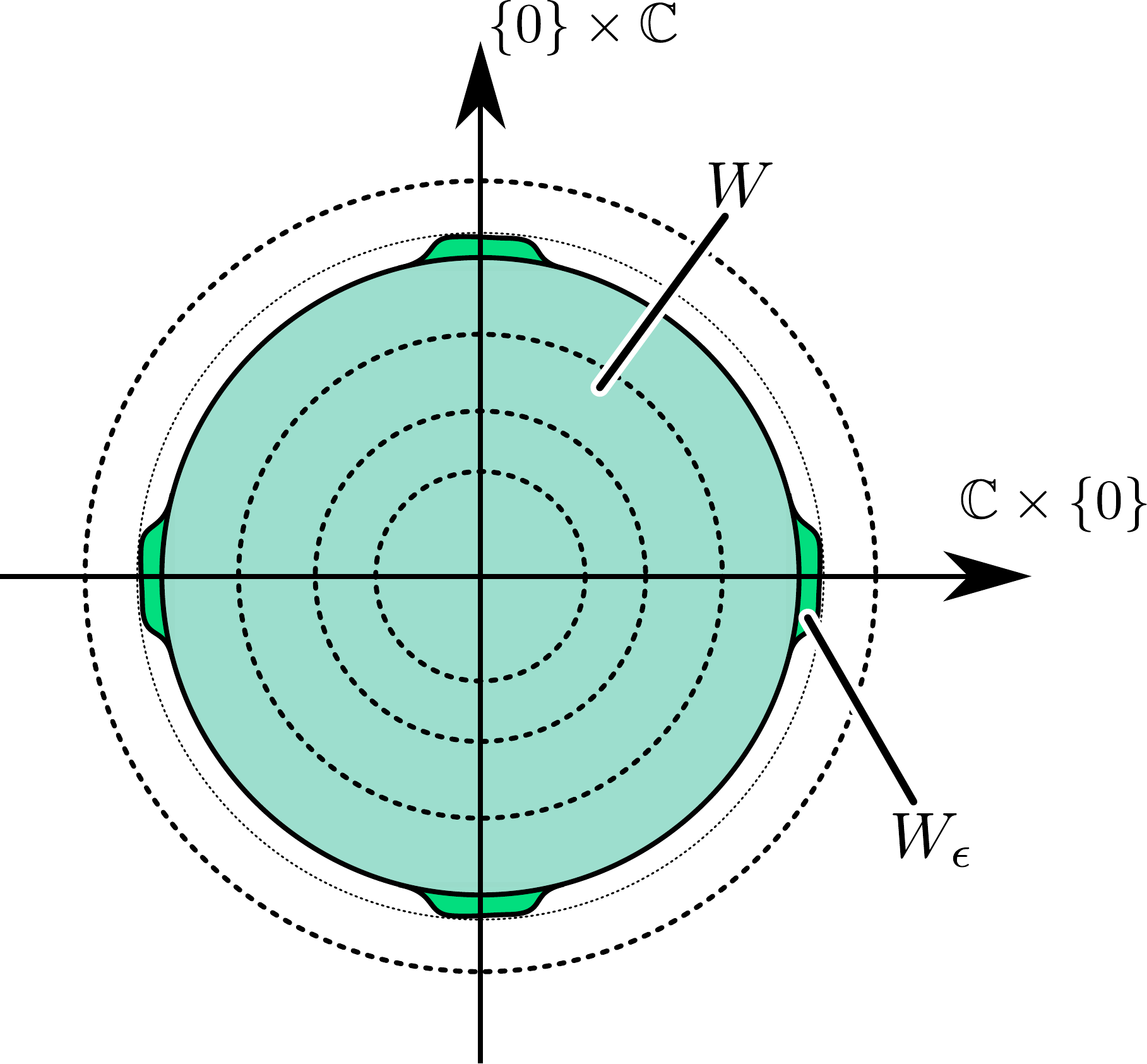}
  \end{figure}

  We see that the level set
  $\bigl\{H = \frac{1}{2}\, (1 + \epsilon) \bigr\}$ is composed of a
  non-connected neighborhood of the two circles
  $\{z_1=0\}\cup \{z_2=0\}$.
\end{example}

As already mentioned above, the connectedness theorem for circle
actions by Atiyah-Guillemin-Sternberg can still be saved by attaching
cylindrical ends to the Hamiltonian manifold.

\begin{theorem_connected_level_sets}
  Let $(W,\omega)$ be a connected compact Hamiltonian $\SS^1$-manifold
  that has convex contact type boundary, and let $H\colon W\to \RR$ be
  the Hamiltonian function of the circle action.  Complete $W$ by
  attaching cylindrical ends with respect to some invariant Liouville
  field (as described in \Cref{def: cylindrical end}) and denote the
  resulting manifold by $(\widehat{W}, \widehat \omega)$ and the
  extended Hamiltonian by $\widehat{H}$.

  Then it follows that the level sets of $\widehat{H}$ are either
  connected or empty.
\end{theorem_connected_level_sets}
\begin{proof}
  By the remarks made after \Cref{def: cylindrical end} if follows
  that $\widehat{H}$ is adapted to the cylindrical ends in the sense
  of \Cref{definition function adapted to cylindrical end}.
  Furthermore, $\widehat{H}$ is Morse-Bott and the indices of all its
  critical points are even, see \Cref{theorem: fixed points and
    boundary components of hamiltonian S1-manifolds}.(c).  The claim
  about the level sets follows then directly from \Cref{connected
    level set for cylindrical end}.
\end{proof}

The connectedness of the level sets of the moment map for closed
Hamiltonian manifolds has been further used by Atiyah
\cite{AtiyahConvexity} and by Guillemin and Sternberg
\cite{GuilleminSternbergConvexity} to show that the image of the
moment map for a Hamiltonian torus action on a closed symplectic
manifold is a convex set.

With \Cref{theorem: connected level sets}, we can easily deduce that
the image of the moment map for a Hamiltonian $\TT^2$-action on a
symplectic manifold with cylindrical ends is always convex.  We do not
have any doubt that this is also true for actions of higher
dimensional tori, but proving this would require us to first extend
the connectedness result above to torus actions.

As a first step, we show that if $(\widehat{W},\widehat{\omega})$ is a
Hamiltonian $\TT^k$-manifold with cylindrical ends, then the image of
its moment map is a closed set.  This, combined with \Cref{theorem:
  connected level sets} leads almost immeadiately to the proof of
convexity in the case $k=2$, using the classical proof for closed
manifolds.

\begin{proposition}\label{proposition: image is closed}
  Let $(\widehat{W},\widehat{\omega})$ be a manifold with cylindrical
  ends that is equipped with a Hamiltonian $\TT^k$-action as explained
  in \Cref{def: cylindrical end}.  The image of the moment map is then
  closed.
\end{proposition}
\begin{proof}
  For simplicity, we identify the dual of the Lie algebra of $\TT^k$
  with $\RR^k$ and denote by $\widehat\mu\colon \widehat W\to\RR^k$
  the corresponding moment map.

  Clearly since $W$ is compact, $\widehat\mu(W)$ is also
  compact. Thus, in order to prove that $\widehat{\mu}(\widehat W)$ is
  a closed subset, it is enough to show that the image of the
  cylindrical end under $\widehat\mu$ is closed.  Denote
  $\restricted{\widehat\mu}{V}$ by $\mu_V$ so that
  $\widehat\mu(s,x) = e^s\,\mu_V(x)$ for all
  $(s,x)\in [0,\infty)\times V$.
  
  Take a point $c = (c_1,\dotsc,c_k) \in \RR^k$ in the closure of
  $\widehat\mu\bigl([0,\infty)\times V\bigr)$.  Then, there is a
  sequence of points $(s_n, x_n)_n\subset [0,\infty)\times V$ such
  that $\widehat\mu(s_n,x_n) = e^{s_n}\,\mu_V(x_n)$ converges to
  $c\in\RR^k$.  We need to find an $(s,x)\in[0,\infty)\times V$ with
  $c = e^s\,\mu_V(x)$.  Since the boundary~$V$ is compact, the
  sequence~$(x_n)_n$ has a convergent subsequence that we still denote
  for simplicity by $(x_n)_n$ with a limit point~$x\in V$.

  If $\mu_V(x) = (\mu_1,\dotsc, \mu_k)\ne 0$, there is one component,
  say $\mu_1$, that is non-zero.  Clearly then $c_1$ does not vanish
  either, and $e^{s_n}$ converges to $\frac{c_1}{\mu_1}$, so that
  $s_n$ converges to $s = \log \frac{c_1}{\mu_1}$ and thus
  $c = \widehat\mu(s,x)$.

  Assume instead that $\mu_V(x) = 0$.  If $c = 0$ we are done, hence
  suppose that $c\ne 0$.  According to
  \cite[Theorem~4.8]{GuilleminSternbergConvexity} there is an open
  neighborhood $U_x\subset\widehat W$ around
  $(0,x)\in [0,\infty)\times V$ and an open neighborhood
  $U_0'\subset\RR^k$ around $\widehat\mu(0,x) = \mu_V(x) = 0$ such
  that
  \begin{equation}\label{equation:intersection}
    \widehat\mu(U_x) = U_0'\cap (C\times\RR^{k_1}) \;,
  \end{equation}
  where $C$ is a point if $x$ has a zero-dimensional stabilizer or it
  is a $(k-k_1)$-dimensional convex cone defined by the weights of the
  representation of the non-trivial stabilizer of $x$ on the tangent
  space~$T_x\widehat W$ and $0\le k_1\le k$ is the codimension of the
  stabilizer of $x$.  We may take the neighborhood~$U_x$ to be of the
  form $U_x = (-\epsilon,\epsilon)\times V_x$, for some open set
  $V_x\subset V$ containing $x\in V$ and some $\epsilon>0$
  sufficiently small, so that $U_x$ intersects the interior of $W$ in
  the standard collar neighborhood of the boundary~$V$.

  Let us prove that
  \begin{equation*}
    C\times\RR^{k_1}\subset\widehat\mu(\widehat W)  \;.
  \end{equation*}
  Note that $C\times\RR^{k_1}$ is for any $s\in \RR$ invariant under
  multiplication by $e^s$.  Thus if we take any point
  $b\in C\times\RR^{k_1}$, we can choose $s\gg 1$ so that
  $e^{-s}\, b\in U_0'\cap (C\times\RR^{k_1})$.  According to
  \eqref{equation:intersection} there is then a point
  $(s_b,x_b)\in (-\epsilon,\epsilon)\times V_x$ such that
  $\widehat\mu(s_b,x_b) = e^{s_b}\,\mu_V(x_b) = e^{-s}\, b$.  It
  follows that $b = \widehat\mu(s_b + s, x_b)$ so that
  $b \in \widehat\mu(\widehat W)$ as desired.
  
  Recall that we wanted to show that $c$ lies in the image of the
  moment map.  Since the sequence $x_n\in V$ converges to $x\in V$,
  all but finitely many of the $x_n$ will lie in $V_x$ so that
  $\widehat\mu(0,x_n) \in U_0'\cap (C\times\RR^{k_1})$.  It follows
  that
  \begin{equation*}
    \widehat\mu(s_n,x_n) = e^{s_n}\,  \mu_V(x_n) \in
    C\times\RR^{k_1}  \;,
  \end{equation*}
  and since the set on the right is closed, it contains the limit
  point~$c$.  We conclude that
  \begin{equation*}
    c\in C\times\RR^{k_1}\subset\widehat\mu(\widehat W) \;,
  \end{equation*}
  that is, $\widehat\mu(\widehat W)$ is a closed subset.
\end{proof}

\begin{corollary}
  If a Hamiltonian $\TT^k$-action on a symplectic manifold
  $(\widehat W,\widehat\omega)$ with cylindrical ends has only
  discrete stabilizers, then the moment map~$\widehat{\mu}$ is
  surjective.
\end{corollary}
\begin{proof} 
  Since $\widehat{\mu}(\widehat W)$ is by \Cref{proposition: image is
    closed} a closed set, it is enough to show that
  $\widehat{\mu}(\widehat W)$ is also open.  We will see that
  $\widehat{\mu}$ is a submersion, so that it is actually an open map.
  Choosing a basis $X_1,\dotsc,X_k$ allows us to identify the Lie
  algebra~$\tfrak^k$ of the torus~$\TT^k$ with $\RR^k$ and to write
  the moment map~$\widehat\mu$ as
  $(\widehat\mu_1, \dotsc, \widehat\mu_k)$ where each $\widehat\mu_j$
  for $j\in \{1,\dotsc,k\}$ is the Hamiltonian function for the action
  generated by $X_j$
 
  Clearly if there were a point~$x\in \widehat{W}$ at which
  $d_x\widehat\mu_1,\dotsc, d_x\widehat\mu_k$ were not linearly
  independent, then after possibly reordering the indices
  $d_x\widehat\mu_k$ would be a linear combination of
  $d_x\widehat\mu_1,\dotsc, d_x\widehat\mu_{k-1}$.  In particular, the
  Hamiltonian vector field corresponding to $X_k$ would be spanned by
  the other $k-1$ vector fields, giving a contradiction to the
  assumption that $\TT^k$ has only discrete stabilizers.  This shows
  that $\widehat{\mu}$ needs to be a submersion, and thus an open map.
\end{proof}

\begin{corollary}[Convexity for $\TT^2$-ations]
  Let $(\widehat{W},\widehat{\omega})$ be a symplectic manifold with
  cylindrical ends equipped with a Hamiltonian $\TT^2$-action, as
  explained in \Cref{def: cylindrical end}.  The corresponding moment
  map image $\widehat{\mu}(\widehat W)$ is then a convex set.
\end{corollary}
\begin{proof} 
  The only difference between our proof and the classical one for
  closed manifolds is that we need to appeal to \Cref{proposition:
    image is closed}, because $\widehat{\mu}(\widehat W)$ is not
  compact.  To clarify where this particular difference intervenes, we
  sketch now the proof.
  
  If the $\TT^2$-action is reducible, the image of $\widehat{W}$ is a
  connected subset inside a line so that it is convex.  Assume now
  that the action is irreducible and let
  $p_0,p_1\in\widehat{\mu}(\widehat W)\subset\RR^2$ be any two points
  in the image of $\widehat\mu$.  We need to show that the segment
  between $p_0$ and $p_1$ lies in the image of the moment map.  We can
  choose two constants $a_1,a_2 \in \RR$ such that the kernel of the
  linear map $\pi\colon \RR^2 \to \RR$ defined by
  $\pi(x_1,x_2) = a_1x_1+a_2x_2$ is parallel to the line~$\ell$
  passing through $p_0$ and $p_1$.

  The composition $\mu = \pi\circ\widehat\mu$ generates an
  $\RR$-action induced by the restriction of the $\TT^2$-action to the
  subgroup $t\mapsto (e^{ia_1 t},e^{ia_2 t})$.  If we denote by
  $c = \pi(\ell) \in \RR$ the projection of $\ell$, one easily
  verifies that
  \begin{equation}\label{equation:connected}
    \widehat\mu(\widehat W)\cap \ell = \widehat\mu(\mu^{-1}(c))  \;.
  \end{equation}

  If $a_1$ and $a_2$ are rational numbers, then it follows that $\mu$
  generates a circle action so that according to \Cref{theorem:
    connected level sets}, $\mu^{-1}(c)$ is as the level set of a
  Hamiltonian $\SS^1$-action connected.  We conclude that the
  intersection on the left side of \eqref{equation:connected} must
  also be connected set, and in particular, the segment joining $p_0$
  and $p_1$ lies as desired in $\widehat{\mu}(\widehat W)$.

  \smallskip
  
  If $a_1$ or $a_2$ are irrational, then there is a sequence of points
  $x_n, y_n\in\widehat W$ for $n\in\NN$ such that the sequences
  $\widehat\mu(x_n)$ and $\widehat\mu(y_n)$ converge to $p_0$ and
  $p_1$ respectively.  Since we assume that the $\TT^2$-action is
  irreducible, we can suppose that $x_n$ and $y_n$ have trivial
  stabilizers, so that $\widehat\mu$ will be a submersion on a
  neighborhood of $x_n$ and $y_n$.  By a perturbation, we can then
  further assume that $\widehat\mu(x_n)$ and $\widehat\mu(y_n)$ lie
  for all $n\in \NN$ in $\QQ^k$.
  
  By the argument stated above it follows that the segments joining
  the points~$\widehat\mu(x_n)$ and $\widehat\mu(y_n)$ lie for every
  $n\in \NN$ in the image of the moment map.  If $p'\in \ell$ is a
  point that lies between $p_0$ and $p_1$, then there is a
  sequence~$(p_n')_n$ such that $p_n' \to p$ and such that $p_n$ lies
  on the segment that connects $\widehat\mu(x_n)$ and
  $\widehat\mu(y_n)$.  Since $p_n' \in \widehat{\mu}(\widehat W)$ and
  since $\widehat{\mu}(\widehat W)$ is by \Cref{proposition: image is
    closed} a closed set, it follows that $p'$ also lies in the image
  of the moment map proving the convexity.
\end{proof}

\section{Outlook and open questions}

We already asked the following question in \Cref{sec: dimension 4}.

\begin{question}
  Do there exist connected compact Hamiltonian $\SS^1$-manifolds of
  dimension$\ge 6$ that have disconnected contact type boundary?
\end{question}

Karshon classified closed symplectic $4$-manifolds with a Hamiltonian
circle action \cite{KarshonHamiltonianCircleActions}.

\begin{question}
  Classify $4$-dimensional connected compact Hamiltonian
  $\SS^1$-manifolds with contact type boundary (or probably better
  suited for the classification result, with cylindrical ends).
\end{question}

\medskip

It would be very interesting to extend the results obtained in this
article to more general compact Lie groups.  In particular, it would
be interesting to see whether the convexity result of
Atiyah-Guillemin-Sternberg \cite{AtiyahConvexity,
  GuilleminSternbergConvexity} also holds in a suitable form for
connected compact symplectic $\TT^k$-manifolds with contact type
boundary.  The classical proof for closed manifolds is based on the
fact that all level sets of the moment map are connected or empty.
With \Cref{theorem: connected level sets} we have done a first step in
this direction by showing that this claim is valid for circle actions.

\begin{question}
  Do Hamiltonian torus actions on symplectic manifolds with contact
  type boundary satisfy some Atiyah-Guillemin-Sternberg convexity
  result?
\end{question}

\begin{question}
  What about symplectic toric manifolds with cylindrical ends?  Is
  there a Delzant classification \cite{DelzantToric} for such spaces?
\end{question}

Every symplectic toric manifold (of dimension$\ge 4$) with contact
type boundary will always have connected boundary by \Cref{rank two
  connected boundary}, and in fact, the boundary will be a contact
toric manifold as classified by Lerman \cite{Lerman1}.  Can one
characterize all symplectic toric manifolds having a certain contact
toric boundary?  In particular, can one understand combinatorically
how the contact moment map used by Lerman compactifies at the origin?

\bigskip

In symplectic topology, there is a hierarchy of different conditions
that can be imposed on the boundary of a symplectic manifold, reaching
from being the level set of a pluri-subharmonic function of a Stein
manifold to having weak contact type boundary, see
\cite{McDuff_contactType} in dimension~$4$ and
\cite{WeakFillabilityHigherDimension} in general.

\begin{question}
  In how far do these other boundary conditions interact with
  symplectic circle actions?  In particular, what can be said about
  compact symplectic $\SS^1$-manifolds that have weak contact type
  boundary?
\end{question}

\begin{question}
  Let $(W,\omega)$ be a Hamiltonian $G$-manifold that is a weak
  filling of some contact structure on its boundary.  Is it possible
  to find a $G$-invariant contact structure on $\p W$ such that the
  boundary is still of weak contact type?
\end{question}

\begin{example}
  Consider $W := \DD^2 \times \TT^2$ with symplectic structure
  $\omega = r\,dr\wedge d\phi + dx\wedge dy$, where $(r,\phi)$ are the
  polar coordinates on $\DD^2$ and $(x,y)$ are the coordinates on
  $\TT^2$.  Act on the $\TT^2$-factor by translations along the
  $x$-direction.  This circle action is clearly non-Hamiltonian.

  We can equip the boundary of $W$ with an $\SS^1$-invariant contact
  structure given as
  $\alpha_k = C\,d\phi + \sin(k\phi)\,dx - \cos(k\phi)\,dy$ for a
  constant $C>0$.  If $C$ is chosen large enough, then
  $\alpha_k\wedge \omega > 0$ so that $W$ will have weak contact type
  boundary \cite{Giroux_plusOuMoins}.

  It follows that there do exist symplectic $\SS^1$-manifolds with
  $\SS^1$-invariant weak contact type boundary that are not
  Hamiltonian, and that do thus not satisfy \Cref{corollary:
    symplectic group action}.  Bourgeois contact structures generalize
  this example to every dimension$\ge 4$, see
  \cite[Example~1.1]{WeakFillabilityHigherDimension} and
  \cite{LisiPropertiesBourgeoisStructures}.
\end{example}

\bigskip

With the introduction of holomorphic curves by Gromov
\cite{Gromov_HolCurves} symplectic topology bifurcated away from
symplectic geometry.  As mentioned in \Cref{sec: hamiltonian
  morse-bott and convex boundary}, if we choose a compatible almost
complex structure that is $\SS^1$-invariant, we obtain a
$J$-holomorphic foliation by holomorphic curves that are tangent to
the orbits of the circle action and to the gradient flow.

An interesting question in contact topology is to study the symplectic
fillings of a given contact manifold, that is, to study all compact
symplectic manifolds whose contact type boundary is the given contact
manifold.

Let $(W,\omega)$ be one of the compact Hamiltonian $\SS^1$-manifolds
with contact type boundary~$(V,\xi)$ studied in this article and
attach cylindrical ends along $V$.
\begin{itemize}
\item If $V$ is connected and decomposes into $\xi^+$ and $\xi^-$ as
  in the second case of \Cref{no extrema but boundary}, then the
  generic leaves of the holomorphic foliation will be holomorphic
  cylinders with one end that points out of $W$ through $\xi^-$ and
  the other one pointing out through $\xi^+$.
\item If $V$ is connected and if all boundary points lie either in
  $\xi^+$ or are fixed points as in \Cref{local min and boundary},
  then the generic leaves of the holomorphic foliation will be
  holomorphic cylinders that can be compactified on one of its ends by
  adding in a fixed point of the $\SS^1$-action.  This way, they will
  effectively be holomorphic planes.

  Assume that there are fixed points on the boundary, which
  necessarily have to be maxima of the Hamiltonian function.
\end{itemize}

In both situations considered above, the holomorphic foliation
contains many leaves that are contained in the cylindrical ends.  If
$W'$ is a different symplectic filling of $V$ that does not need to
have any $\SS^1$-symmetry, we will still keep the same cylindrical end
including the family of holomorphic cylinders or holomorphic planes.

If these holomorphic curves persist as they descent into the
symplectic filling, it could maybe be possible to find a relation
between the homology groups of $W$ and those of $W'$.

\begin{question}
  Let $(W,\omega)$ be a compact Hamiltonian $\SS^1$-manifold with
  contact type boundary~$(V,\xi)$ and attach cylindrical ends.  Assume
  we are in one of the situations mentioned above.

  Do the holomorphic curves in the cylindrical ends persist after
  entering the symplectic filling and allow us to study the filling
  with the strategy used by Wendl \cite{WendlGirouxTorsion}?
\end{question}

\appendix

\section{Łojasiewicz's theorem: Gradient trajectories of Morse-Bott
  functions converge to critical points}\label{section: Lojasiewicz}

It is clear that on closed manifolds, gradient trajectories of general
smooth functions always accumulate at the set of critical points.  For
Morse-Bott functions though, every trajectory even \emph{converges} to
a critical point.  The proof of this fact was explained to us by
Krzysztof Kurdyka and goes back to a result by Łojasiewicz concerning
analytic functions (see \cite{Lojasiewicz}) but it applies equally
well to our situation \cite{KurdykaGradientConjecture}.  Since the
proof is relatively elementary and not very well-known in the
symplectic community, we have decided to include it in this text.

\begin{convergence_trajectory}[Łojasiewicz]
  Let $(W, g)$ be a compact Riemannian manifold that might possibly
  have boundary, and let $f\colon W\to \RR$ be a Morse-Bott function.

  If $\gamma$ is a gradient trajectory such that $\gamma(t)$ is
  defined for all $t\ge 0$, then it follows that $\gamma(t)$ converges
  for $t\to \infty$ to a critical point of $f$.
\end{convergence_trajectory}
\begin{proof}
  Let $\gamma$ be a gradient trajectory that is defined for all
  $t\ge 0$.  The key point in the proof is that $\gamma$ is a
  trajectory of bounded length, and that such trajectories converge.

  We can assume that $W$ is a closed manifold by applying first
  \Cref{doubling Morse-Bott} in case that $\p W \ne \emptyset$.  This
  does not change anything about the existence or non-existence of a
  limit of $\gamma$, because the initial manifold is compact.

  Since $W$ is compact, $\gamma$ accumulates by \Cref{flow can only
    accumulate at critical points}.(a) at the set of critical points
  of $f$ so that there is a sequence~$(t_k)_k$ with $t_k \to \infty$
  with $\gamma(t_k)$ converging to a critical point~$p_\infty \in W$.
  For simplicity, we replace now $f$ by $f - f(p_\infty)$.  This
  modification does not change the gradient, but it simplifies our
  situation, because then $f(p_\infty) = 0$ and $f$ is strictly
  negative along $\gamma(t)$.
  
  By \Cref{gradient estimate close to critical point}, there exists a
  neighborhood~$U$ of $p_\infty$, and a constant~$c > 0$ such that
  \begin{equation*}
    \norm{\nabla f (p)} \ge c\, \sqrt{\abs{f(p)}}
  \end{equation*}
  for any $p\in U$.  Choose $\epsilon > 0$ so small that $U$ contains
  an $\epsilon$-neighborhood of $p_\infty$.  This allows us to apply
  \Cref{length bound for certain gradient curves} to $\gamma$ with
  $q = p_\infty$.  It follows that $\gamma(t)$ converges for
  $t\to \infty$ to $p_\infty$.
\end{proof}

\begin{lemma}\label{gradient estimate close to critical point}
  Let $(W,g)$ be a Riemannian manifold, and let $f\colon W\to \RR$ be
  a smooth function.  Assume that $q \in \Crit(f)$ is a critical point
  of Morse-Bott type that lies in the $0$-level set of $f$.

  Then there exists a neighborhood~$U_q$ of $q$ and a constant
  $c_q > 0$ such that
  \begin{equation*}
    \norm{\nabla f (p)} \ge c_q\, \sqrt{\abs{f(p)}}
  \end{equation*}
  for any $p\in U_q$.
\end{lemma}
\begin{proof}
  Using a Morse-Bott chart~$(U_q,\phi)$ (see \cite{BanyagaMorseBott})
  centered around $q$, we have coordinates~$(\x,\y)$ with
  $\x = (x_1,\dotsc,x_k)$ and $\y = (y_1,\dotsc,y_{n-k})$ such that
  \begin{equation*}
    f(\x,\y) = c_1\, y_1^2 + \dotsm  + c_{n-k}\, y_{n-k}^2
  \end{equation*}
  where the $c_1,\dotsc, c_{n-k}$ are equal to $\pm 1$.  The
  points~$(\x,\y)$ with $\y = \0$ in this chart correspond to the
  points in $\Crit(f)$, and $(\0,\0)$ corresponds to $q$.

  The differential of $f$ is
  $df = 2c_1\,y_1\, dy_1 + \dotsm + 2c_{n-k}\, y_{n-k}\, dy_{n-k}$,
  and denoting the standard Euclidean metric on this chart by
  $g_{\mathrm{Eucl}}$, we easily verify that
  \begin{equation*}
    \norm{df}_{\mathrm{Eucl}}^2 = 4\,\bigl(y_1^2 + \dotsm  + y_{n-k}^2\bigr)
    \ge \abs{f(\x,\y)} \;.
  \end{equation*}

  Denote the pull-back of the metric~$g$ from $W$ to the Morse-Bott
  chart for simplicity also by $g$, and let $\norm{\cdot}$ be the norm
  with respect to this product.  We find a constant~$c_q>0$ such that
  $\norm{\cdot} \ge c_q\, \norm{\cdot}_{\mathrm{Eucl}}$ at any point
  close to the origin.

  On this smaller neighborhood of $(\0,\0)$ we obtain then the desired
  statement
  \begin{equation*}
    \norm{\nabla f}^2 = \norm{df}^2 \ge c_q^2\, \norm{df}_{\mathrm{Eucl}}
    \ge c_q^2\, \abs{f} \;.  \qedhere
  \end{equation*}
\end{proof}

\begin{lemma}\label{length bound for certain gradient curves}
  Let $(W,g)$ be a geodesically complete Riemannian manifold.  Let
  $f\colon W\to \RR$ be a smooth function and let $q$ be a point in
  the $0$-level set of $f$ such that there exist an $\epsilon>0$ and a
  constant~$c > 0$ with
  \begin{equation*}
    \norm{\nabla f (p)} \ge c\, \sqrt{\abs{f(p)}}
  \end{equation*}
  for any $p$ in the $\epsilon$-neighborhood of $q$.

  Let $\gamma\colon [0,T_{\max}) \to W$ be a gradient trajectory with
  $T_{\max}\in (0,\infty]$, and assume that there is a
  sequence~$(t_k)_k$ with $t_k \to T_{\max}$ such that all
  $\gamma(t_k)$ lie in an $\epsilon/2$-neighborhood around $q$ and
  such that $f\bigl(\gamma(t_k)\bigr) \to 0$ for $k\to \infty$.

  Then it follows that $\gamma$ has finite length and that it can be
  completed to a continuous path
  $\hat \gamma\colon [0,T_{\max}] \to W$.
\end{lemma}
\begin{proof}
  Note that since $f$ is monotonous along gradient trajectories,
  $f\bigl(\gamma(t)\bigr)$ is strictly negative so that
  $\abs{f\bigl(\gamma(t)\bigr)} = -f\bigl(\gamma(t)\bigr)$.  It
  follows that $\abs{f\bigl(\gamma(t)\bigr)}^{1/2}$ is differentiable,
  and using that $\gamma'(t)$ and $\nabla f\bigl(\gamma(t)\bigr)$ are
  co-linear, we compute
  \begin{equation*}
    \frac{d}{dt} \sqrt{\abs{f\bigl(\gamma(t)\bigr)}} =
    \frac{-\frac{d}{dt} f\bigl(\gamma(t)\bigr) }{2\,
      \sqrt{\bigl|f\bigl(\gamma(t)\bigr)\bigr|}} = -
    \frac{df\bigl(\gamma'(t)\bigr)}{2\,
      \sqrt{\bigl|f\bigl(\gamma(t)\bigr)\bigr|}} = -
    \frac{\bigl\langle\nabla f \bigl(\gamma(t)\bigr), \gamma'(t)
      \bigr\rangle}{2\, \sqrt{\bigl|f\bigl(\gamma(t)\bigr)\bigr|}} =
    - \frac{ \norm{\nabla f\bigl(\gamma(t)\bigr)}\cdot
      \norm{\gamma'(t)} }{2\,
      \sqrt{\bigl|f\bigl(\gamma(t)\bigr)\bigr|}} \;.
  \end{equation*}
  Denote the $\epsilon$-neighborhood of $q$ by $U$.  Then combining
  the previous equation with the lower bound of the gradient, we
  obtain for any $\gamma(t)\in U$ that
  \begin{equation*}
    \frac{d}{dt} \sqrt{\abs{f\bigl(\gamma(t)\bigr)}}
    \le - \frac{c\, \bigl|f\bigl(\gamma(t)\bigr)\bigr|^{1/2}\,
      \norm{\gamma'(t)}}
    {2\, \bigl|f\bigl(\gamma(t)\bigr)\bigr|^{1/2}}
    = - \frac{c}{2}\, \norm{\gamma'(t)} \;.
  \end{equation*}
  It follows for every $\gamma(t) \in U$ that $\norm{\gamma'(t)}$ is
  bounded from above by
  $-\frac{2}{c}\, \frac{d}{dt} \sqrt{\abs{f\bigl(\gamma(t)\bigr)}}$.
  Thus if $\gamma\bigl([t_0,t_1]\bigr)$ is a segment that lies inside
  $U$, its length is bounded by
  \begin{multline*}
    \ell\bigl(\restricted{\gamma}{[t_0,t_1]}\bigr) = \int_{t_0}^{t_1}
    \norm{\gamma'(t)}\,dt \le - \frac{2}{c}\; \int_{t_0}^{t_1}
    \frac{d}{dt} \sqrt{\abs{f\bigl(\gamma(t)\bigr)}} \,dt \\
    = \frac{2}{c}\; \Bigl( \sqrt{\abs{f\bigl(\gamma(t_0)\bigr)}} -
    \sqrt{\abs{f\bigl(\gamma(t_1)\bigr)}} \Bigr) < \frac{2}{c} \,
    \sqrt{\abs{f\bigl(\gamma(t_0)\bigr)}} \;.
  \end{multline*}
  Using the sequence~$(t_k)_k$, we can find a $t'$ such that
  $\gamma(t')$ will be at distance less than $\epsilon /2$ from $q$,
  and such that
  $\frac{2}{c} \, \sqrt{\abs{f\bigl(\gamma(t')\bigr)}} <
  \frac{\epsilon}{2}$.  To escape from the neighborhood~$U$, the
  length of $\gamma\bigl([t',\infty)\bigr)$ would need to be more than
  $\epsilon /2$, but since the length of any path
  $\gamma\bigl([t',t]\bigr)$ with $t>t'$ is bounded by $\epsilon /2$,
  $\gamma$ is trapped in $U$, and the length of
  $\gamma\bigl([t',T_{\max})\bigr)$ is also bounded by $\epsilon/2$.
  Due to \Cref{finite length path extends} below, $\gamma$ extends to
  a continuous path on $[0,T_{\max}]$.
\end{proof}

\begin{lemma}\label{finite length path extends}
  Let $(W,d)$ be a complete metric space and let
  $\gamma\colon [0,T) \to W$ with $T\in (0,\infty]$ be a continuous
  path of finite length (in the sense that the restriction of $\gamma$
  to every compact subinterval is rectifiable and of bounded length).
  Then it follows that $\gamma$ extends to all of $[0,T]$
  continuously.
\end{lemma}
\begin{proof}
  Define a function
  $\ell(t) := \length\bigl(\restricted{\gamma}{[0,t]}\bigr)$.  Since
  $\ell(t)$ is increasing and bounded, it converges for $t\to T$ to
  some real number.  In particular, if $(t_k)_k$ is any sequence with
  $t_k\to T$, then if follows that $\bigl(\ell(t_k)\bigr)_k$ is a
  Cauchy sequence.  Assuming without loss of generality that
  $t_k > t_n$, we find
  \begin{equation*}
    d\bigl(\gamma(t_n), \gamma(t_k)\bigr)
    \le \length\bigl(\restricted{\gamma}{[t_n,t_k]}\bigr)
    = \ell(t_n) - \ell(t_k) = \abs{\ell(t_n) - \ell(t_k)}
  \end{equation*}
  from which it follows for every sequence~$(t_k)_k$ with $t_k\to T$
  that $\bigl(\gamma(t_k)\bigr)_k$ is also a Cauchy sequence.  As
  desired, we obtain that $\gamma(t)$ converges for $t\to T$.
\end{proof}

\section{Technical lemmas about Morse-Bott flows}

In order to avoid some of the technicalities arising along the
boundary, we often make use of the following ``doubling trick'' that
was suggested to us by Marco Mazzucchelli.

\begin{lemma}\label{doubling Morse-Bott}
  Let $W$ be a compact manifold with boundary and let
  $f\colon W \to \RR$ be a Morse-Bott function (in the sense of
  \Cref{def:Morse-Bott}).
  
  We can cap-off $W$ to a closed manifold~$W^{\mathrm{cap}}$ and
  extend $f$ to a $C^2$-function
  $ f^{\mathrm{cap}}\colon W^{\mathrm{cap}} \to \RR$ that is
  Morse-Bott and such that every connected component of $\Crit(f)$
  with boundary is completed to a closed connected component of
  $\Crit(f^{\mathrm{cap}})$.
\end{lemma}
\begin{proof}
  We obtain $W^{\mathrm{cap}}$ by doubling $W$ along the boundary.
  For this, choose a vector field~$X$ that is positively transverse to
  $\p W$ and that is tangent to any component $C_j \subset \Crit(f)$
  that intersects $\p W$.  Recall that by our assumption, any such
  $C_j$ intersects $\p W$ transversely.  By following the flow of $X$
  from $\p W$ in negative time direction, we define a collar
  neighborhood~$U$ via the diffeomorphism
  $(-\epsilon, 0] \times \p W \to W,\; (t,p) \mapsto \Phi^X_t(p)$.  If
  $\epsilon > 0$ has been chosen sufficiently small, it follows that
  $\Crit(f) \cap U$ is given by
  $(-\epsilon,0] \times \bigl(\cup_j \p C_j\bigr)$ where the union is
  over all components $C_j \subset \Crit(f)$ with
  $\p C_j := C_j\cap \p W \ne \emptyset$.  This collar neighborhood
  allows us to glue a small open collar $[0, \epsilon) \times \p W$ to
  $W$ using the obvious identification.  We denote the extended
  manifold by $W_\epsilon$.

  Using an elementary construction by Lichtenstein
  \cite{Lichtenstein_C1doubling}, we define an explicit
  extension~$f_\epsilon$ of $f$ as
  \begin{equation*}
    f_\epsilon(t,p) =
    \begin{cases}
      f(t,p) &  \text{ if $(t,p)\in (-\epsilon,0] \times \p W$} \\
      6f(0,p) + 3 f(-t,p) - 8 f\bigl(-\tfrac{1}{2}\,t,p\bigr) & \text{
        if $(t,p)\in (0,\epsilon) \times \p W$}
    \end{cases} \;.
  \end{equation*}
  This $f_\epsilon$ is only a $C^2$-function, but we could have easily
  improved the construction by including additional terms to obtain
  higher regularity (or even a smooth extension
  \cite{Seeley_SmoothDoubling}), but $C^2$-regularity is sufficient
  for us.

  After possibly shrinking the size of $\epsilon >0$ further, all
  critical points of $f_\epsilon$ in the collar
  $(-\epsilon, \epsilon) \times \p W$ are of the form
  $(-\epsilon,\epsilon)\times \p C_j$ and satisfy the Morse-Bott
  condition.

  \medskip

  Let $\overline{W}_\epsilon$ be a copy of $W_\epsilon$ with reversed
  orientation.  We double $W$ by gluing the disjoint union
  \begin{equation*}
    W_\epsilon \sqcup \overline{W}_\epsilon \;,
  \end{equation*}
  along the boundary collar using  the diffeomorphism
  \begin{equation*}
    (-\epsilon,\epsilon) \times \p W \to (-\epsilon,\epsilon) \times \p
    W,\; (t,p) \mapsto (-t,p) \;.
  \end{equation*}
  We call the resulting manifold~$W^{\mathrm{cap}}$.

  \smallskip
  
  We construct now an extension~$f^{\mathrm{cap}}$ that agrees on
  $W^{\mathrm{cap}}\setminus \bigl((-\epsilon,\epsilon)\times \p
  W\bigr)$ with $f$ (both on $W$ and on $\overline{W}$); on the
  boundary collar $(-\epsilon,\epsilon)\times \p W \subset W_\epsilon$
  we need to do an interpolation to glue both halves together.  For
  this choose a sufficiently small~$\delta >0$ and define a cut-off
  function $\rho_\delta\colon (-\epsilon,\epsilon) \to [0,1]$ that is
  equal to $1$ on $(-\epsilon,0]$, equal to $0$ on
  $[\delta, \epsilon)$, and that is monotonously decreasing in
  between.  We define $f^{\mathrm{cap}}$ piecewise on
  $W^{\mathrm{cap}}$ by setting
  \begin{equation*}
   f^{\mathrm{cap}} \colon W^{\mathrm{cap}} \to \RR, \quad
    \begin{cases}
      p \mapsto f(p) &  \text{ for all $p\in W_\epsilon \setminus (-\epsilon,\epsilon)\times \p W$} \\
      (t,p) \mapsto \rho_\delta(t)\, f_\epsilon (t,p)
      + \bigl(1 - \rho_\delta(t)\bigr)\, f_\epsilon (-t,p) &
      \text{ for all $(t,p)\in (-\epsilon,\epsilon)\times \p W$} \\
      p \mapsto f(p) & \text{ for all
        $p\in \overline{W}_\epsilon \setminus (-\epsilon,\epsilon)
        \times \p W$}
    \end{cases}
  \end{equation*}
  The function~$ f^{\mathrm{cap}}$ is a $C^2$-regular extension of $f$ on $W$.
  Outside the gluing collar
  $(0,\epsilon)\times \p W \subset W_\epsilon$, the critical points of
  $ f^{\mathrm{cap}}$ are identical to the ones of $f$ on $W$ and to the ones of
  the mirror of $f$ on
  $\overline{W}_\epsilon \setminus (-\epsilon, \epsilon)\times \p W$.
  These are all of Morse-Bott type, and it only remains to understand
  $\Crit( f^{\mathrm{cap}})$ on $(0,\epsilon)\times \p W \subset W_\epsilon$.

  If $C_j \subset \Crit(f)$ is a component that intersects $\p W$, we
  easily check that the first derivatives of $ f^{\mathrm{cap}}$ vanish along
  $(0,\epsilon)\times \p C_j$, because
  $f_\epsilon(t, p) \equiv f_\epsilon(0, p)$ for every $p\in \p C_j$,
  and because $df_\epsilon = 0$ on $(0,\epsilon)\times \p C_j$.  It
  follows that $(0,\epsilon)\times \p C_j\subset \Crit( f^{\mathrm{cap}})$.

  To check the non-degeneracy condition, choose in $\p W$ a chart
  around $p\in \p C_j$ with coordinates
  $(\x,\y) = \bigl(x_1,\dotsc,x_k, y_1, \dotsc, y_l\bigr)$ such that
  $\p C_j$ corresponds in this chart to the points where $\x = \0$.
  We can compute the Hessian of $ f^{\mathrm{cap}}$ in coordinates $(t; \x, \y)$
  by computing the second partial derivatives.  We are only interested
  in the restriction of the Hessian to the normal bundle of
  $(0,\epsilon)\times \p C_j$, so that it is sufficient to consider
  $\frac{\partial^2  f^{\mathrm{cap}}}{\partial x_i \partial x_j} (t; \0, \y)$
  for all $1\le i,j \le k$.
  
  Using that all first derivatives of $f_\epsilon$ vanish along
  $(0,\epsilon)\times \p C_j$, we find on this subset
  \begin{equation*}
    \frac{\p^2  f^{\mathrm{cap}}}{\partial x_i \partial x_j} (t; \0, \y) =
    \rho_\delta(t)\, \frac{\p^2  f_\epsilon }{\partial x_i \partial x_j} (t; \0, \y) +
    \bigl(1 - \rho_\delta(t)\bigr)\,
    \frac{\p^2 f_\epsilon}{\partial x_i \partial x_j}  (-t; \0, \y) \;.
  \end{equation*}
  The non-degeneracy condition is open, and
  $\frac{\p^2 f_\epsilon }{\partial x_i \partial x_j} (t; \x, \y)$ and
  $\frac{\p^2 f_\epsilon }{\partial x_i \partial x_j} (-t; \x, \y)$
  obviously agree for $t=0$.  Thus if we choose $\delta > 0$
  sufficiently small, it follows that
  $\frac{\p^2  f^{\mathrm{cap}}}{\partial x_i \partial x_j} (t; \0, \y)$ is a
  non-degenerate $k\times k$-matrix, and all points in
  $(-\epsilon,\epsilon) \times \p C_j$ are Morse-Bott singularities.

  \smallskip

  So far we have only shown that $\hat f$ has Morse-Bott singularities
  on $W$, on
  $\overline{W}_\epsilon \setminus \bigl((-\epsilon,\epsilon) \times
  \p W\bigr)$, and on a neighborhood of
  $(0,\epsilon) \times \bigl(\Crit(f) \cap \p W\bigr)$.  In the
  remaining subset of $(0,\epsilon) \times \p W$, the critical points
  can be arbitrarily degenerate, but we use that every function can be
  perturbed to a $C^2$ Morse function without modifying it on the
  domain where the function is already Morse-Bott
  \cite{MilnorHCobordism}.
\end{proof}

We thank Marco Mazzucchelli for also having told us about the
following classical results.

\begin{lemma}\label{flow can only accumulate at critical points}
  \label{box around critical component}
  Let $(W,g)$ be a Riemannian manifold that might have boundary, but
  that does not need to be compact, and let $f\colon W \to \RR$ be a
  smooth function.
  \begin{itemize}
  \item [(a)] Choose a compact subset~$U\subset W$.  Every gradient
    trajectory that is defined for all positive times and that is
    trapped in $U$ in forward time direction, accumulates at the
    critical points of $f$.
  \item [(b)] Let $C_0$ be a compact isolated path-connected component
    of $\Crit(f)$.  Then there exist arbitrarily small open
    neighborhoods~$U_0$ and $U_1$ of $C_0$ with $U_0\subset U_1$ such
    that every gradient trajectory entering $U_0$ either
    \begin{itemize}
    \item escapes for some positive time from the larger
      neighborhood~$U_1$ and does never again reenter afterwards the smaller
      neighborhood~$U_0$;
    \item or it is trapped inside $U_1$.
    \end{itemize}
    In the second situation, the trajectory either ends at a point in
    $\p W$, or it tends asymptotically towards $C_0$.
  \end{itemize}
\end{lemma}
\begin{proof}
  \textbf{(a)} Let $\gamma\colon [0,\infty)\to W$ be a trajectory that
  is trapped in $U$, in the sense that there is a $t_0\in \RR$ such
  that $\gamma(t)\in U$ for all $t > t_0$.  We have to show that for
  every open set~$U_\epsilon \subset U$ containing $\Crit(f) \cap U$,
  there exists a time~$T \in \RR$ such that all $\gamma(t)$ with
  $t > T$ lie in $U_\epsilon$.

  Attach a small collar to $W$ along the boundary to avoid
  technicalities, and extend $f$ and $g$ smoothly to this new set.
  Since $U$ is compact, this extension does not change anything about
  the existence or not of a limit of $\gamma$.
  
  If there exists for every $k\in \NN$, a $t_k > k$ such that
  $\gamma(t_k) \notin U_\epsilon$, then using that $U$ is compact, we
  find a subsequence $(t_k)_k$ with $t_k\to \infty$ such that
  $\gamma(t_k)$ converges to a point~$p_\infty$.  By construction,
  $p_\infty$ does not lie in $U_\epsilon$, and is thus a regular point
  of $f$.

  Since $\nabla f(p_\infty) \ne 0$, we can put a flow-box around
  $p_\infty$ with coordinates $(x_0,\dotsc, x_n)$ such that
  $\abs{x_j} < \delta$, such that $\nabla f$ agrees in this chart with
  the constant vector field $\frac{\partial}{\partial x_0}$, and such
  that $\{x_0 = 0\}$ is the level set of $f$ containing $p_\infty$.
  By our assumption there is a $k_0 > 0$ such that $\gamma(t_{k_0})$
  lies in the flow-box, and since the flow in the chart follows the
  $x_0$-direction, it takes $\gamma$ only finite time after $t_{k_0}$
  to reach and cross the hyperplane~$\{x_0 = 0\}$.  Since $f$
  increases along the gradient trajectories, $\gamma$ can never again
  approach the level set containing $p_\infty$, so that it is
  impossible that $\gamma(t_k) \to p_\infty$ for any sequence with
  $t_k\to \infty$.

  \textbf{(b)} The claim follows directly from the slightly more
  technical lemma below.  Choose any open neighborhood~$U_1$ of $C_0$
  that has compact closure, and assume that $\overline{U_1}$ does not
  intersect any other critical points of $f$ outside $C_0$.  Let $U_0$
  be now a smaller open neighborhood of $C_0$ such that
  $\overline{U_0} \subset U_1$.  Applying \Cref{function increases to
    leave box around critical component} to these neighborhoods
  provides a constant~$\epsilon > 0$ such that every gradient
  trajectory~$\gamma$ with $\gamma(0) \in U_0$ and
  $\gamma(T) \notin U_1$ for some $T > 0$ satisfies
  \begin{equation*}
    f\bigl(\gamma(T)\bigr) >  f\bigl(\gamma(0)\bigr) + \epsilon \;.
  \end{equation*}

  Replace $U_0$ now by the intersection
  \begin{equation*}
    U_0 \cap f^{-1}\bigl((f(C_0) - \tfrac{\epsilon}{2},
    f(C_0)+\tfrac{\epsilon}{2})\bigr) \;.
  \end{equation*}
  This new subset is still an open neighborhood of $C_0$ with
  $\overline{U_0}\subset U_1$, and if $\gamma$ is a gradient
  trajectory that passes through $U_0$ and later escapes from $U_1$,
  then the value of $f$ along $\gamma$ will have grown to more than
  $f(C_0) - \epsilon/2 + \epsilon = f(C_0) + \epsilon/2$ once $\gamma$
  has left $U_1$.  Since $f$ increases along gradient trajectories, it
  is impossible for $\gamma$ to ever come back to $U_0$ after escaping
  from $U_1$, because
  $f(U_0) \subset \bigl(f(C_0) -\epsilon/2, f(C_0) +
  \epsilon/2\bigr)$.

  \smallskip
  
  If $\gamma$ is a gradient trajectory that never leaves $U_1$ and
  that is defined for all positive times, then it follows from
  part~(a) that $\gamma$ has to accumulate at the critical points of
  $f$, and since $\overline{U_1}\cap \Crit(f) = C_0$, $\gamma$ will
  tend asymptotically to $C_0$.

  If $\gamma$ is a gradient trajectory that never leaves $U_1$ and
  that is only defined up to time~$T_{\max} \ge 0$, then because
  $\nabla f$ is bounded on the compact set~$\overline{U_1}$, it
  follows from \Cref{finite length path extends} that $\gamma$ extends
  continuously to \emph{closed} interval~$[0,T_{\max}]$ with
  $\gamma(T_{\max}) = p_\infty$.  The point~$p_\infty$ cannot lie
  outside $U_1$, because $f(p_\infty) < f(C_0) + \epsilon/2$.

  Furthermore, $p_\infty$ cannot be a regular point of $f$ in
  $U_1\setminus \p W$, because otherwise the gradient trajectory would
  continue for $t>T_{\max}$ even after it had reached the
  point~$p_\infty$ which is in contradiction to our previous
  assumption.  That $p_\infty$ cannot be a critical point of $f$
  follows from the uniqueness of solutions of ordinary differential
  equations.  The only option left is that $p_\infty$ lies on $\p W$.
\end{proof}

\begin{lemma}\label{function increases to leave box around critical
    component}
  Let $(W,g)$ be a Riemannian manifold that might have boundary, and
  let $f\colon W \to \RR$ be a smooth function that has a compact
  isolated path-connected component~$C_0$ of critical points.

  Let $U_0$ and $U_1$ be open neighborhoods of $C_0$ such that $U_1$
  has compact closure, such that $\overline{U_1}\cap \Crit(f) = C_0$,
  and such that $\overline{U_0}\subset U_1$.

  Then there exists a constant~$\epsilon > 0$ such that $f$ increases
  by more than $\epsilon$ along any gradient trajectory of $f$ that
  starts inside $U_0$, and that later escapes from $U_1$, that is, let
  $\gamma$ be a gradient trajectory such that $\gamma(0) \in U_0$ and
  $\gamma(T) \notin U_1$ for $T > 0$, then
  \begin{equation*}
    f\bigl(\gamma(T)\bigr) >  f\bigl(\gamma(0)\bigr) +  \epsilon \; .
  \end{equation*}
\end{lemma}
\begin{proof}
  Choose open neighborhoods~$U_0, U_1$ of $C_0$ as in the lemma, and
  denote the distance between $\overline{U_0}$ and $W\setminus U_1$
  (note that if $W\setminus U_1 = \emptyset$, there is nothing to
  show) by $\rho$. Denote the minimum of $\norm{\nabla f}$ over
  $\overline{U_1} \setminus U_0$ by $m_\nabla$, and define
  $\epsilon := \rho\, m_\nabla$.  Clearly $m_\nabla > 0$ because
  $\overline{U_1}\cap \Crit(f) = C_0$ and $C_0\subset U_0$.

  Assume that $\gamma$ is a gradient trajectory such that
  $\gamma(0) \in U_0$ and such that $\gamma(T) \notin U_1$ for some
  $T > 0$.  Then there is a time~$t_0$ and a time~$t_1$ with
  $0\le t_0 < t_1\le T$ such that $\gamma(t_0)$ lies in the boundary
  of $U_0$, $\gamma(t_1)$ lies in the boundary of $U_1$, and all
  $\gamma(t)$ with $t\in (t_0,t_1)$ lie in
  $U_1 \setminus \overline{U_0}$.  This yields the following lower
  bound
  \begin{equation*}
    \begin{split}
      f\bigl(\gamma(T)\bigr) - f\bigl(\gamma(0)\bigr) &\ge
      f\bigl(\gamma(t_1)\bigr) - f\bigl(\gamma(t_0)\bigr) =
      \int_{t_0}^{t_1} \frac{d}{dt} f\bigl(\gamma(t)\bigr) \; dt =
      \int_{t_0}^{t_1}
      df\bigl(\gamma'(t)\bigr) \; dt \\
      &= \int_{t_0}^{t_1} \norm{\gamma'(t)}^2 \; dt \ge m_\nabla \cdot
      \int_{t_0}^{t_1} \norm{\gamma'(t)} \; dt \ge \rho\, m_\nabla =
      \epsilon \;.  \qedhere
    \end{split}
  \end{equation*}
\end{proof}

\begin{theorem}\label{stable subset smooth on closed manifold}
  Let $(W,g)$ be a \emph{closed} Riemannian manifold, and let
  $f\colon W\to \RR$ be a Morse-Bott function.  The stable and
  unstable subset of a component~$C_j$ of $\Crit(f)$ are smooth
  submanifolds of dimension $\dim C_j + i^-(C_j)$ and
  $\dim C_j + i^-(C_j)$ respectively.
\end{theorem}
\begin{proof}
  The Hadamard-Perron Theorem in
  \cite[Theorem~4.1]{HirshInvariantManifolds} gives us locally the
  existence of a stable manifold~$W^s_{\mathrm{loc}}$ for
  $C_j\subset \Crit(f)$ of dimension $\dim C_j + i^-(C_j)$.  Clearly
  $W^s_{\mathrm{loc}}$ lies in $W^s(f;C_j)$, but this does not yet
  imply that $W^s(f;C_j)$ is itself (even locally around $C_j$) a
  smooth submanifold.

  For this, we can choose arbitrarily small neighborhoods~$U_0$ and
  $U_1$ of $C_j$ as in \Cref{box around critical component}.(b) such
  that $U_0\subset U_1$ and such that every gradient trajectory of
  $\nabla f$ that passes through $U_0$ and later escapes from $U_1$
  can never again return to $U_0$.  Choosing $U_1$ so small that it
  lies in the neighborhood of $C_j$ in which the Hadamard-Perron
  Theorem gives us the local existence of $W^s_{\mathrm{loc}}$, we can
  guarantee that the only points in $U_0$ that lie in the stable set
  $W^s(f;C_j)$ are those that lie in $W^s_{\mathrm{loc}}$, because a
  trajectory through a point of $U_0$ that converges to $C_j$ has to
  stay inside $U_1$.  This shows that the intersection
  $W^s(f;C_j)\cap U_0 = W^s_{\mathrm{loc}}\cap U_0$ is a smooth
  submanifold.

  With this information it follows easily that $W^s(f;C_j)$ is
  globally a submanifold, because if $p$ is any point in $W^s(f;C_j)$,
  then there is necessarily a $T> 0$ such that $q:=\gamma(T)$ lies in
  $W^s(f;C_j) \cap U_0$.  Choose an open neighborhood~$U_q$ of $q$
  inside $U_0$.

  The gradient flow~$\Phi_T^{\nabla f}$ is a diffeomorphism, thus
  setting $U_p := \Phi_{-T}^{\nabla f}\bigl(U_q\bigr)$, we obtain an
  open neighborhood of $p$, and by the invariance of $W^s(f;C_j)$ it
  follows that
  $W^s(f;C_j) \cap U_p = \Phi_{-T}^{\nabla f}\bigl(W^s(f;C_j)\cap
  U_q\bigr)$ so that $W^s(f;C_j)$ is globally a smooth submanifold.

  For the result about the unstable subset it suffices to invert the
  sign of $f$.
\end{proof}

\section{Proof of \Cref{symplectic completion independent}}
\label{sec: proof symplectic completion independent}

\begin{proof}[Proof of \Cref{symplectic completion independent}]
  A boundary collar is obtained with the help of an (invariant)
  Liouville vector field.  For any two such choices $Y_0$ and $Y_1$,
  attach cylindrical ends to $(W,\omega)$, and denote the resulting
  manifolds by $(\widehat{W}_0, \widehat \omega_0)$ and
  $(\widehat{W}_1,\widehat \omega_1)$ respectively.  We will construct
  a symplectic fibration~$\widehat{W}_{[0,1]}$ over $[0,1]$ such that
  the fibers over $0$ and $1$ correspond to $\widehat{W}_0$ and
  $\widehat{W}_1$ respectively, and then use parallel transport to
  show that both fibers are isomorphic.

  \smallskip
  
  To construct the fibration, let
  $Y_\tau = (1-\tau)\,Y_0 + \tau\, Y_1$ for $\tau\in[0,1]$ be a family
  of $G$-invariant Liouville vector fields obtained by interpolating
  linearly between $Y_0$ and $Y_1$.  Choose $\epsilon > 0$
  sufficiently small so that the flow~$\Phi_s^{Y_\tau}(p)$ exists for
  every $\tau\in [0,1]$, every $s\in (-\epsilon,0]$, and every
  $p\in \p W$.

  Consider now $W\times [0,1]$ with the natural $G$-action, and
  generalize the construction of \Cref{collar neighborhood} by using
  the collar neighborhood of $\p W\times [0,1]$ obtained via a
  diffeomorphism
  \begin{equation*}
    \Psi\colon (-\epsilon,0]\times \p W \times [0,1] \to W\times
    [0,1], \; (s,p,\tau) \mapsto \bigl(\Phi_s^{Y_\tau}(p),
    \tau\bigr) \;.
  \end{equation*}
  On the collar, the group action pulls-back to
  $g\cdot (s,p,\tau) = (s,gp,\tau)$, and $\Psi$ is by construction
  $G$-equivariant.  We extend this group action in the obvious way to
  $(-\epsilon,\infty)\times \p W \times [0,1]$.

  The total space of the fibration is obtained by gluing the two
  components of
  \begin{equation*}
    W\times [0,1] \;\sqcup\; (-\epsilon,\infty)\times \p W \times [0,1]
  \end{equation*}
  with the map~$\Psi$ along $(-\epsilon,0]\times \p W \times [0,1]$.
  We denote the resulting $G$-manifold by $\widehat{W}_{[0,1]}$.  The
  $\tau$-coordinate is well-defined and $G$-invariant on
  $\widehat{W}_{[0,1]}$, and makes $\widehat{W}_{[0,1]}$ into a
  fibration over $[0,1]$ whose fibers~$\widehat{W}_\tau$ are the
  subsets with constant $\tau$-value.

  \smallskip  
  
  We define now a $2$-form~$\hat\omega$ on $\widehat{W}_{[0,1]}$ that
  restricts to a symplectic form on every fiber.  On $W\times [0,1]$
  define $\hat\omega$ to be the pull-back of $\omega$.  We then define
  on the collar neighborhood of $\p W \times [0,1]$ a $1$-form
  $\lambda = \iota_{Y_\tau} \omega$ that restricts on every
  $\tau$-slice to a Liouville form.  More explicitly we find
  $\lambda = (1-\tau)\,\lambda_0 + \tau\, \lambda_1$ so that
  $d\lambda = \omega + \bigl(\lambda_0 - \lambda_1\bigr)\wedge d\tau$.
  It follows that $\iota_{Y_\tau} \lambda = 0$ and
  $\lie{Y_\tau} \lambda = \iota_{Y_\tau} \omega +
  \bigl(\lambda_0(Y_\tau) - \lambda_1(Y_\tau)\bigr) \, d\tau = \lambda
  + \omega(Y_0,Y_1) \, d\tau$.

  Let $F\colon (-\epsilon,0]\times \p W \times [0,1] \to \RR$ be the
  solution of the ordinary differential equation
  \begin{equation*}
    \frac{\partial}{\partial u} F(u,p,\tau) =
    F(u,p,\tau)  
    + \omega(Y_0,Y_1)\circ \Phi_u^{Y_\tau}
  \end{equation*}
  with initial value~$F(0,p,\tau) = 0$.  Then it follows that
  \begin{equation*}
    (\Phi_u^{Y_\tau})^*\lambda  = e^u\,\lambda + F(u,p,\tau) \, d\tau
  \end{equation*}
  as can be easily seen by taking the derivative of both sides
  \begin{multline*}
    \frac{d}{du} (\Phi_u^{Y_\tau})^*\lambda = (\Phi_u^{Y_\tau})^*
    \lie{Y_\tau} \lambda =
    (\Phi_u^{Y_\tau})^* \Bigl(\lambda + \omega(Y_0,Y_1) \, d\tau\Bigr) \\
    = e^u\,\lambda + \Bigl(F(u,p,\tau) + (\Phi_u^{Y_\tau})^*
    \bigl(\omega(Y_0,Y_1)\bigr)\Bigr) \, d\tau \;.
  \end{multline*}

  The pull-back of $\lambda$ under $\Psi$ yields
  \begin{equation*}
    \Psi^* \lambda = e^s\,\alpha_\tau
    + G(s,p,\tau) \, d\tau
  \end{equation*}
  with $\alpha_\tau = \restricted{\lambda}{T\p W\times\{\tau\}}$, and
  some smooth function
  $G\colon (-\epsilon,0]\times \p W \times [0,1] \to \RR$.

  We can define on the total space~$\widehat{W}_{[0,1]}$ a
  $1$-form~$\hat \lambda$ that agrees on the collar neighborhood of
  $\p W\times [0,1]$ with $\lambda$ by setting
  $e^s\, \alpha_\tau + \hat G(s,p,\tau)\, d\tau$ on
  $(-\epsilon,\infty)\times \p W \times [0,1]$, where $\hat G$ is a
  smooth $G$-invariant extension of $G$ to all of
  $(-\epsilon,\infty)\times \p W \times [0,1]$ that vanishes for large
  $s$-values.

  We easily verify that the restriction of
  $\hat \omega = d \hat \lambda$ to every fiber~$\widehat{W}_\tau$ is
  the cylindrical completion of $(W,\omega)$ with respect to the
  collar neighborhood defined by $Y_\tau$.  In particular it follows
  that every fiber is symplectic.

  \smallskip
    
  This allows us to define a symplectic connection that we use to lift
  the vector field~$\partial_\tau$ from $[0,1]$ to
  $\widehat{W}_{[0,1]}$.  It remains to show that the associated
  parallel transport is defined for all $\tau \in [0,1]$ to prove that
  $\widehat{W}_0$ is $G$-equivariantly symplectomorphic to
  $\widehat{W}_1$.

  For large values of $s$, $\hat\lambda$ simplifies to
  $e^s\, \alpha_\tau$ on the cylindrical end so that
  $\hat \omega = e^s\, \bigl(ds\wedge \alpha_\tau + d\alpha_\tau -
  (\frac{\partial}{\partial \tau} \alpha_\tau) \wedge d\tau \bigr) =
  e^s\, \bigl(ds\wedge \alpha_\tau + d\alpha_\tau + \bigl(\alpha_0 -
  \alpha_1\bigr)\wedge d\tau \bigr)$.  The lift~$\hat \partial_\tau$
  is the unique vector field that lies in the kernel of $\hat \omega$
  and that projects onto $\partial_\tau$.  Writing
  $\hat \partial_\tau = A\, \partial_s + Z + \partial_\tau$, where
  $A(s,p,\tau)$ is a smooth function, and $Z$ is a vector field that
  is tangent to the $\p W$-slices.  We easily compute from
  $A\, \alpha_\tau - \alpha_\tau(Z)\, ds + d\alpha_\tau(Z,\cdot) +
  \bigl(\alpha_0(Z) - \alpha_1(Z)\bigr)\, d\tau - \alpha_0 + \alpha_1
  = 0$ that $A = \alpha_0(R_\tau) - \alpha_1(R_\tau)$ by plugging the
  Reeb vector field~$R_\tau$ of $\alpha_\tau$ into the first equation.

  In particular we see that $A$ does not depend on $s$ so that the
  parallel transport cannot escape to $s=+\infty$ in
  time~$\tau \le 1$.  This shows that the parallel transport is
  defined and $(\widehat{W}_0, \widehat \omega_0)$ and
  $(\widehat{W}_1,\widehat \omega_1)$ are thus symplectomorphic.
\end{proof}

\bibliographystyle{amsalpha}

\begin{thebibliography}{MNW13}

\bibitem[AH91]{AharaHattoriHamiltonianCircleActions}
K.~Ahara and A.~Hattori, \emph{{$4$}-dimensional symplectic {$S^1$}-manifolds
  admitting moment map}, J. Fac. Sci. Univ. Tokyo Sect. IA Math. \textbf{38}
  (1991), no.~2, 251--298.

\bibitem[Ati82]{AtiyahConvexity}
M.~F. Atiyah, \emph{Convexity and commuting {H}amiltonians}, Bull. London Math.
  Soc. \textbf{14} (1982), no.~1, 1--15.

\bibitem[BH04]{BanyagaMorseBott}
A.~Banyaga and D.E. Hurtubise, \emph{A proof of the {M}orse-{B}ott lemma},
  Expo. Math. \textbf{22} (2004), no.~4, 365--373.

\bibitem[BtD95]{BrockerDieck}
T.~Bröcker and T.~tom Dieck, \emph{Representations of compact {L}ie groups},
  Graduate Texts in Mathematics, vol.~98, Springer-Verlag, New York, 1995.

\bibitem[Del88]{DelzantToric}
T.~Delzant, \emph{Hamiltoniens périodiques et images convexes de l'application
  moment}, Bull. Soc. Math. France \textbf{116} (1988), no.~3, 315--339.

\bibitem[Fra59]{Frankel}
T.~Frankel, \emph{Fixed points and torsion on {K}ähler manifolds}, Ann. of
  Math. (2) \textbf{70} (1959), 1--8.

\bibitem[Gei94]{Geiges_disconnected}
H.~Geiges, \emph{Symplectic manifolds with disconnected boundary of contact
  type}, Internat. Math. Res. Notices (1994), no.~1, 23--30.

\bibitem[Gir94]{Giroux_plusOuMoins}
E.~Giroux, \emph{Une structure de contact, même tendue, est plus ou moins
  tordue}, Ann. Sci. École Norm. Sup. (4) \textbf{27} (1994), no.~6, 697--705.

\bibitem[Gro85]{Gromov_HolCurves}
M.~Gromov, \emph{Pseudo holomorphic curves in symplectic manifolds}, Invent.
  Math. \textbf{82} (1985), 307--347.

\bibitem[GS82]{GuilleminSternbergConvexity}
V.~Guillemin and S.~Sternberg, \emph{Convexity properties of the moment
  mapping}, Invent. Math. \textbf{67} (1982), no.~3, 491--513.

\bibitem[Hir94]{HirschDiffTopology}
M.~Hirsch, \emph{Differential topology}, Graduate Texts in Mathematics,
  vol.~33, Springer-Verlag, New York, 1994, Corrected reprint of the 1976
  original.

\bibitem[HPS77]{HirshInvariantManifolds}
M.~W. Hirsch, C.~C. Pugh, and M.~Shub, \emph{Invariant manifolds}, Lecture
  Notes in Mathematics, Vol. 583, Springer-Verlag, Berlin-New York, 1977.

\bibitem[Kar99]{KarshonHamiltonianCircleActions}
Y.~Karshon, \emph{Periodic {H}amiltonian flows on four-dimensional manifolds},
  Mem. Amer. Math. Soc. \textbf{141} (1999), no.~672, viii+71.

\bibitem[Kat20]{KatzMorseWithConvexBoundary}
G.~Katz, \emph{Morse theory of gradient flows, concavity and complexity on
  manifolds with boundary}, World Scientific Publishing Co. Pte. Ltd.,
  Hackensack, NJ, 2020.

\bibitem[KMP00]{KurdykaGradientConjecture}
K.~Kurdyka, T.~Mostowski, and A.~Parusiński, \emph{Proof of the gradient
  conjecture of {R}.\ {T}hom}, Ann. of Math. (2) \textbf{152} (2000), no.~3,
  763--792.

\bibitem[KT91]{S1-contact_dim3}
Y.~Kamishima and T.~Tsuboi, \emph{{CR}-structures on {S}eifert manifolds},
  Invent. Math. \textbf{104} (1991), no.~1, 149--163.

\bibitem[Lau11]{LaudenbachMorseBoundary}
F.~Laudenbach, \emph{A {M}orse complex on manifolds with boundary}, Geom.
  Dedicata \textbf{153} (2011), 47--57.

\bibitem[Ler03]{Lerman1}
E.~Lerman, \emph{Contact toric manifolds}, J. Symplectic Geom. \textbf{1}
  (2003), no.~4, 785--828.

\bibitem[Lic29]{Lichtenstein_C1doubling}
L.~Lichtenstein, \emph{Eine elementare {B}emerkung zur reellen {A}nalysis},
  Math. Z. \textbf{30} (1929), no.~1, 794--795.

\bibitem[LMN19]{LisiPropertiesBourgeoisStructures}
S.~Lisi, A.~Marinkovi\'{c}, and K.~Niederkrüger, \emph{On properties of
  {B}ourgeois contact structures}, Algebr. Geom. Topol. \textbf{19} (2019),
  no.~7, 3409--3451.

\bibitem[Lut79]{Lutz_invariantes}
R.~Lutz, \emph{Sur la géométrie des structures de contact invariantes}, Ann.
  Inst. Fourier (Grenoble) \textbf{29} (1979), no.~1, 283--306.

\bibitem[McD88]{McDuffFixedPointsHamiltonianGroupsActions}
D.~McDuff, \emph{The moment map for circle actions on symplectic manifolds}, J.
  Geom. Phys. \textbf{5} (1988), no.~2, 149--160.

\bibitem[McD91]{McDuff_contactType}
\bysame, \emph{Symplectic manifolds with contact type boundaries}, Invent.
  Math. \textbf{103} (1991), no.~3, 651--671.

\bibitem[Mil65]{MilnorHCobordism}
J.~Milnor, \emph{Lectures on the {$h$}-cobordism theorem}, Notes by L.
  Siebenmann and J. Sondow, Princeton University Press, Princeton, N.J., 1965.

\bibitem[MNW13]{WeakFillabilityHigherDimension}
P.~Massot, K.~Niederkrüger, and C.~Wendl, \emph{Weak and strong fillability of
  higher dimensional contact manifolds}, Invent. Math. \textbf{192} (2013),
  no.~2, 287--373.

\bibitem[Mor29]{MorseSingularPointsBoundary}
M.~Morse, \emph{Singular {P}oints of {V}ector {F}ields {U}nder {G}eneral
  {B}oundary {C}onditions}, Amer. J. Math. \textbf{51} (1929), no.~2, 165--178.
  \MR{1506710}

\bibitem[MS98]{McDuffSalamonIntro}
D.~McDuff and D.~Salamon, \emph{{Introduction to symplectic topology. 2nd
  ed.}}, {Oxford Mathematical Monographs. New York, NY: Oxford University
  Press. }, 1998.

\bibitem[Ori18]{MorseBottBoundary}
R.~Orita, \emph{{M}orse-{B}ott inequalities for manifolds with boundary}, Tokyo
  J. Math. \textbf{41} (2018), no.~1, 113--130.

\bibitem[See64]{Seeley_SmoothDoubling}
R.T. Seeley, \emph{Extension of {$C^{\infty }$} functions defined in a half
  space}, Proc. Amer. Math. Soc. \textbf{15} (1964), 625--626.

\bibitem[Wen10]{WendlGirouxTorsion}
C.~Wendl, \emph{Strongly fillable contact manifolds and {$J$}-holomorphic
  foliations}, Duke Math. J. \textbf{151} (2010), no.~3, 337--384.

\bibitem[Wen13]{WendlCobordisms}
\bysame, \emph{Non-exact symplectic cobordisms between contact 3-manifolds}, J.
  Differential Geom. \textbf{95} (2013), no.~1, 121--182.

\bibitem[Ło63]{Lojasiewicz}
S.~Łojasiewicz, \emph{Une propriété topologique des sous-ensembles
  analytiques réels}, Les {É}quations aux {D}érivées {P}artielles ({P}aris,
  1962), Éditions du Centre National de la Recherche Scientifique, Paris,
  1963, pp.~87--89.

\end{thebibliography}

\providecommand{\bysame}{\leavevmode\hbox to3em{\hrulefill}\thinspace}
\providecommand{\MR}{\relax\ifhmode\unskip\space\fi MR }
\providecommand{\MRhref}[2]{%
  \href{http://www.ams.org/mathscinet-getitem?mr=#1}{#2}
}
\providecommand{\href}[2]{#2}


\end{document}